\documentclass[12pt,a4paper]{amsart}

\usepackage{graphicx}
\usepackage[inline]{enumitem}
\usepackage{amsmath, amssymb, amsthm, amscd}
\usepackage{xcolor}
\usepackage{caption}
\usepackage{subcaption}
\usepackage{multirow}
\usepackage{stmaryrd}
\usepackage{accents}
\usepackage{tikz}
\usepackage{bigints}
\usetikzlibrary{arrows.meta}
\usepackage[colorlinks=true, linkcolor=black, urlcolor=black, citecolor=black]{hyperref}
\usepackage{orcidlink}
\theoremstyle{plain}
\newtheorem{theorem}{Theorem}[section]
\newtheorem{lemma}[theorem]{Lemma}
\newtheorem{corollary}[theorem]{Corollary}
\newtheorem{proposition}[theorem]{Proposition}

\theoremstyle{definition}
\newtheorem{definition}[theorem]{Definition}

\theoremstyle{remark}
\newtheorem{remark}[theorem]{Remark}
\numberwithin{equation}{section}

\def\bold#1{\mbox{\boldmath $#1$}}
\newcommand{\uu}[1]{\bold{#1}}

\newcommand{\abs}[1]{\lvert#1\rvert}
\newcommand{\D}{\partial}
\newcommand{\dd}{\mathrm{d}}
\newcommand{\dive}{\mathrm{div}}
\newcommand{\bdive}{\uu{\mathrm{div}}}

\newcommand{\Dt}{\partial_t}

\newcommand{\grd}{\nabla}
\newcommand{\bgrd}{\uu{\nabla}}
\newcommand{\mbb}{\mathbb}

\newcommand{\mcal}{\mathcal}

\newcommand{\half}{\frac{1}{2}}
\newcommand{\veps}{\varepsilon}
\newcommand{\s}{\sigma}

\newcommand{\norm}[1]{\left \lVert #1 \right \rVert}
\newcommand{\M}{\mcal{M}}
\newcommand{\E}{\mcal{E}}
\newcommand{\Q}{\mcal{Q}}
\newcommand{\Ds}{D_\sigma}
\newcommand{\Ek}{\mcal{E}(K)}

\newcommand{\Lm}{L_{\M}(\Omega)}
\newcommand{\Hez}{\uu{H}_{\E,0}(\Omega)}

\newcommand{\dt}{\delta t}
\newcommand{\chark}{\mcal{X}_K}
\newcommand{\fesig}{F_{\epsilon , \sigma}}

\newcommand{\divm}{\dive_{\M}}
\newcommand{\intr}{\mathrm{int}}
\newcommand{\extr}{\mathrm{ext}}
\newcommand{\diam}{\text{diam}}
\newcommand{\absq}[1]{\abs{#1}^2}
\def\bold#1{\mbox{\boldmath $#1$}}
\newcommand\ubar[1]{% 
\underaccent{\bar}{#1}}

\makeindex

\title[An AP Scheme for the EPB System]{An Asymptotic Preserving Scheme for the Euler-Poisson-Boltzmann System in the Quasineutral Limit}

\author[Arun]{K.\ R.\ Arun\scalebox{1.2}{\orcidlink{0000-0001-9676-861X}}}
\address{School of Mathematics, Indian Institute of Science Education and Research Thiruvananthapuram, Thiruvananthapuram 695551, India} 
\email{arun@iisertvm.ac.in, rahuldev19@iisertvm.ac.in}
\thanks{K.R.A.\ acknowledges the support from Science and Engineering Research Board, Department of Science \& Technology, Government of India through grant CRG/2021/004078.}

\author[Ghorai]{R.\ Ghorai\,\scalebox{1.2}{\orcidlink{0000-0003-0549-3417}}}
\thanks{R.G.\ would like to thank the Ministry of Education, Government of India for the PMRF fellowship support.}

\date{\today}



\begin{document}

\begin{abstract}
    In this paper, we study an asymptotic preserving (AP), energy stable and positivity preserving semi-implicit finite volume scheme for the Euler-Poisson-Boltzmann (EPB) system in the quasineutral limit. The key to energy stability is the addition of appropriate stabilisation terms into the convective fluxes of mass and momenta, and the source term. The space-time fully-discrete scheme admits the positivity of the mass density, and is consistent with the weak formulation of the EPB system upon mesh refinement. In the quasineutral limit, the numerical scheme yields a consistent, semi-implicit discretisation of the isothermal compressible Euler system, thus leading to the AP property. Several benchmark numerical case studies are performed to confirm the robustness and efficacy of the proposed scheme in the dispersive as well as the quasineutral regimes. The numerical results also corroborates scheme's ability to very well resolve plasma sheaths and the related dynamics, which indicates its potential to applications involving low-temperature plasma problems. 
\end{abstract}

\maketitle

\section{Introduction}
\label{sec:intro}

A plasma consists of positively charged ions, negatively charged electrons and possibly neutral particles whose evolution is coupled through collision processes and long-range electromagnetic forces. The physics of plasma is a wide-ranging domain with applications many different fields, such as astrophysics, magnetic fusion, laser physics, spatial plasmas, to name but a few. The mathematical and numerical modelling of plasma primarily use two classes of models: kinetic models and fluid models. Within a kinetic framework, the plasma is characterised by a distribution function describing the behaviour of constituent particles in a phase space. In contrast, the fluid models rely on macroscopic parameters, such as the plasma particle mass density, the average velocity and the pressure. Although the kinetic models are typically more accurate than the fluid models, their practical application in numerical simulations is constrained by the computational expenses and the complexities inherent to high-dimensional phase spaces, and the need to use more involved approximation strategies. For an extensive study of plasma physics, various mathematical models used, and the associated intricacies, we refer the interested reader to \cite{Che13, GKP19, KT86}.

Among the different fluid models of plasma, the pressureless Euler-Poisson-Boltzmann (EPB) system is commonly used to describe the dynamics of cold plasma ions when the ion temperature is notably smaller than the electron temperature \cite{Che13, DLS+12, TK73}. The EPB system can be deduced from a two-species hydrodynamic model of plasma after making the assumptions that the ions have negligible temperature and the electrons have almost zero mass; see, e.g.\ \cite{DLS+12} for details. The zero electron mass assumption in turn yields a Boltzmann relation between the electron density and the electrostatic potential. Consequently, the system of governing equations consists of the conservation laws of mass and momenta for the ions and a nonlinear Poisson equation for the electrostatic potential. In plasma physics problems, the EPB model is widely used to capture the flow dynamics concealed within plasma-sheath transition layers and inner sheath layers in the case of planar plasma motions \cite{LS04}. From a mathematical perspective, the EPB system is quite intriguing since it admits multi-valued solutions as demonstrated in the studies \cite{LW07b,LW07a}. Yet another complex behaviour exhibited its solutions is the existence of multiple scales in time and space. The multiscale nature of the solutions is well understood once the equations are cast in non-dimensional variables. The non-dimensionalisation of the EPB system leads to one of the most important characteristic parameters in plasma physics, known as the Debye length, which represents the characteristic distance over which the electric field and the charges are screened in a plasma. Depending on the problem, the Debye length can be almost equal to a characteristic length of physical interest or it can be much smaller. In other words, the Debye length when scaled by a typical length scale of the problem, can quantify the multiscale behaviour of the solutions. Thus, it is instructive to analyse the EPB system parameterised by a scaled Debye length, say $\veps$. When $\veps=O(1)$, the flow regime is dispersive in nature due to the non-linear coupling of the electric potential term and the convective terms. In this regime, the EPB system is hyperbolic-elliptic in nature and its solutions can develop singularities; see \cite{BCK24} for details. Furthermore, in the dispersive regime, the system of equations is weakly unstable as the hydrodynamic part is not strictly hyperbolic due to the absence of the pressure term. In contrast, when $\veps\to 0$, referred to as the quasineutral regime in the literature, the EPB system approximates the isothermal compressible Euler (ICE) system \cite{CG00,FZ10,Xue14} which is strongly hyperbolic in nature. Hence, the quasineutral limit is a singular limit for the EPB system and accordingly, its numerical simulation in the quasineutral regime turns out to be a challenging task.

Explicit numerical schemes encounter instabilities and loss of accuracy in the quasineutral regime unless their time and space steps resolve the small parameter $\veps$. Furthermore, the stability restrictions bring in stiffness when $\veps\to0$. In order to overcome such impediments, we take recourse to the so-called `Asymptotic Preserving' (AP) methodology which was originally introduced in the context of diffusive limits of kinetic transport equations \cite{Jin99, Jin12}. The AP framework provides a robust numerical approximation platform for treating singularly perturbed problems governed by partial differential equations and it has been successfully implemented for the quasineutral limit in the last decade; see \cite{AGK24,Deg13,DLS+12} and the references therein. The working principle behind the AP methodology can be described as follows. Let $\mcal{P}_{\veps}$ denote a singularly perturbed problem with perturbation parameter $\veps$ and let $\mcal{P}_{0}$ be a well-posed limit problem obtained from $\mcal{P}_{\veps}$ in the limit $\veps \to 0$. A consistent discretisation $\mcal{P}_{\veps}^{h}$ for $\mcal{P}_{\veps}$, with $h$ being a space-time discretisation parameter, is said to be AP if     
\begin{enumerate}[label=(\roman*)]
\item $\mcal{P}_{\veps}^{h}$ converges to a consistent discretisation $\mcal{P}_{0}^{h}$ of the limit problem $\mcal{P}_{0}$ as $\veps \to 0$, and 
\item the stability constraints on the space-time discretisation parameter $h$ remain independent of $\veps$. 
\end{enumerate}
Therefore, the AP discretisation captures the asymptotic behaviour of a singularly perturbed problem at the discrete level. It can automatically emulate the regime shifts with respect to $\veps$ without any need to enforce transition layers or matching conditions. Consequently, the AP methodology proves to be an obvious choice for the numerical approximation of the EPB system in all regimes ranging from quasineutral to dispersive.

Implicit-explicit (IMEX) or semi-implicit time discretisations are shown to overcome several of the shortcomings of explicit time-stepping schemes and maintain the AP property when applied to the numerical approximation of stiff hydrodynamic models; see \cite{BAL+14,BJR+19,CDV05,CDV07,DLV08,DT11,HLS21} for some developments in this direction. IMEX time discretisations have the advantage of avoiding the costly inversion of dense and large matrices, typical of fully implicit schemes. At the same time, they can overcome the stringent time-step restrictions required by explicit schemes, thereby achieving computational efficiency. In the literature, one can find several attempts to derive semi-implicit AP schemes for the Euler-Poisson (EP) and the EPB systems with varying implicit treatments of the fluxes and the source term; see \cite{AGK24,CDV05,CDV07,DLS+12} and the references cited therein. The space discretisation in all these papers follow a finite volume framework. Since the quasineutral limit of the EP system yields the incompressible Euler equations, a natural choice for the finite volume grid is a standard one used by incompressible solvers. Accordingly, the AP finite volume scheme proposed in \cite{AGK24} for the EP system uses the so-called `Marker and Cell' (MAC) grid. The MAC grid has the advantage that the elliptic problems arising from reformulations of the implicit terms is stable \cite{NJJ+09}, which is pivotal in establishing the asymptotic limit of the discrete equations. Another crucial requirement while dealing with the quasineutral limit or any other singular limit of hydrodynamic models in general is to maintain the energy stability. The energy stability of a scheme not only ensures physically realistic numerical solutions but also provides apriori estimates that are essential to carry out thorough asymptotic convergence analyses. 

In the present work, our aim is to devise a semi-implicit in time, finite volume in space, AP and energy stable scheme for the EPB system which is valid across the different regimes ranging from dispersive to quasineutral. The key to energy stability is a stabilisation technique which was introduced in \cite{CDV17, DVB17, DVB20} for the shallow water equations and later pursued and analysed in \cite{AGK24,AGK23,AK24b,AK24a} for different hydrodynamic models. In the present case, the stabilisation technique involves introducing a shift in the mass and momentum convective fluxes as well as in the source term. These shift terms play a pivotal role in stabilising the flow via the dissipation of mechanical energy, at the same time respecting CFL-type numerical stability conditions. A formal analysis of the continuous equations with the added shift terms reveals that the momentum shift must be proportional to the gradient of the electrostatic potential and the source shift, to the divergence of the velocity. The time discretisation follows an IMEX approach wherein the fluxes are treated explicitly while the sources implicitly. In space, a staggered finite volume technique on a MAC grid is adopted. The application of certain tools from topological degree theory and a discrete comparison principle for the nonlinear Poisson equation, the energy stability in turn yields the existence of a discrete solution of the space-time fully-discrete, semi-implicit numerical scheme. The overall scheme is shown to be consistent with the weak formulation of the EPB system, {\`a} la Lax-Wendroff for a fixed $\veps$, as the space-time discretisation parameters goes to zero. Furthermore, we establish the consistency of the scheme with the quasineutral ICE limit system when $\veps\to 0$, after making some reasonable boundedness assumptions. At the end, several benchmark numerical case studies carried out to validate the theoretical findings. The rest of this paper is organised as follows. The scaled EPB system, its quasineutral limit and the stabilisation technique are introduced in Section~\ref{sec:cont-case}. Section~\ref{sec:MAC_disc_diff} deals with the space discretisation, the discrete differential operators and the construction of numerical fluxes. The presentation of the numerical scheme and its stability attributes are taken up in Section~\ref{sec:stable_scheme}. Sections \ref{sec:weak_cons} and \ref{sec:quas_lim} are devoted to proving the two major consistency results. In Section~\ref{sec:num_res}, we present the results of numerical experiments, and finally we close the paper with some concluding remarks in Section~\ref{sec:conclusion}.   

\section{Euler-Poisson-Boltzmann System and the Quasineutral Limit}
\label{sec:cont-case}

As mentioned in Section\,\ref{sec:intro}, the EPB system describes the dynamics of cold ions in a collisionless regime of plasma. The system of governing equations consists of the conservation laws of mass and momentum, where the latter contains an electrostatic forcing term. The electric potential is coupled with the ion and electron densities through a nonlinear Poisson equation. The EPB system comprises the following coupled equations:
\begin{subequations}
\label{eq:EPB_model}
\begin{gather}
\partial_{t}\rho+\dive\left(\rho\uu{u}\right)=0, \\
m\left[\partial_{t}\left(\rho \uu{u}\right)+\bdive (\rho \uu{u} \otimes \uu{u})\right]=-e \rho \grd\phi, \\
\quad-\epsilon_{0} \Delta \phi=e\left(\rho-\rho_{e}\right),
\end{gather}
\end{subequations}
for $(t, \uu{x})\in Q_T:=(0, T ) \times \Omega$, where $T>0$ and $\Omega$ is an open, bounded and connected subset of $\mathbb{R}^d$, $d=1,2,3$. The dependent variables $\rho=\rho(t,\uu{x})>0$, $\uu{u}=\uu{u}(t,\uu{x})\in\mbb{R}^d$ and $\phi=\phi(t,\uu{x})\in\mbb{R}$ are, respectively, the ion density, the average ion velocity and the electrostatic potential. The constants $m$, $\epsilon_0$ and $e$ denote, respectively, the ion mass, the vacuum permittivity and the positive elementary charge. The electron density $\rho_{e}$ is assumed to satisfy the following Boltzmann relation:
\begin{equation}
\label{eq:boltz_relation}
    \rho_{e} = \rho_0 \exp\Big(\frac{e\phi}{k_{B}T}\Big), 
\end{equation}
where $\rho_0$ is the ion density chosen at a reference point inside the plasma, $k_{B}$ is the Boltzmann constant and $T$ is the electron temperature which is assumed to be uniform and constant in time. In the literature, the system \eqref{eq:epb} is referred to as the pressureless one-fluid model of ions, with the electron density satisfying the Boltzmann relation. 

In plasma, the characteristic distance over which electrostatic effects persist between charged particles is usually quantified by a significant physical scale within the system \eqref{eq:EPB_model}, known as the Debye length, defined by 
\begin{equation}
\label{eq:veps_DL}
    \veps_{D}=\left(\frac{\epsilon_0 k_B T}{e^2 \rho_0}\right)^{1/2}.
\end{equation}
The Debye length is the typical spatial scale associated with the electrostatic interaction between particles. In many practical situations of interest, the Debye length can be much smaller than a characteristic length scale associated with the problem, e.g.\ the distance between two electrodes. Subsequently, the co-existence of two different length scales within same problem contributes to a multiscale nature of the flow characteristics. In order to further elucidate the multiscale behaviour of the EPB system \eqref{eq:EPB_model}, we introduce a scaling of the system \eqref{eq:EPB_model} along the lines of \cite{Deg13}. We choose a reference length scale $L$ and non-dimensionalise the independent variables via $\bar{x}=\frac{x}{L}$ and $\bar{t}=\frac{t u_{0}}{L}$, where $u_{0}\equiv u_B=\sqrt{\frac{k_{B} T}{m}}$ is the Boltzmann velocity scale, also known as the sound velocity in plasma. We take $\rho_0$ to be a reference density and define the scaled dependent variables as $\bar{\rho}=\frac{\rho}{\rho_{0}}$, $\bar{u}=\frac{u}{u_{0}}$, $\bar{\phi}=\frac{e \phi}{m u_{0}^{2}}$. Inserting these non-dimensional variables into the equations \eqref{eq:EPB_model} and omitting the bars for convenience yields the following non-dimensional EPB system:  
\begin{subequations}
\label{eq:epb}
\begin{align}
  \D_t\rho^\veps+\dive (\rho^\veps\uu{u}^\veps)&=0, \label{eq:cons_mas}\\ 
  \D_t(\rho^\veps\uu{u}^\veps)+\bdive
  (\rho^\veps\uu{u}^\veps\otimes\uu{u}^\veps) &=-\rho^\veps\bgrd\phi^\veps, \label{eq:cons_mom}\\
  -\veps^{2}\Delta\phi^\veps&=\rho^\veps-e^{\phi^\veps}. \label{eq:poisson-boltz}
\end{align}
\end{subequations}
Here, $\veps = \frac{\veps_D}{L}$ is the scaled Debye length and throughout this paper we assume that it is an infinitesimal taking values in $(0,1]$. The system \eqref{eq:epb} is supplemented with the following initial and boundary conditions:
\begin{equation}
     \rho^\veps\vert_{t=0}=\rho^\veps_0, \quad \uu{u}^\veps\vert_{t=0}=\uu{u}^\veps_0, \quad \uu{u}^\veps\cdot\uu{\nu}\vert_{\partial\Omega}=0, \quad \bgrd\phi^\veps\cdot\uu{\nu}\vert_{\partial\Omega}=0,\label{eq:eq_ic-bc}
\end{equation}
where $\uu{\nu}$ denotes the unit outward normal to the boundary. 

It is well-established that smooth solutions to nonlinear hyperbolic equations fail to exist globally in time which necessitates the concept of weak solutions. In the following, we define a weak solution of the system \eqref{eq:epb}. 
\begin{definition}
\label{def:weak_soln_defn}
    We say that a triple $(\rho^\veps, \uu{u}^\veps,\phi^\veps)$ is a weak solution to the EPB system \eqref{eq:epb} with initial-boundary conditions \eqref{eq:eq_ic-bc} if
    \begin{enumerate}[label=(\roman *)]
        \item $\rho^\veps>0$ a.e.\ in $Q_T$, $\rho^\veps\in L^{\infty}(Q_T)$, $\uu{u}^\veps\in L^{\infty}(Q_T)^{d}$ and $\phi^\veps\in L^{\infty}(0,T;H^1(\Omega))$ with the boundary condition $\bgrd\phi^\veps\cdot\uu{\nu}\vert_{\partial\Omega}=0$ satisfied in the sense of distributions.
        \item The integral identity
        \begin{equation}
        \int_0^T\int_\Omega\left(\rho^\veps\D_t\psi+\rho^\veps \uu{u}^\veps\cdot\grd \psi\right)\dd\uu{x}\dd t =-\int_\Omega\rho^\veps_0\psi(0,\cdot)\dd\uu{x} \label{eq:weak_soln_mas}
        \end{equation}
        holds for all  $\psi \in C_c^\infty([0,T)\times\bar{\Omega})$.
        \item The integral identity
        \begin{gather}
        \int_0^T\int_\Omega\left(\rho^\veps \uu{u}^\veps\cdot\D_t\uu{\psi}+(\rho^\veps \uu{u}^\veps\otimes\uu{u}^\veps):\bgrd \uu{\psi}\right)\dd\uu{x}\dd t =-\int_0^T\int_\Omega\rho^\veps\grd{\phi^\veps}\cdot\uu{\psi}\dd\uu{x}\dd t\nonumber\\
        -\int_{\Omega}\rho^\veps_0 \uu{u}^\veps_0\cdot\uu{\psi}(0,\cdot)\dd\uu{x},
        \label{eq:weak_soln_mom}
        \end{gather}
        holds for all  $\uu{\psi} \in C_c^\infty([0,T)\times\bar{\Omega})^d$.
        \item The integral identity
        \begin{equation}
        \veps^2\int_0^T\int_\Omega\grd\phi^\veps\cdot\grd \psi\dd\uu{x}\dd t =\int_0^T\int_\Omega(\rho^\veps-e^{\phi^\veps})\psi\dd\uu{x}\dd t \label{eq:weak_soln_poisson}
        \end{equation}
        holds for all  $\psi \in C_c^\infty([0,T)\times\bar{\Omega})$.        
    \end{enumerate}
\end{definition}

\subsection{Apriori Energy Estimates}
\label{sec:apr_est}

In order to carry out the analysis of a numerical scheme for the EPB system \eqref{eq:epb}, a few apriori energy estimates are essential. We start with recalling the energy identities satisfied by classical solutions of \eqref{eq:epb}. 

\begin{proposition}
  \label{prop:engy_balance}
  The classical solutions of \eqref{eq:cons_mas}-\eqref{eq:poisson-boltz} satisfy the following identities:
  \begin{itemize}
  \item a potential energy identity
    \begin{equation}
      \label{eq:potbal}
     \Dt\Big(\frac{\veps^2}{2}\abs{\bgrd\phi^\veps}^{2}+(\phi^\veps-1)e^{\phi^\veps}\Big)
      -\veps^2\dive(\phi^\veps\Dt(\bgrd\phi^\veps)=-\dive(\rho^\veps\uu{u}^\veps)\phi^\veps,
    \end{equation}
  \item a kinetic energy identity
    \begin{equation}
      \label{eq:kinbal}
      \Dt\Big(\frac{1}{2}\rho^\veps{\abs{\uu{u}^\veps}}^2\Big)
      +\dive\Big(\frac{1}{2}\rho^\veps{\abs{\uu{u}^\veps}}^2\uu{u}^\veps\Big)
      = -\rho^\veps\uu{u}^\veps\cdot\bgrd\phi^\veps.
    \end{equation}
  Adding \eqref{eq:potbal} and \eqref{eq:kinbal} we get
  \item the total energy identity
  \begin{align}
  \label{eq:eng_id}
  \Dt\Big(\frac{\veps^2}{2}\abs{\bgrd\phi^\veps}^{2}+(\phi^\veps-1)e^{\phi^\veps}+\frac{1}{2}\rho^\veps{\abs{\uu{u}^\veps}}^2\Big)
      +\dive\Big((\rho^\veps\phi^\veps+\frac{1}{2}\rho^\veps{\abs{\uu{u}^\veps}}^2)\uu{u}^\veps-\veps^2\phi^\veps\Dt(\bgrd\phi^\veps)\Big)=0.      
  \end{align}
  \end{itemize}
\begin{proof}
Proof of \eqref{eq:kinbal} is classical. To prove \eqref{eq:potbal}, we multiply \eqref{eq:cons_mas} by $\phi^\veps$ and write
\begin{align*}
0&=(\D_t\rho^\veps)\phi^\veps+\dive (\rho^\veps\uu{u}^\veps)\phi^\veps\\
&=(\D_t(\rho^\veps-e^{\phi^\veps}))\phi^\veps+\D_t(e^{\phi^\veps})\phi^\veps+\dive (\rho^\veps\uu{u}^\veps)\phi^\veps ,\\
&=-\veps^2(\D_t(\Delta\phi^\veps))\phi^\veps+\D_t(e^{\phi^\veps}\phi^\veps)-e^{\phi^\veps}\D_t(\phi^\veps)+\dive (\rho^\veps\uu{u}^\veps)\phi^\veps ,\text{  using \eqref{eq:poisson-boltz}}\\
&=\Dt\Big(\frac{\veps^2}{2}\abs{\bgrd\phi^\veps}^{2}\Big)+\D_t((\phi^\veps-1)e^{\phi^\veps})
      -\veps^2\dive(\phi^\veps\Dt(\bgrd\phi^\veps)+\dive(\rho^\veps\uu{u}^\veps)\phi^\veps.
\end{align*}
Rearranging the terms gives \eqref{eq:potbal}.
\end{proof}
\end{proposition}
\begin{remark}
    The total energy identity \eqref{eq:eng_id} is proven for classical solutions of the system \eqref{eq:epb} and the same remains as an inequality for weak solutions. Consequently, throughout this paper, we assume the existence of an energy stable weak solution, i.e.\ a weak solution satisfying the estimate
    \begin{equation}
    \label{eq:glob_eng}
    \begin{aligned}
    & \int_{\Omega}\Big[\frac{\veps^2}{2} \abs{\bgrd\phi^\veps}^{2}+(\phi^\veps-1)e^{\phi^\veps} +\frac{1}{2}\rho^\veps |\uu{u}^\veps|^2 \Big](t, \cdot) \dd \uu{x}  \\
    & \leqslant \int_{\Omega}\Big[\frac{\veps^2}{2} \abs{\bgrd\phi^\veps_0}^{2} + (\phi^\veps_0-1)e^{\phi^\veps_0} +\frac{1}{2} \rho^\veps_0 |\uu{u}^\veps_0|^2\Big] \dd \uu{x}, \quad t \in (0, T) \text{ a.e.},
    \end{aligned}
    \end{equation}    
    where $\phi_0^\veps$ is the solution of \eqref{eq:poisson-boltz} with $\rho_0^\veps$ on the right hand side. 
\end{remark}

\subsection{Quasineutral limit}
\label{sec:QN_lt}

The quasineutral limit can be performed by letting $\veps\rightarrow 0$ in the system \eqref{eq:epb}. In the following, we assume that $\rho^\veps\to\tilde{\rho}$ and $\uu{u}^\veps\to\uu{U}$ as $\veps\to0$. Formally letting $\veps\to 0$ in \eqref{eq:epb}, the following ICE equations are obtained:
\begin{subequations}
    \label{eq:iso_euler}
    \begin{align}
    \D_t\tilde{\rho}+\dive (\tilde{\rho}\uu{U})&=0,\\
    \D_t(\rho\uu{U})+\bdive (\tilde{\rho}\uu{U}\otimes\uu{U})+\bgrd\tilde{\rho} &=0.
\end{align}
\end{subequations}

The literature on the quasineutral limit of the pressuresless EPB system \eqref{eq:epb} is very sparse; see, e.g.\ \cite{Xue14} for a treatment of the quasineutral limit of the pressureless EPB system in the case of strong solutions. For related accounts on the quasineutral limit, we refer the reader to \cite{CDM+95,CG00,DM08,FZ10,GS03,Wan04}.

\subsection{Stabilisation}
\label{sec:stab}

The goal of the present work is to propose and analyse a semi-implicit finite volume scheme for the EPB system \eqref{eq:epb} which is energy stable and AP in the quasineutral limit. In other words, our objective is to establish a discrete equivalent of the energy identity \eqref{eq:eng_id} for the numerical scheme. In order to achieve the energy stability of the numerical solutions, we employ the formalism introduced \cite{DVB17, DVB20, PV16} and pursued in \cite{AGK24,AGK23,AK24b,AK24a}. This approach involves the incorporation of stabilisation terms in the mass and momentum fluxes in \eqref{eq:cons_mas}-\eqref{eq:cons_mom} and also in the electrostatic forcing term on the right-hand side of the momentum equation \eqref{eq:cons_mom}, thus yielding the modified system
\begin{subequations}
\label{eq:stab_epb}
\begin{align}
  \D_t\rho^\veps+\dive (\rho^\veps\uu{u}^\veps-\uu{q}^\veps)&=0, \label{eq:r_cons_mas}\\ 
  \D_t(\rho^\veps\uu{u}^\veps)+\bdive
  (\uu{u}^\veps\otimes(\rho^\veps\uu{u}^\veps-\uu{q}^\veps))&=-\rho^\veps\bgrd\phi^\veps+\bgrd\Lambda^\veps, \label{eq:r_cons_mom}\\
  -\veps^{2}\Delta\phi^\veps&=\rho^\veps-e^{\phi^\veps}.\label{eq:r_poisson-boltz}
\end{align}
\end{subequations}
Similar to Proposition \ref{prop:engy_balance}, one can establish the apriori estimates satisfied by classical solutions of the modified system \eqref{eq:r_cons_mas}-\eqref{eq:r_poisson-boltz}. Expressions for the stabilisation terms is determined accordingly from the total energy estimate so as to guarantee the stability of solutions.
\begin{proposition}
  \label{prop:r_engy_balance}
  The classical solutions of \eqref{eq:r_cons_mas}-\eqref{eq:r_poisson-boltz} satisfy the following identities. 
  \begin{itemize}
  \item A potential energy identity
    \begin{equation}
      \label{eq:r_potbal}
     \Dt\left(\frac{\veps^2}{2}\abs{\bgrd\phi^\veps}^{2}\right)+\D_t((\phi^\veps-1)e^{\phi^\veps})
      -\veps^2\dive(\phi^\veps\Dt(\bgrd\phi^\veps)=-\dive(\rho^\veps\uu{u}^\veps-\uu{q}^\veps)\phi^\veps.
    \end{equation}
  \item A kinetic energy identity
    \begin{align}
      \label{eq:r_kinbal}
      \Dt\left(\frac{1}{2}\rho^\veps{\abs{\uu{u}^\veps}}^2\right)+\dive\left(\frac{1}{2}{\abs{\uu{u}^\veps}}^2(\rho^\veps\uu{u}^\veps-\uu{q}^\veps)\right)=-\bgrd \phi^\veps\cdot(\rho^\veps\uu{u}^\veps-\uu{q}^\veps)
      -\uu{q}^\veps\cdot\bgrd\phi^\veps+\uu{u}^\veps\cdot\bgrd\Lambda^\veps.  
      \end{align}
  \item Adding \eqref{eq:r_potbal} and \eqref{eq:r_kinbal} yields total energy identity
  \begin{align}
  \label{eq:r_eng_id}
  &\Dt\left(\frac{\veps^2}{2}\abs{\bgrd\phi^\veps}^{2}+(\phi^\veps-1)e^{\phi^\veps}+\frac{1}{2}\rho^\veps{\abs{\uu{u}^\veps}}^2\right)\\ \nonumber
  &+\dive\Big((\phi^\veps+\frac{1}{2}{\abs{\uu{u}^\veps}}^2)(\rho^\veps\uu{u}^\veps-\uu{q}^\veps)-\veps^2\phi^\veps\Dt(\bgrd\phi^\veps)-\uu{u}^\veps\Lambda^\veps\Big)
  =-\uu{q}^\veps\cdot\bgrd\phi^\veps-\dive(\uu{u}^\veps)\Lambda^\veps  
  \end{align}
  \end{itemize}
  \begin{proof}
  The proof follows similar lines as in Proposition~\ref{prop:engy_balance} and hence omitted. 
  \end{proof}
  \end{proposition}
  \begin{remark}
  The total energy identity \eqref{prop:engy_balance} motivates us to choose $\uu{q}^\veps$ and $\Lambda^\veps$, proportional to the terms $\bgrd\phi^\veps$ and $\dive(\uu{u}^\veps)$, respectively. Thus, at the continuous level, we immediately see that formally choosing
  \begin{align}
      \uu{q}^\veps&=\eta\bgrd\phi^\veps, \\
      \Lambda^\veps&=\alpha\dive(\uu{u}^\veps)
  \end{align}
  with $\eta > 0$ and  $\alpha > 0$, leads to the energy stability inequality:
  \begin{align}
    \label{eq:r_eng_id_ineq}
  &\Dt\Big(\frac{\veps^2}{2}\abs{\bgrd\phi^\veps}^{2}+(\phi^\veps-1)e^{\phi^\veps}
  +\frac{1}{2}\rho^\veps{\abs{\uu{u}^\veps}}^2\Big)\nonumber\\
  &+\dive\Big((\phi^\veps+\frac{1}{2}{\abs{\uu{u}^\veps}}^2)(\rho^\veps\uu{u}^\veps-\uu{q}^\veps)-\veps^2\phi^\veps\Dt(\bgrd\phi^\veps)-\uu{u}^\veps\Lambda^\veps\Big)
  \leqslant 0 
  \end{align}
\end{remark}

Based on the above considerations, we design a semi-implicit scheme in which the stabilisation of the numerical momentum flux functions and the source term are pivotal for achieving nonlinear energy stability; see also \cite{ADD+21,DVB17,DVB20,GVV13,PV16} for related discussions. 

\section{Space Discretisation and Discrete Differential Operators}
\label{sec:MAC_disc_diff}

This section is devoted to discretising the computational space-domain $\Omega\subseteq\mathbb{R}$ using a MAC grid and defining the corresponding function spaces to approximate the stabilised EPB system \eqref{eq:r_cons_mas}-\eqref{eq:r_poisson-boltz} within a finite volume framework. We begin by assuming that the closure of the domain $\Omega$ is a union of closed rectangles $(d=2)$ or closed orthogonal parallelepipeds $(d=3)$, with mutually disjoint interiors. Without loss of generality, we further assume that the edges (or faces) of these rectangles or parallelepipeds are orthogonal to the canonical basis vectors, denoted by $(\uu{e}^{(1)},\dots,\uu{e}^{(d)})$. We refer the interested reader to, e.g.\ \cite{GHL+18, GHM+16} for a detailed account on numerical schemes employing MAC grids and their analysis.

\subsection{Mesh and Unknowns}
\label{subsec:mesh_unkn}
A MAC grid consists of a pair $\mcal{T}=(\mcal{M},\mcal{E})$, where $\mcal{M}$ represents the primal mesh which is a partition of $\bar{\Omega}$ consisting of possibly non-uniform closed rectangles ($d=2$) or parallelepipeds ($d=3$) and $\mcal{E}$ denotes the set of all edges of the primal cells. We express $\E$ as $\mcal{E}=\mcal{E}_\intr\cup\mcal{E}_\extr$, where $\mcal{E}_\intr$ and $\mcal{E}_\extr$ are, respectively, the collection of internal and external edges of $\E$. We denote by $\E^{(i)}$ the set of edges that are orthogonal to $\uu{e}^{(i)}$ and decompose $\E^{(i)}$ as $\mcal{E}^{(i)}=\mcal{E}_\intr^{(i)}\cup\mcal{E}_\extr^{(i)}$, where $\E_\intr^{(i)}$ (resp. $\E_\intr^{(i)}$) are, respectively, the internal and external edges of $\mcal{E}^{(i)}$. We denote $\s=K|L$, where the edge $\s\in\E_\intr$ is such that $\s=\bar{K}\cap\bar{L}$ with $K,L\in\M$. For each $\sigma\in\mcal{E}$, we construct a dual cell $\Ds$ which is defined as follows:
\begin{equation*}
\Ds = 
\begin{cases}
    D_{\s, K}\cup D_{\s, L} , & \text{if } \s=K|L\in\E_\intr, \\
    D_{\s, K}, & \text{if } \s\in\E_\extr \text{ and adjacent to } K\in\M,
\end{cases}
\end{equation*}
where $D_{\s, K}$ is the half portion of the primal cell $K$ adjacent to $\s$. Furthermore, we denote by $\mcal{E}(K)$, the set of all edges of $K\in\mcal{M}$ and by $\tilde{\mcal{E}}(D_\sigma)$, the set of all edges of the dual cell $D_\sigma$. The mesh size is defined by 
\begin{equation}
    \label{eq:mesh_size}
    h_\mcal{T}=\max_{K\in\M}\{h_K\},
\end{equation}
where $h_K$ denotes the diameter of $K$.

We denote by $L_{\mcal{M}}(\Omega)$, the space of scalar-valued functions which are piecewise constant on each primal cell $K\in\M$. The space $L_\M$ is used for approximating the density and the potential. Analogously, we define by $\uu{H}_{\mcal{E}}(\Omega)=\prod_{i=1}^{d} H^{(i)}_{\mcal{E}}(\Omega)$, the set of vector-valued (in $\mbb{R}^d$) functions which are constant on each dual cell $D_\sigma$ and for each $i=1,2,\dots,d$. The space of vector-valued functions vanishing on the external edges is denoted as  $\uu{H}_{\mcal{E},0}(\Omega)=\prod_{i=1}^d H^{(i)}_{\mcal{E},0}(\Omega)$, where  $H^{(i)}_{\mcal{E},0}(\Omega)$ contains those elements of $H^{(i)}_{\mcal{E}}(\Omega)$ which vanish on the external edges. For a primal grid function $q\in L_{\mcal{M}}(\Omega)$, represented as $q =\sum_{K\in\mcal{M}}q_K\chark$, and for each $\sigma = K|L \in\cup_{i=1}^d\mcal{E}^{(i)}_\intr$, the dual average $q_{D_\sigma}$ of $q$ over $D_\sigma$ is defined through the following relation:
  \begin{equation}
   \label{eq:mass_dual}
     \abs{D_\sigma}q_{D_\sigma}=\abs{D_{\sigma,K}}q_K+\abs{D_{\sigma,L}}q_L.
  \end{equation}

\subsection{Discrete Convection Fluxes and Differential Operators}
\label{sec:dic_convect}

We now introduce the discrete convection fluxes and the discrete differential operators on the functional spaces defined above, aimed at approximating the differential operators in the EPB system \eqref{eq:epb}. We initiate by constructing the mass fluxes, which relies on the following technical result presented below for convenience and proved in \cite[Lemma 2]{GH+21}.
\begin{lemma}
    \label{lem:rho_sig}
    Let $\psi$ be a strictly convex and continuously differentiable function over an open interval $I$ of $\mathbb{R}$. For each $\rho_K, \rho_L\in I$, there exists a unique $\rho_{KL}\in\llbracket \rho_K, \rho_L\rrbracket$ such that
    \begin{equation}
    \psi(\rho_K) + \psi^\prime(\rho_K)(\rho_{KL}-\rho_K) = \psi(\rho_L) + \psi^\prime(\rho_L)(\rho_{KL}-\rho_L),\; \mbox{if}\; \rho_K\neq\rho_L,
    \end{equation}
    and $\rho_{KL}=\rho_K=\rho_L$, otherwise. Specifically, with the function $\psi(\rho)=\rho\ln\rho$, $I=(0,\infty)$ and for each $K,L\in\M$, there exists a unique $\rho_{KL}\in\llbracket \rho_K, \rho_L\rrbracket$ such that 
    \begin{equation}
    \label{eq:rho_sig_choice}
    \begin{aligned}
    \rho_K - \rho_L &= \rho_{KL}[\ln(\rho_K)-\ln(\rho_L)], \
    \mbox{if} \ \rho_K\neq\rho_L,\\
    \rho_{KL}=\rho_K &=\rho_L, \ \mbox{otherwise}.
    \end{aligned}
    \end{equation}
    \end{lemma}
\begin{remark}
    Throughout this paper, the symbol $\llbracket a, b\rrbracket$ denotes the interval $[\min(a, b), \max(a, b)]$, for any two real numbers $a$ and $b$.
\end{remark}

\begin{definition}
\label{def:disc_conv_flux}
Suppose we discretise $\Omega$ with a MAC grid and introduce the discrete function spaces as outlined above. The definitions of the mass and momentum fluxes are then given as follows.
\begin{itemize}
  \item For each $K \in \mcal{M}$ and $\sigma\in \mcal{E}(K)$, the mass flux $F_{\sigma,K} \colon L_{\mcal{M}}(\Omega) \times \uu{H}_{\mcal{E},0}(\Omega) \to \mbb{R}$ is defined by
    \begin{equation}
      \label{eq:mass_flux}
      F_{\sigma,K}(\rho,\uu{u}):=\abs{\sigma}(\rho_{\s}u_{\sigma, K}-\Q_{\s, K}), \ (\rho,\uu{u})\in
      L_{\mcal{M}}(\Omega) \times \uu{H}_{\mcal{E},0}(\Omega).
    \end{equation}
Here, $u_{\sigma , K}=u_{\sigma} \uu{e}^{(i)}\cdot\uu{\nu}_{\sigma, K}$ is an approximation of the normal velocity on the edge $\s$ and $\uu{\nu}_{\sigma, K}$ is the unit vector normal to the edge     $\sigma\in\mcal{E}^{(i)}_{\mathrm{int}}\cap\Ek$ in the direction outward to the cell $K$. The density on each interface $\sigma=K|L\in\mcal{E}$ is approximated using Lemma~\ref{lem:rho_sig} as $\rho_\s=\rho_{KL}$, where $\rho_{KL}$ is obtained by \eqref{eq:rho_sig_choice}. The choice for the stabilisation term $\Q\in\uu{H}_{\mcal{E},0}(\Omega)$ in \eqref{eq:mass_flux} will be made aposteriori after an energy stability analysis of the entire finite volume scheme is carried out.
  \item For a fixed $i=1,2,\dots,d$, for each $\sigma\in\mcal{E}^{(i)}, \epsilon\in\tilde{\E}(D_\sigma)$ and $(\rho,\uu{u},v)\in \Lm\times \uu{H}_{\mcal{E},0}\times H^{(i)}_{\mcal{E},0}$, the upwind momentum convection flux is expressed as
    \begin{align} 
    \label{mom_flux_up} 
      \sum_{\epsilon\in\tilde{\E}(\Ds)}\fesig(\rho,\uu{u})v_{\epsilon, \mathrm{up}},
    \end{align}
    where $\fesig(\rho,\uu{u})$ is the mass flux across the edge $\epsilon$ of the dual cell $\Ds$ which is a suitable linear combination of the primal mass convection fluxes at the neighbouring edges with constant coefficients if $\epsilon\in\mcal{E}_\intr$, otherwise it is $0$.
\item In the expression \eqref{mom_flux_up}, the velocity  $v_{\epsilon,\mathrm{up}}$ at the internal dual edge is determined by the following upwind choice:
\label{eq:mom_up}
\begin{equation}
v_{\epsilon,\mathrm{up}}=\begin{cases}
v_{\sigma}, &\fesig(\rho,\uu{u})\geqslant 0,\\
v_{\sigma^{\prime}}, &\mathrm{otherwise},
\end{cases}
\end{equation}
where $\epsilon\in\tilde{\E}(\Ds)$, $\epsilon=\Ds|D_{\sigma^{\prime}}$.
\end{itemize}
\end{definition}

\begin{definition}[Discrete gradient, discrete divergence and discrete Laplacian]
The discrete gradient operator  $\grd_{\mcal{E}}:L_{\mcal{M}}(\Omega)\rightarrow\uu{H}_{\mcal{E}}(\Omega)$ is defined by the map $q \mapsto \grd_{\mcal{E}}q=\Big(\D^{(1)}_{\mcal{E}}q,\D^{(2)}_{\mcal{E}}q,\dots,\D^{(d)}_{\mcal{E}}q\Big)$, where for each $i=1,2,\dots,d$, $\partial^{(i)}_{\mcal {E}}q$ denotes
\begin{equation}
\label{eq:dis_grad}
\partial^{(i)}_{\mcal {E}}q=\sum_{\sigma\in \mcal{E}^{(i)}_\intr}(\partial^{(i)}_{\mcal{E}}q)_{\sigma}\mcal{X}_{D_{\sigma}}, \ \mbox{with} \  (\partial^{(i)}_{\mcal{E}}q)_{\sigma}= \frac{\abs{\sigma}}{\abs{D_\sigma}}(q_{L}-q_K)\uu{e}^{(i)}\cdot \uu{\nu}_{\sigma,K}, \; \s=K|L\in\E_\intr^{(i)}.
\end{equation}
The discrete divergence operator $\dive_\M:\uu{H}_{\E}(\Omega)\rightarrow L_{\mcal{M}}(\Omega)$ is defined as $\uu{v} \mapsto \dive_\M \uu{v}=\sum_{K\in\M}(\dive_{\mcal{M}} \uu{v})_K \mcal{X}_{K}$, where for each $K\in\M$, $(\dive_{\mcal{M}} \uu{v})_K $ denotes
\begin{equation}
\label{eq:dis_div}
(\dive_\M \uu{v})_K =\frac{1}{\abs{K}}\sum_{\sigma\in\mcal{E}(K)}\abs{\sigma} v_{\sigma,K},
\end{equation}
where $v_{\sigma,K}$ is as in the mass flux. Finally, we define the discrete Laplacian  $\Delta_{\M}:L_{\mcal{M}}(\Omega)\rightarrow L_{\mcal{\M}}(\Omega)$ via the map $q \mapsto \Delta_{\M}q =\sum_{K\in\M}(\Delta_{\M} q)_K \mcal{X}_{K}$, where for each $K\in\M$, $(\Delta_{\M} q)_K $ denotes
\begin{equation}
\label{eq:disc_lap}
(\Delta_{\M} q)_K  = (\dive_{\M}(\grd_{\E}{q}))_{K}.
\end{equation}
\end{definition}
We conclude this section with the following `gradient-divergence duality'; see, e.g.\ \cite{EG+10,GHL+18} for further details. 
\begin{proposition}
  For any $(q,\uu{v})\in\Lm\times\Hez$, the discrete gradient and divergence operators satisfy the duality relation
  \begin{equation}
    \label{eq:dis_dual}
    \int_{\Omega}q(\dive_\M \uu{v})\dd\uu{x}+\int_{\Omega}\grd_{\mcal{E}}q\cdot\uu{v}\dd\uu{x}=0.
  \end{equation}
\end{proposition}
\begin{remark}
    In this section and throughout the rest of this paper, we have assumed that the dual variables and the gradient vanish on the boundaries for the ease of analysis. Nevertheless, in Section~\ref{sec:num_res}, we also explore several other relevant boundary conditions in our numerical case studies.
\end{remark}

\section{Energy Stable Semi-implicit Scheme}
\label{sec:stable_scheme}
In the following, we introduce our semi-implicit in time, fully-discrete scheme for the EPB system \eqref{eq:r_cons_mas}-\eqref{eq:r_poisson-boltz}.

Let us consider a discretisation $0=t_0<t_1<\cdots<t_N=T$ of the time interval $(0,T)$ and let $\dt=t_{n+1}-t_n$ for $n=0,1,\dots,N-1$ be the constant time-step. We introduce the following fully-discrete scheme for $0\leqslant n\leqslant{N-1}$:
\begin{subequations}
\label{epb_dis}
\begin{equation}
\label{eq:dis_cons_mas}
    \frac{1}{\dt}(\rho_{K}^{n+1}-\rho_{K}^{n})+\frac{1}{\left|K\right|}\sum_{\s\in\E(K)}F_{\sigma,K}^{n}=0 ,\text{$\forall K\in \M, $}
\end{equation}
\begin{align}
    \label{eq:dis_cons_mom}
    \frac{1}{\dt}(\rho_{\Ds}^{n+1}u_{\sigma}^{n+1}-\rho_{\Ds}^{n}u_{\sigma}^{n})+\frac{1}{\left|\Ds\right|}\sum_{\epsilon\in\bar\E(\Ds)}F_{\epsilon,\sigma}^{n}u_{\epsilon, \mathrm{up}}^{n}=-\rho_{\s}^{n}(\partial^{(i)}_{\E}\phi^{n+1})_{\s}+(\partial^{(i)}_{\E}\Lambda^{n})_{\s},\, \text{for} \, 1\leqslant i\leqslant d,\forall\s\in\E_{int}^{(i)},
\end{align}
\begin{equation}
    \label{eq:dis_poisson}
    -\veps^2(\Delta_{\M}\phi^{n+1})_{K}=\rho_{K}^{n+1}-e^{\phi^{n+1}_{K}}.
\end{equation}
\end{subequations}
The mass flux in \eqref{eq:dis_cons_mas} is defined by 
\begin{align}
\label{eq:stabterms}
&F_{\s,K}^n=\abs{\s}(\rho_{\s}^n\uu{u}_{\s,K}^n-\mcal{Q}_{\s,K}^{n+1}),
\end{align}
with the stabilisation term $\mcal{Q}_{\s,K}^{n+1}$ to be chosen after a stability analysis of the scheme. Averaging the mass balance \eqref{eq:dis_cons_mas} over a dual cell yields the dual mass balance 
\begin{equation}
    \label{eq:dis_cons_mass_dual}
    \frac{1}{\dt}(\rho_{\Ds}^{n+1}-\rho_{\Ds}^{n})+\frac{1}{\left|\Ds\right|}\sum_{\epsilon\in\tilde{\E}(\Ds)}F_{\epsilon,\sigma}^{n}=0.
\end{equation}
Using the dual mass balance \eqref{eq:dis_cons_mass_dual} in the momentum update \eqref{eq:dis_cons_mom} gives the following update for the velocity:
\begin{equation}
    \label{eq:dis_vel_dual}
    \frac{u_{\s}^{n+1}-u_{\s}^n}{\dt}+\frac{1}{|\Ds|}\sum_{\epsilon\in\tilde{\E}(\Ds)}(F_{\epsilon,\s}^n)^{-}\frac{u_{\s^{\prime}}^{n}-u_{\s}^n}{\rho_{\Ds}^{n+1}}=-\frac{\rho_{\s}^{n}}{\rho_{\Ds}^{n+1}}(\partial_{\E}^{(i)}\phi^{n+1})_{\s}+\frac{(\partial^{(i)}_{\E}\Lambda^{n})_{\s}}{\rho_{\Ds}^{n+1}}.
\end{equation}
Finally, we take initial approximation for $\rho$ and $\phi$ as the average of the initial data $\rho_{0}$ and $\phi_{0}$ respectively on the primal cells. Analogously, we take initial approximation of $\uu{u}$ as the average of the initial data $\uu{u}_{0}$ on the dual cells, i.e.\
\begin{equation}
\label{eq:dis_ic}
\begin{aligned}
    \rho_{K}^{0}&=\frac{1}{|K|}\int_{K}\rho_{0}(\uu{x})\dd\uu{x}, \ ; \forall K\in\M \\
    \phi_{K}^{0}&=\frac{1}{|K|}\int_{K}\phi_{0}(\uu{x})\dd\uu{x}, \; \forall K\in\M \\
    u_{\s}^{0}&=\frac{1}{|\Ds|}\int_{\Ds}(\uu{u}_{0}(\uu{x}))_{i}\dd\uu{x}, \; \forall \s\in\E_{\intr}^{(i)}, \, 1\leqslant i\leqslant d.
\end{aligned}
\end{equation}

\begin{remark}
Throughout the paper, we denote by $a^\pm=\half(a \pm \abs{a})$, the positive and negative parts of a real number $a$.
\end{remark}
\subsection{Discrete Identities}
\label{sec:id}
The aim of this section is to prove the apriori energy estimates satisfied by the scheme \eqref{eq:dis_cons_mas}-\eqref{eq:dis_poisson} which are discrete counterparts of the stability estimates stated
in Proposition~\ref{prop:r_engy_balance}.
\begin{lemma}[Discrete potential energy identity]
A solution to the system \eqref{eq:dis_cons_mas}-\eqref{eq:dis_poisson} satisfies the following equality for $1\leqslant i\leqslant d,\s\in\E_{\intr}^{(i)}$ and $0\leqslant n\leqslant{N-1}\colon$
\begin{equation}
\label{eq:dis_potbal}
\begin{aligned}    
    &\frac{\veps^2}{4\dt}\sum_{i=1}^{d}\sum_{\s\in\E^{(i)}(K)}\abs{\Ds}\left(\abs{(\partial^{(i)}_\E\phi^{n+1})_{\s}}^2-\abs{(\partial^{(i)}_\E\phi^{n})_{\s}}^2\right)+\frac{\abs{K}}{\dt}\left(e^{\phi_K^{n+1}}(\phi_K^{n+1}-1)-e^{\phi_K^{n}}(\phi_K^{n}-1)\right)+\mcal{R}_K\\
    &-\veps^2\sum_{i=1}^{d}\sum_{\substack{\s\in\E^{(i)}(K)\\\s=K|L}}\abs{\s}\frac{\phi_K^{n+1}+\phi_L^{n+1}}{2}(\partial^{(i)}_\E(\phi^{n+1}-\phi^n))_{\s,K}
    =-\phi_{K}^{n+1}\sum_{\s\in\E(K)}\abs{\s}(\rho_{\s}^{n}u_{\s,K}^{n}-Q_{\s,K}^{n+1}),
\end{aligned}
\end{equation}
with the remainder term 
    \begin{align}
     \label{eq:pot_id_const_b}
     &\mcal{R}_K =\frac{\abs{K}}{\dt}\left(e^{\phi_K^{n+1}}(\phi_K^{n+1}-\phi_K^{n}-1)+e^{\phi_K^{n}}\right)+\frac{\veps^2}{4\dt}\sum_{i=1}^d\sum_{\s\in\E^{(i)}(K)}\abs{\Ds}\abs{(\partial^{(i)}_\E\phi^{n+1})_{\s}-(\partial^{(i)}_\E\phi^{n})_{\s}}^2.
    \end{align}
\end{lemma}
\begin{proof}
After multiplying the discrete mass equation by $|K|\phi_{K}^{n+1}$ and using the relation $(a-b)a=\frac{a^2}{2}-\frac{b^2}{2}+\frac{(a-b)^2}{2}$, for any two real numbers $a$ and $b$, we obtain the final result.

Note that the remainder term $\mcal{R}_{K}$ is positive because $e^z(z-1)\geqslant -1$ for any real number $z$ which further implies $e^{x}(x-y-1)+e^y\geqslant 0$, for any two real numbers $x$ and $y$.
\end{proof}
\begin{lemma}[Discrete kinetic energy identity]
\label{lem:dis_kin_energy_loc}
Any solution to the system \eqref{eq:dis_cons_mas}-\eqref{eq:dis_poisson} satisfies the following identity for $1\leqslant i\leqslant d,\; \s\in\E_\intr^{(i)}$ and $0\leqslant n\leqslant{N-1}$:
\begin{align}
\label{eq:dis_kinid_loc}
&\frac{\abs{\Ds}}{2\dt}(\rho^{n+1}_{\Ds}(u^{n+1}_\sigma)^2 -\rho^n_{\Ds}(u^n_\sigma)^2) +\sum_{\substack{\epsilon\in\tilde{\E}(\Ds)}}F^n_{\epsilon, \sigma}\frac{(u^n_{\epsilon,\mathrm{up}})^2}{2}\\& = -\abs{\Ds}(\rho_\s^n u^n_\sigma-Q_\s^{n+1})(\partial^{(i)}_\E \phi^{n+1})_\sigma-\abs{\Ds}Q_\s^{n+1}(\partial^{(i)}_\E \phi^{n+1})_\s+\abs{\Ds} u_\s^n(\partial^{(i)}_\E \Lambda^{n})_\s+\mcal{R}_{\s},\nonumber
\end{align}
where the remainder term is given by
\begin{align}
\label{eq:dis_kinloc_rem}
\mcal{R}_{\s} =\frac{\abs{\Ds}}{2\dt}\rho^{n+1}_{\Ds}\abs{u^{n+1}_\sigma - u^n_\sigma}^2
+\frac{1}{2}\sum_{\substack{\epsilon\in\tilde{\E}(\Ds)\\ \epsilon = \Ds|D_{\s^\prime}}}(F^n_{\epsilon, \sigma})^{-}(u^n_{\s^\prime} - u^n_\s)^2.
\end{align}
\end{lemma}
\begin{proof}
Multiplying the momentum balance equation \eqref{eq:dis_cons_mom} with $\abs{\Ds}u^n_\sigma$ and using the dual mass balance \eqref{eq:dis_cons_mass_dual}, we readily obtain \eqref{eq:dis_kinid_loc}.
\end{proof}

\begin{theorem}[Total energy estimate]
\label{thm:dis_tot_energy}
Any solution to the system \eqref{eq:dis_cons_mas}-\eqref{eq:dis_poisson} satisfies the global energy inequality
\begin{align}
\label{eq:dis_totbal}
  \sum_{i=1}^{d}\sum_{\s\in\E_{int}^{(i)}}\frac{1}{2}\frac{|\Ds|}{\dt}\left(\rho_{\Ds}^{n+1}(u_{\sigma}^{n+1})^2-\rho_{\Ds}^{n}(u_{\sigma}^{n})^2\right)+&\sum_{K\in\M}\frac{\abs{K}}{\dt}\left(e^{\phi_K^{n+1}}(\phi_K^{n+1}-1)-e^{\phi_K^{n}}(\phi_K^{n}-1)\right)\\
  &+\frac{\veps^2}{2}\sum_{i=1}^d\sum_{\s\in\E^{(i)}_{int}}\frac{\abs{\Ds}}{\dt}\left(|(\partial^{(i)}_\E\phi^{n+1})_{\s}|^2-|(\partial^{(i)}_\E\phi^{n})_{\s}|^2\right) \leqslant 0\nonumber
\end{align}
under the following conditions. 
\begin{enumerate}[label=(\roman*)]
    \item A CFL restriction on the time-step:
    \begin{equation}
        \label{eq:cfl}
        \dt\leqslant\min\Big\{\frac{\rho_{\Ds}^{n+1}\abs{\Ds}}{4\sum_{\epsilon\in\tilde{\E}(\Ds)}(-(F^n_{\epsilon,\s})^{-})}, \frac{1}{2}\Big(\frac{1}{\alpha_{K,\s}^n}\Big)^{\half}\Big\},
    \end{equation}
    with 
    \begin{equation}
    \label{eq:eqn_ansigma}
        \alpha_{K, \s}^n = \frac{\abs{\partial K}}{\abs{K}}\frac{\abs{\sigma}}{\abs{\Ds}}\frac{1}{\rho^{n+1}_{\Ds}}.
    \end{equation}
    \item The following choices for the stabilisation terms $\uu{Q}^{n+1}$ and $\Lambda^{n}$:
    \begin{align}
        \label{eq:dis_Q}
        Q_{\s}^{n+1} &=\eta\dt(\partial^{(i)}_{\E}\phi^{n+1})_{\s},\;\forall\s\in\E^{(i)}_\intr,\; 1\leqslant i \leqslant d, \\
        \label{eq:dis_Lambda}
        \Lambda_K^n &=\frac{\dt}{|K|}\sum_{\s\in\Ek}|\s| u_{\s,K}^n,\;\forall K\in\M,
    \end{align} 
\end{enumerate}
with $\eta\geqslant\frac{(\rho_{\s}^{n})^2}{\rho_{\Ds}^{n+1}}$ for each $\s\in\E^{(i)}_\intr, \ 1\leqslant i \leqslant d$.
\end{theorem}
\begin{proof}
Taking the sum over all $\s\in\E^{(i)}_\intr,~i=1,2,\dots,d$ in the discrete kinetic energy identity \eqref{eq:dis_kinid_loc}, subsequently dropping the locally conservative terms, and finally rearranging the summands gives
\begin{equation}
\begin{aligned}
\label{eq:sum_kinid_glob}
    &\sum_{\s\in\E_{\intr}}\frac{\abs{\Ds}}{2\dt}(\rho^{n+1}_{\Ds}(u^{n+1}_\sigma)^2 -\rho^n_{\Ds}(u^n_\sigma)^2) \\
    &= -\sum_{i=1}^{d}\sum_{\s\in\E_{int}^{(i)}}\abs{\Ds}\rho_\s^n u^n_\sigma(\partial^{(i)}_\E \phi^{n+1})_\sigma +\sum_{i=1}^{d}\sum_{\s\in\E_{int}^{(i)}}\abs{\Ds}\rho_\s^n u^n_\sigma(\partial^{(i)}_\E \Lambda^n)_\sigma +\sum_{i=1}^{d}\sum_{\s\in\E_{\intr}^{(i)}}\mcal{R}_{\s}.
\end{aligned}    
\end{equation}
Next, we estimate the net remainder term on the right hand side of \eqref{eq:sum_kinid_glob}. To this end, squaring the velocity update \eqref{eq:dis_vel_dual} and using the inequality $(a+b)^2 \leqslant 2a^2 + 2b^2$, we get for each $\s\in\E^{(i)}_\intr$, $1\leqslant i\leqslant d$,
\begin{equation}
\begin{aligned}
\label{eq:dis_kinloc_rhs_rem}
\half\rho^{n+1}_{\Ds}\absq{u^{n+1}_\s - u^n_\s} &\leqslant \dt^2\frac{(\rho^n_\s)^2}{\rho^{n+1}_{\Ds}}((\partial^{(i)}_\E\phi^{n+1})_\s)^2+2\dt^2\frac{1}{\rho^{n+1}_{\Ds}}(\partial^{(i)}_\E\Lambda^n)_\s^2\\
 &+2\frac{\dt^2}{\rho^{n+1}_{\Ds}}\Big(\frac{1}{\abs{\Ds}}\sum_{\epsilon\in\E(\Ds)}(F^n_{\epsilon,\sigma})^{-}(u^n_{\s^{\prime}} - u^n_\s)\Big)^2.
\end{aligned}
\end{equation}
Estimating the remainder term $\mcal{R}_{\s}$ using \eqref{eq:dis_kinloc_rhs_rem} and applying the Cauchy-Schwartz inequality to the second term on the right hand side of \eqref{eq:dis_kinloc_rhs_rem}, analogously as done in \cite[Lemma 3.1]{DVB17}, the equation \eqref{eq:sum_kinid_glob} finally yields
\begin{equation}
\begin{aligned}
\label{eq:dis_kinid_est}
 &\sum_{\s\in\E_{int}}\frac{|\Ds|}{2\dt}\left(\rho_{\Ds}^{n+1}(u_{\sigma}^{n+1})^2-\rho_{\Ds}^{n}(u_{\sigma}^{n})^2\right)\leqslant -\sum_{i=1}^{d}\sum_{\s\in\E_{int}^{(i)}}|\Ds|\rho_{\s}^{n}u_{\s}^{n}(\partial^{(i)}_{\mcal{E}}\phi^{n+1})_{\s}\\
&+\sum_{i=1}^d\sum_{\s\in\E^{(i)}_\intr}\bigg(\half + \frac{2\dt}{\rho^{n+1}_{\Ds}\abs{\Ds}}(F^n_{\epsilon,\s})^{-}\bigg)\sum_{\epsilon\in\tilde{\E}(\Ds)}(F^n_{\epsilon,\s})^{-}(u^n_{\s^{\prime}} - u^n_\s)^2+\mcal{A}+\tilde{\mcal{R}}.
\end{aligned}
\end{equation}
Here,
\begin{subequations}
\begin{align}
\mcal{A}&=4\dt\sum_{K\in\M}(\Lambda_K^n)^2\sum_{\s\in\Ek}\frac{|\s|^2}{|\Ds|}\frac{1}{\rho_{\Ds}^{n+1}}-\sum_{K\in\M}\Lambda_K^n\sum_{\s\in\Ek}|\s|u_{\s,K}^{n},\\
\tilde{\mcal{R}}&=\dt\sum_{i=1}^{d}\sum_{\s\in\E_{int}^{(i)}}|\Ds|\frac{(\rho_{\s}^n)^2}{\rho_{\Ds}^{n+1}}(\partial^{(i)}_{\E}\phi^{n+1})_{\s}^2.
\end{align}
\end{subequations}
Next, taking the sum over all $K\in\M$ in the potential balance \eqref{eq:dis_potbal}, dropping the locally conservative third term on the right hand side, and using the discrete div-grad duality \eqref{eq:dis_dual} yields
\begin{equation}
\label{eq:dis_potbal_glob}
\begin{aligned}    
    &\frac{\veps^2}{2\dt}\sum_{i=1}^{d}\sum_{\s\in\E^{(i)}}\abs{\Ds}\left(\abs{(\partial^{(i)}_\E\phi^{n+1})_{\s}}^2-\abs{(\partial^{(i)}_\E\phi^{n})_{\s}}^2\right)+\sum_{K\in\M}\sum_{\s\in\E(K)}|\s|\rho_{\s}^{n}u_{\s,K}^{n}\phi_{K}^{n+1}\\
 &+\sum_{K\in\M}\frac{\abs{K}}{\dt}\left(e^{\phi_K^{n+1}}(\phi_K^{n+1}-1)-e^{\phi_K^{n}}(\phi_K^{n}-1)\right)\leqslant-\sum_{i=1}^{d}\sum_{\s\in\E^{(i)}}|\Ds|\Q_{\s}^{n+1}(\D^{(i)}_{\E}\phi^{n+1})_{\s}.
\end{aligned}
\end{equation}
In the above inequality, we have used the identity \eqref{eq:dis_potbal} and the fact that $\mcal{R}_K\geqslant 0$ for each $K\in\M$. 

Finally, adding \eqref{eq:dis_kinid_est} and \eqref{eq:dis_potbal_glob}, and using the discrete div-grad duality \eqref{eq:dis_dual} again, we obtain the inequality
\begin{equation}
\label{eq:dis_tot_est}
\begin{aligned}
    \sum_{i=1}^{d}\sum_{\s\in\E_{\intr}^{(i)}}\frac{1}{2}\frac{|\Ds|}{\dt}\left(\rho_{\Ds}^{n+1}(u_{\sigma}^{n+1})^2-\rho_{\Ds}^{n}(u_{\sigma}^{n})^2\right)+\frac{\veps^2}{2}\sum_{i=1}^d\sum_{\s\in\E^{(i)}_\intr}\frac{\abs{\Ds}}{\dt}\left(|(\partial^{(i)}_\E\phi^{n+1})_{\s}|^2-|(\partial^{(i)}_\E\phi^{n})_{\s}|^2\right)\\+\sum_{K\in\M}\frac{\abs{K}}{\dt}\left(e^{\phi_K^{n+1}}(\phi_K^{n+1}-1)-e^{\phi_K^{n}}(\phi_K^{n}-1)\right)\leqslant \mcal{A} + \mcal{R}+\mcal{S},
\end{aligned}
\end{equation}
where
\begin{align}
     \mcal{A} &= \dt\sum_{K\in\M}\sum_{\s\in\Ek}\frac{\abs{\s}}{\abs{\partial K}}(4\dt^2\alpha_{K, \s}^n-1)\frac{1}{\abs{K}}\Big(\sum_{\s\in\Ek}|\s| u_{\s,K}^n\Big)^2,\\
    \mcal{R} &= \dt\sum_{i=1}^{d}\sum_{\s\in\E_{\intr}^{(i)}}|\Ds|(\partial_{\E}^{(i)}\phi^{n+1})_{\s}^2\bigg(\frac{(\rho_\s^n)^2}{\rho_{\Ds}^{n+1}}-\eta\bigg), \\
     \mcal{S} &= \sum_{i=1}^d\sum_{\s\in\E^{(i)}_\intr}\bigg(\half + \frac{2\dt}{\rho^{n+1}_{\Ds}\abs{\Ds}}(F^n_{\epsilon,\s})^{-}\bigg)\sum_{\epsilon\in\tilde{\E}(\Ds)}(F^n_{\epsilon,\s})^{-}(u^n_{\s^{\prime}} - u^n_\s)^2,
\end{align}
with $\alpha_{K, \s}^n = \frac{\abs{\partial K}}{\abs{K}}\frac{\abs{\sigma}}{\abs{\Ds}}\frac{1}{\rho^{n+1}_{\Ds}}$ and the stabilisation terms  $\uu{Q}^{n+1}$ and $\Lambda^{n}$ being chosen as in \eqref{eq:dis_Q} and \eqref{eq:dis_Lambda}, respectively. 

At the end, in order to get the required entropy inequality \eqref{eq:dis_totbal}, it is enough to show that the remainder terms $\mcal{A}$, $\mcal{R}$, $\mcal{S}$ are non-positive. It is straightforward to verify that the following conditions are sufficient to make the remainder terms non-positive:
\begin{align}
    \label{eq:Q_quad}
    \eta&\geqslant \frac{(\rho_\s^n)^2}{\rho_{\Ds}^{n+1}},\;\forall\s\in\E^{(i)}_\intr,\; 1\leqslant i\leqslant d,
    \\
    \label{eq:cfl1}
    4\dt^2\alpha^n_{K,\s}-1 &\leqslant 0,\;\forall K\in\M,
    \\
    \label{eq:tstep_ke}
    \half + \frac{2\dt}{\rho^{n+1}_{\Ds}\abs{\Ds}}(F^n_{\epsilon,\s})^{-} &\geqslant 0,\;\forall\s\in\E^{(i)}_\intr,\; 1\leqslant i\leqslant d.
\end{align}
\end{proof}
\begin{remark}
  In the above theorem, the choice for the stabilisation parameter $\eta$ is implicit in nature. However, we observe that the following implicit time-step restriction 
  \begin{equation}
    \dfrac{\dt}{|\Ds|}\sum_{\epsilon\in\tilde\E(\Ds)}
    \frac{\abs{F^n_{\epsilon,\s}}}{\rho_{\Ds}^{n+1}}\leqslant\frac{1}{4}
    \label{eq:suftime}  
  \end{equation}
  gives a sufficient condition which, along with the time-step restriction \eqref{eq:cfl1}, yields the required CFL condition \eqref{eq:cfl}. From \eqref{eq:suftime} and the dual mass balance
  \eqref{eq:dis_cons_mass_dual}, we deduce that
  \begin{equation}
    \frac{5}{4}\rho_{\Ds}^{n+1} - \rho_{\Ds}^n \geqslant  \rho_{\Ds}^{n+1}
    - \rho_{\Ds}^n +
    \frac{\dt}{\abs{\Ds}}\sum_{\epsilon\in\tilde\E(\Ds)}\abs{F^n_{\epsilon,\s}}\geqslant
    0. 
  \end{equation}
  Hence, we get
  \begin{equation}
    \label{eq:rho_5/4}
    \frac{\rho_{\Ds}^n }{\rho_{\Ds}^{n+1}}\leqslant \frac{5}{4}.
  \end{equation}
  Therefore, at each interface $\sigma$, choosing $\eta$ such that
  \begin{equation}
  \label{eq:eta_expl}
    \eta>\frac{5(\rho_\s^n)^2}{4\rho_{\Ds}^n},
  \end{equation}
  will enable an explicit choice for $\eta$; see also \cite{AGK23, DVB20} for analogous considerations.   
\end{remark}

\begin{remark}
  Note the the stability analysis presented above does not rely on the choice of $\rho_\s$. However, it is the quasineutral limit analysed in Section~\ref{sec:quas_lim} that necessitates the above choice.
\end{remark}

\subsection{Existence of a Numerical Solution}
\label{sec:exist_soln}
The mass update \eqref{eq:dis_cons_mas} and the Poisson equation \eqref{eq:dis_poisson} are coupled due to the presence of the stabilisation terms and both are nonlinear in $\phi^{n+1}$. However, once $\phi^{n+1}$ is calculated, the mass and the momentum updates can be explicitly evaluated to get the density and the velocity. In what follows, we establish the existence of a discrete solution to the numerical scheme \eqref{eq:dis_cons_mas}-\eqref{eq:dis_poisson}. Our treatment uses a few classical tools from topological degree theory in finite dimensions \cite{Dei85} and a discrete comparison principle for nonlinear elliptic equations.  
\begin{theorem}
    \label{thm:existence}
    Let $(\rho^n,\uu{u}^n, \phi^n)\in
    L_{\mcal{M}}(\Omega)\times\uu{H}_{\mcal{E},0}(\Omega)\times L_{\mcal{M}}(\Omega)$ be such
    that $\rho^{n}>0$ on $\Omega$. Then, there exists a solution
    $(\rho^{n+1},\uu{u}^{n+1},\phi^{n+1})\in
    L_{\mcal{M}}(\Omega)\times\uu{H}_{\mcal{E},0}(\Omega)\times L_{\mcal{M}}(\Omega)$ of
    \eqref{eq:dis_cons_mas}-\eqref{eq:dis_poisson}, satisfying
    $\rho^{n+1}>0$ on $\Omega$. 
\end{theorem}
\begin{proof}
Let $\bar{C}>0$ and $\underline{c}>0$ be constants such that 
\begin{equation}
\label{eq:C_thm_ex}
    \bar{C}>\frac{\abs{\Omega}\max_{K\in\mcal{M}}\{\rho_{K}^n\}}{\min_{K\in\mcal{M}}\{\abs{K}\}},\, \quad\underline{c}<\frac{4}{5}\min_{K\in\mcal{M}}\{\rho_K^n\}.
\end{equation}
Consider the bounded, open subset $V$ of $L_{\mcal{M}}(\Omega)^2$ defined by
\begin{equation*}
    V=\big\{(\rho,\phi)\in L_{\mcal{M}}(\Omega)^2\colon \underline{c}<\rho_{K}<\bar{C} \text{ and } \ln(\underline{c})<\phi_{K}<\ln(\bar{C}), \ \forall K\in\mcal{M}\big\}.
\end{equation*}
We introduce a continuous function $H\colon[0,1]\times L_{\mcal{M}}(\Omega)^2\rightarrow L_{\mcal{M}}(\Omega)$, via 
\begin{equation*}
    H(\lambda,\rho,\phi)=\sum_{K\in\mcal{M}}H(\lambda,\rho,\phi)_{K}\mcal{X}_{K},
\end{equation*}
where
\begin{equation*}
    H(\lambda,\rho,\phi)_{K}=\frac{1}{\dt}(\rho_{K}-\rho^n_{K})+\frac{\lambda}{\abs{K}}\sum_{\s\in\mcal{E}(K)}\abs{\s}(\rho^n_{\s} u^n_{\s,K}-\eta\dt(\partial_{\E}^{(i)}\phi)_{\s,K}), \ \forall K \in \mcal{M}.
\end{equation*}
Here, $\phi$ is the solution of the discrete Neumann elliptic problem
\begin{equation}
\label{eq:dis_neumann}
\begin{aligned}
   -\lambda\veps^2(\Delta_{\M}\phi)_{K}+e^{\phi_K}&=\rho_K, \quad \forall K\in\M \\
   (\partial_\E^{(i)}\phi)_\s&=0, \quad \forall \s\in\E_\extr^{(i)}, \quad 1\leq i\leq d.
\end{aligned}
\end{equation}
Note that the above problem is nonlinear. In the Appendix, we prove a comparison principle that ensures the existence of a solution for a given $\rho_K>0$. 

It is easy to note that $H(\lambda,\rho,\phi)$ is a homotopy connecting $H(0,\rho,\phi)$ and $H(1,\rho,\phi)$. In order to establish the existence of discrete solution $(\rho^{n+1},\phi^{n+1})$ of \eqref{eq:dis_cons_mas} and \eqref{eq:dis_poisson}, we need to show that for any $\lambda\in [0,1]$, $H(\lambda,\cdot,\cdot)$ has a non-zero topological degree with respect to $V$. In other words, we need to show that $H(\lambda,\cdot,\cdot)\neq 0$ on $\D V$, for all $\lambda\in [0,1]$. On the contrary, let us assume that $H(\lambda,\rho,\phi)=0$, for some $\lambda\in[0,1]$ and $(\rho,\phi)\in\D V$, which implies
\begin{equation}
\label{eq:H_zero}
    \frac{1}{\dt}(\rho_{K}-\rho^n_{K})+\frac{\lambda}{|K|}\sum_{\s\in\mcal{E}(K)}\abs{\s}(\rho^n_{\s} u^n_{\s,K}-\eta\dt(\partial_{\E}^{(i)}\phi)_{\s,K})=0, \ \forall K \in \mcal{M}.
\end{equation}
with $\phi$ solving the problem \eqref{eq:dis_neumann}. Summing \eqref{eq:H_zero} over all $K\in\mcal{M}$, and using the definition of $\bar{C}$ from \eqref{eq:C_thm_ex}, we obtain
\begin{equation}
\label{eq:rhoKub}
    \rho_K\leq \frac{\abs{\Omega}\max_{K\in\mcal{M}}\{\rho_{K}^n\}}{\min_{K\in\mcal{M}}\{\abs{K}\}}< \bar{C}, \ \forall K\in\mcal{M}.
\end{equation}
Further, from the bound \eqref{eq:rho_5/4} obtained from the sufficient condition on the time-step, we have for each $K\in\mcal{M}$ and any $\lambda\in[0,1]$,
\begin{equation}
    \label{eq:ex_pos}
    \rho_K>\frac{4}{5}\min_{K\in\mcal{M}}\{\rho_K^n\}>\underline{c}.
\end{equation}
Combining the above inequality with \eqref{eq:rhoKub} gives $\underline{c}<\rho_K<\bar{C}$. Theorem~\ref{thm:phi_bdd} then gives the inequality $\ln(\underline{c})<\phi_{K}<\ln(\bar{C})$ for each $K\in\mcal{M}$, a contradiction as $(\rho,\phi)\in\D V$. Hence, $0\notin H(\{\lambda\}\times\partial V)$ for any $\lambda\in[0,1]$, and by topological degree arguments for finite dimensional spaces \cite{Dei85}, we get that $\deg(H(1,\cdot, \cdot),V,0)=\deg(H(0,\cdot, \cdot),V,0)$. Since $\deg(H(0,\cdot,\cdot),V,0)\neq 0$, we conclude that $H(1,\cdot,\cdot)$ has a zero in $V$, i.e.\ there exists a solution $(\rho^{n+1},\phi^{n+1})\in L_{\mcal{M}}(\Omega)^2$ of \eqref{eq:dis_cons_mas}, such that $\rho^{n+1}>0$.
\end{proof}
\begin{remark}
    Note that the positivity of $\rho$ follows from the sufficient condition on the time-step in Section~\ref{sec:num_res}.
\end{remark}

\section{Weak Consistency of the Scheme}
\label{sec:weak_cons}
The aim of this section is to establish a Lax-Wendroff-type weak consistency of the scheme \eqref{eq:dis_cons_mas}-\eqref{eq:dis_poisson}. We consider a sequence of numerical solutions satisfying certain boundedness assumptions, and generated by successive mesh refinements. If this sequence of solutions converges strongly to a limit, then the limit must be a weak solution of the EPB system. Considering a bounded initial data $(\rho^{\veps}_0,\uu{u}^{\veps}_0,\phi^{\veps}_0)\in L^\infty(\Omega)^{d+2}$ and recalling the definition of a weak solution as in Definition~\ref{def:weak_soln_defn}, we prove the following Lax-Wendroff-type consistency result. In the theorem, we omit the superscript $\veps$ on the unknowns, wherever necessary for the sake of simplicity of writing.
\begin{theorem}
\label{thm:weak_cons}
Let $\Omega$ be an open, bounded set of $\mbb{R}^d$. Assume that $\big(\mcal{T}^{(m)},\delta t^{(m)}\big)_{m\in\mbb{N}}$ is a sequence of discretisations such that both $\lim_{m\rightarrow \infty}\delta t^{(m)}$ and $\lim_{m\rightarrow \infty}h_{\mcal{T}^{(m)}}$ are $0$. Let $\big(\rho^{(m)},\uu{u}^{(m)},\phi^{(m)}\big)_{m\in\mbb{N}}$ be the corresponding sequence of discrete solutions with respect to an initial data $(\rho^\veps_0,\uu{u}^\veps_0,\phi^\veps_{0})\in L^\infty(\Omega)^{d+2}$, where $\veps>0$ is fixed. We assume that $(\rho^{(m)},\uu{u}^{(m)}, \phi^{(m)})_{m\in\mbb{N}}$ satisfies the following.
\begin{enumerate}[label=(\roman*)]
\item $\big(\rho^{(m)},\uu{u}^{(m)}\big)_{m\in\mbb{N}}$ is uniformly bounded in $L^\infty(Q)^{1+d}$, i.e.\ 
\begin{align}
\ubar{C}<(\rho^{(m)})^n_K &\leqslant \bar{C}, \ \forall K\in\mcal{M}^{(m)}, \ 0\leqslant n\leqslant N^{(m)}, \ \forall m\in\mbb{N}\label{eq:dens_abs_bound}, \\
|(u^{(m)})^n_\sigma| &\leqslant C, \ \forall \sigma\in\mcal{E}^{(m)}, \ 0\leqslant n\leqslant N^{(m)}, \ \forall m\in\mbb{N}\label{eq:u_abs_bound}. 
\end{align}
where $\ubar{C}, \bar{C}, C>0$ are constants independent of the discretisations. 
\item $\big(\rho^{(m)},\uu{u}^{(m)}, \phi^{(m)}\big)_{m\in\mbb{N}}$ converges to $(\rho^\veps,\uu{u}^\veps, \phi^\veps)\in L^\infty(0, T; L^\infty(\Omega)\times L^\infty(\Omega)^d\times H^1(\Omega))$ in $L^r(Q_T)^{1+d+1}$ for $1\leqslant r<\infty$.
\end{enumerate}
Assume furthermore that the sequence of discretisations $\big(\mcal{T}^{(m)},\delta t^{(m)}\big)_{m\in\mbb{N}}$ satisfies the mesh regularity conditions:
\begin{equation}
\label{eq:CFL_meshpar}
\frac{\delta t^{(m)}}{\min_{K\in\mcal{M}^{(m)}}\abs{K}}\leqslant\theta,\ \max_{K \in \mcal{M}^{(m)}} \frac{\diam(K)^2}{\abs{K}}\leqslant\theta,\;\forall m\in\mbb{N},
\end{equation}
where $\theta>0$ is independent of the discretisations. Then $(\rho^\veps,\uu{u}^\veps,\phi^\veps)$ satisfies the weak formulation \eqref{eq:weak_soln_mas}-\eqref{eq:weak_soln_poisson}.
\end{theorem}
\begin{proof}
We adopt an approach similar to that in \cite[Lemma 4.1]{GHL22} and hence we skip most of the calculations except the ones related to the stabilisation terms and the Poisson equation. To begin with, we multiply the mass update \eqref{eq:dis_cons_mas} by $\psi^n_K\abs{K}$, where $\psi^n_K$ denotes the interpolant of a smooth and compactly supported test function $\psi$ at $(t^n,\uu{x}_K)$. To simplify the calculations, we choose a sufficiently large value for $m$ such that $(\psi^{(m)})^n_K = 0$, for all $K\in\M^{(m)}$ that has a face lying on the boundary. We obtain the following remainder term $R^{(m)}$ from the stabilisation terms in the mass update \eqref{eq:dis_cons_mas} upon summing over all $n=0,\dots,N^{(m)}-1$ and $K\in\M^{(m)}$:
\begin{equation}
R^{(m)}=\sum_{n=0}^{N^{(m)}-1}\dt^n\sum_{K\in\M}\diam(K)\sum_{\s\in\E(K)}\abs{\s}\abs{Q^{n+1}_\sigma}.
\end{equation}
Thanks to the uniform boundedness assumption \eqref{eq:dens_abs_bound}, we can estimate the above term to get
\begin{equation}
\label{eq:source_trunc}
R^{(m)}\leqslant C(\psi, \bar{C},\ubar{C})\sum_{n=0}^{N^{(m)}-1}\dt^n\sum_{K\in\M}\diam(K)\sum_{\s\in\E(K)}|\s|\dt^n\frac{\abs{\s}}{\abs{\Ds}}\abs{\phi^{n+1}_{L}-\phi^{n+1}_{K}},
\end{equation}
where $C(\psi, \bar{C}, \ubar{C})>0$ is a constant depending only on $\psi$, $\bar{C}$ and $\ubar{C}$. Using similar techniques as outlined in \cite[Lemma A.1]{GHL22} regarding the convergence of discrete space translates, we can establish that the right hand side of \eqref{eq:source_trunc} tends to $0$ as $m\to\infty$ under the uniform boundedness assumptions \eqref{eq:dens_abs_bound}, the strong convergence assumptions stated in (ii), and the regularity conditions \eqref{eq:CFL_meshpar} on the mesh parameters.

Let us now proceed to achieve the weak consistency of the momentum update \eqref{eq:dis_cons_mom}. We start multiplying \eqref{eq:dis_cons_mom} by $\psi^n_\s\abs{\Ds}$, where $\psi^n_\s$ denotes the interpolant of a compactly supported, smooth test function $\uu{\psi}$ at $(t^n,\uu{x}_\s)$. Under the given assumptions, the weak consistency of the discrete convection operator in the momentum update \eqref{eq:dis_cons_mom} follows using a method akin to the proof of \cite[Theorem 4.1]{HLN+23}. The residual terms appearing due to the stabilisation again converge to $0$ by a similar argument as in the case of the mass update. The consistency of the right hand side terms in the momentum update \eqref{eq:dis_cons_mom} is the only task left to complete, which can be achieved following a similar approach as in \cite[Theorem 5.2]{AGK24} once we obtain uniform bound for the sequence of discrete solutions $(\phi^{(m)})_{m\in\mbb{N}}$. Note that the uniform bound for $(\phi^{(m)})_{m\in\mbb{N}}$ can be derived using the uniform boundedness \eqref{eq:dens_abs_bound} of $(\rho^{(m)})_{m\in\mbb{N}}$, as shown in Theorem~\ref{thm:phi_bdd}.

Finally, the proof of the weak consistency of \eqref{eq:dis_poisson} mostly relies on the similar approach as in \cite[Theorem 5.2]{AGK24}, incorporating the additional information regarding the convergence of $(e^{\phi^{(m)}})_{m\in\mbb{N}}$ to $e^{\phi^\veps}$ in $L^r(Q_T)$ for $1\leqslant r<\infty$, which follows from the uniform boundedness and the strong convergence of $(\phi^{(m)})_{m\in\mbb{N}}$.
\end{proof}

\section{Quasineutral Limit of the Scheme}
\label{sec:quas_lim}
In this section, we study the AP property of the semi-implicit scheme \eqref{epb_dis}. We show that a solution of the semi-implicit scheme \eqref{epb_dis} converges, up to a subsequence, to a solution of a limiting semi-implicit scheme for the ICE system \eqref{eq:iso_euler} as $\veps\rightarrow 0$. To begin with, we assume throughout this section that $(\rho^\veps, \uu{u}^\veps, \phi^\veps)_{\veps>0}\in L^{\infty}(0,T;\Lm\times\uu{H}_{\E,0}(\Omega)\times\Lm)$ is a sequence of approximate solutions to the scheme \eqref{eq:dis_cons_mas}-\eqref{eq:dis_poisson} with respect to a well-prepared initial data as defined below. 
\begin{definition}
\label{def:well_prep_data}
    A sequence of initial data $(\rho^\veps_0, \uu{u}^\veps_0, \phi^\veps_0)_{\veps>0}$ is said to be well-prepared if $\rho^\veps_0 \in L^\infty(\Omega)$, $\rho^\veps_0 > 0 \ \mbox{a.e.}$, $\uu{u}^\veps_0 \in L^2(\Omega)^d$, $\phi^\veps_0 \in L^\infty(\Omega)$, and there exists a constant $C>0$, independent of $\veps$, and a positive function $\tilde{\rho}_0\in L^\infty(\Omega)$, such that
    \begin{enumerate}[label=(\roman*)]
        \item $\rho^\veps_0$ converges to $\tilde{\rho}_0$ in $L^2(\Omega)$ as $\veps\to0$,
        \item $\bigintssss_{\Omega}\big(\frac{\veps^2}{2} \abs{\bgrd\phi^\veps_0}^{2} + (\phi^\veps_0-1)e^{\phi^\veps_0} +\frac{1}{2} \rho^\veps_0 |\uu{u}^\veps_0|^2\big) \dd \uu{x}\leqslant C$ for all $\veps>0$.
    \end{enumerate}
\end{definition}
The subsequent theorem states that any solution of the semi-implicit scheme satisfies a discrete equivalent of the global energy estimate \eqref{eq:glob_eng}. The result follows by summing the inequality \eqref{eq:dis_tot_est} about the total energy over the time-steps. For the ease of writing, we skip the $\veps$ dependence on the functions, whenever there is no scope for confusion. 
\begin{theorem}[Global discrete energy inequality]
    Assume that the initial data $(\rho^\veps_0, \uu{u}^\veps_0, \phi^\veps_0)$ is well-prepared in the sense of Definition~\ref{def:well_prep_data}. Then there exists a solution ($\rho^n,\uu{u}^n, \phi^n)_{0\leqslant n\leqslant N}$ to the scheme \eqref{epb_dis}, satisfying the following inequality for all $0\leqslant n\leqslant N-1$: 
    \begin{equation}
        \label{eq:global_energy_ineq}
        \begin{aligned}
    &\sum_{i=1}^{d}\sum_{\s\in\E_{\intr}^{(i)}}\frac{1}{2}|\Ds|\rho_{\Ds}^{n+1}(u_{\sigma}^{n+1})^2 +\frac{\veps^2}{2}\sum_{i=1}^d\sum_{\s\in\E^{(i)}_{\intr}}\abs{\Ds}|(\partial^{(i)}_\E\phi^{n+1})_{\s}|^2\\
    &+\sum_{K\in\M}\abs{K}e^{\phi_K^{n+1}}(\phi_K^{n+1}-1)+\dt^2\sum_{k=0}^{n}\sum_{i=1}^{d}\sum_{\s\in\E_\intr^{(i)}}\abs{\Ds}\Big(\eta-\frac{(\rho_\s^k)^2}{\rho^{k+1}_{\Ds}}\Big)(\partial_\E^{(i)}\phi^{k+1})_\s^2 \leqslant C, 
    \end{aligned}
    \end{equation}
    where $C > 0$ is a constant independent of $\veps$.
\end{theorem}
\begin{remark}
    Using the elementary inequality $e^z(z-1)\geqslant -1$, for all real number $z$, from \eqref{eq:global_energy_ineq} we can further obtain the following estimate:
    \begin{equation}
    \begin{aligned}
    \label{eq:rm_global_energy_ineq}
    &\sum_{i=1}^{d}\sum_{\s\in\E_{\intr}^{(i)}}\frac{1}{2}|\Ds|\rho_{\Ds}^{n+1}(u_{\sigma}^{n+1})^2 +\frac{\veps^2}{2}\sum_{i=1}^d\sum_{\s\in\E^{(i)}_{\intr}}\abs{\Ds}|(\partial^{(i)}_\E\phi^{n+1})_{\s}|^2\\
    &+\dt^2\sum_{k=0}^{n}\sum_{i=1}^{d}\sum_{\s\in\E_\intr^{(i)}}\abs{\Ds}\Big(\eta-\frac{(\rho_\s^k)^2}{\rho^{k+1}_{\Ds}}\Big)(\partial_\E^{(i)}\phi^{k+1})_\s^2\leqslant C,
    \end{aligned}
    \end{equation}
    where $C$ depends on the initial datum and the domain. The above inequality will be used later to establish the quasineutral limit of the scheme \eqref{epb_dis}.
\end{remark}
The analysis carried out subsequently depends on the assumption that the sequence of approximate solutions for the density is uniformly bounded, i.e.\ 
\begin{equation}
    \label{eq:den_bdd}
    0<\ubar{C}<(\rho^{\veps})^n_K<\bar{C}, \quad \forall K\in\M \text{ and } 0\leqslant n\leqslant N,
\end{equation}
where the constants $\ubar{C}$ and $\bar{C}$ depend on the mesh and the time-step, but not on $\veps$. This hypothesis will play a crucial role in the following asymptotic convergence analysis as it will be used to control the velocity and the electrostatic potential.
\begin{lemma}[]
\label{lem:ql_bdd_den}
    There exists a constant $C$, independent of $\veps$, such that $\norm{\rho^\veps} _{L^{\infty}(0, T; L^{1}(\Omega))}\leqslant C$. Moreover, up to a subsequence, $\rho^\veps$ converges, in any discrete norm, to a positive function $\tilde{\rho}\in L^{\infty}(0,T;\Lm)$. 
\end{lemma}
\begin{proof}
    Summing \eqref{eq:dis_cons_mas} over $K\in\M$ and using the conservativity of the mass flux, we obtain
    \begin{equation*}
        \sum_{K\in\M}\abs{K}(\rho^\veps)^{n+1}_K=\sum_{K\in\M}\abs{K}(\rho^\veps)^{0}_K \quad \forall \, 0\leqslant n\leqslant N-1.
    \end{equation*}
     Using the positivity of the density we obtain that $\norm{\rho^\veps}_{L^{\infty}(0, T; L^{1}(\Omega))}\leqslant C$. Consequently, a subsequence of $\rho^\veps$ converges in the sense of distributions to some $\tilde{\rho}\in \Lm$. A finite dimensional argument along with the hypothesis \eqref{eq:den_bdd} then easily shows that $\rho^\veps$ converges, in any discrete norm, to a positive function $\tilde{\rho}\in L^{\infty}(0,T;\Lm)$.
\end{proof}
\begin{corollary}[Control of the velocity]
\label{cor:bdd_vel}
     The sequence of approximate solutions for the velocity component $(\uu{u}^\veps)_{\veps>0} \subset L^{\infty}(0, T; \uu{H}_{\E,0}(\Omega))$ is uniformly bounded with respect to $\veps$, i.e.\ $\norm{\uu{u}^\veps}_{L^{\infty}(0,T; L^2(\Omega)^d)}\leqslant C$, where $C$ is a constant independent of $\veps$.
\end{corollary} 
\begin{proof}
The required bound on the velocity follows from \eqref{eq:rm_global_energy_ineq} after using the uniform boundedness hypothesis \eqref{eq:den_bdd} and noting that the second and third terms are positive.  
\end{proof}
\begin{corollary}
\label{cor:ql_conv_phi}
    The sequence of approximate solutions for the electric potential $(\phi^\veps)_{\veps>0} \subset L^{\infty}(0, T; L_\M(\Omega))$ converges to $\ln\tilde{\rho}\in L^{\infty}(0, T; L_\M(\Omega))$.
\end{corollary}
\begin{proof}
    The discrete Poisson equation \eqref{eq:dis_poisson} and the convergence of the density from Lemma~\ref{lem:ql_bdd_den} yields the convergence of $e^{\phi^\veps}$ to $\tilde{\rho}$, and as an immediate consequence we note that $(\phi^{\veps})^n_K$ converges to $\ln\tilde{\rho}^n_K$ for all $K\in\M$.    
\end{proof}

\begin{lemma}
\label{lem:ql_bdd_grdPhi}
    Assume that the initial data $(\rho^\veps_0,\uu{u}^\veps_0, \phi^\veps_0)_{\veps>0}$ is well-prepared in the sense of Definition~\ref{def:well_prep_data}. Then one has, for $0 \leqslant n \leqslant N-1$,
    \begin{equation}
    \label{eq:ql_bdd_grdPhi}
        \sum_{i=1}^{d}\sum_{\s\in\E_\intr^{(i)}}\abs{\Ds}(\partial_\E^{(i)}(\phi^\veps)^{n+1})_\s^2 \leqslant C,
    \end{equation}
    where $C > 0$ is a constant independent of $\veps$. Furthermore, for each $n = 1, 2, \dots , N-1$, there exists $\uu{\theta}^{n, n+1} \in \uu{H}_{\E,0}(\Omega)$ such that $\{(\rho^\veps)^n(\bgrd_\E(\phi^\veps)^{n+1})\}_{\veps>0}$ weakly converges to $\uu{\theta}^{n, n+1}$ in $L^2(\Omega)^d$ with
     \begin{equation}
     \label{eq:theta_expr}
         \theta^{n, n+1}_\s=\tilde{\rho}^n_\s(\partial^{(i)}_{\E}\ln\tilde{\rho}^{n+1})_\s
     \end{equation}
     for all $\s\in \E^{(i)}_\intr$, $i = 1, 2, \dots , d$.
\end{lemma}
\begin{proof}
    From the energy estimate \eqref{eq:rm_global_energy_ineq}, we deduce that for each $0 \leqslant n \leqslant N-1$, 
    \begin{equation*}
        \sum_{i=1}^{d}\sum_{\s\in\E_\intr^{(i)}}\abs{\Ds}\Big(\eta-\frac{\abs{(\rho^\veps)_\s^n}^2}{(\rho^\veps)^{n+1}_{\Ds}}\Big)(\partial_\E^{(i)}(\phi^\veps)^{n+1})_\s^2 \leqslant C,
    \end{equation*}
    where $C$ depends on the mesh and the time-step, but not on $\veps$. The uniform boundedness hypothesis \eqref{eq:den_bdd} now readily gives the estimate \eqref{eq:ql_bdd_grdPhi}.

    Next, we turn to the second part of the lemma. The hypothesis \eqref{eq:den_bdd} and the boundedness of the potential gradient $\left\{ (\bgrd_\E (\phi^\veps)^{n+1})\right\}_{\veps > 0} \subset \uu{H}_{\E,0}(\Omega)$ obtained as above implies that $\left\{(\rho^\veps)^n \bgrd_\E (\phi^\veps)^{n+1}\right\}_{\veps>0}\subset\uu{H}_{\E,0}(\Omega)$ is uniformly bounded in \(L^2(\Omega)^d\) for each $n = 1, \dots, N-1$. Therefore, it has a weak limit $\uu{\theta}^{n, n+1} \in L^2(\Omega)^d$ as $\veps \to 0$. Since $\uu{H}_{\E,0}(\Omega)$ is a finite-dimensional subspace of $L^2(\Omega)^d$, it is weakly closed in $L^2(\Omega)^d$. Consequently, $\uu{\theta}^{n, n+1} \in \uu{H}_{\E,0}(\Omega)$ for each $n = 1, \dots, N-1$.

    Now, using \eqref{eq:ql_bdd_grdPhi}, we have
    \begin{align*}        &\sum_{i=1}^{d}\sum_{\s\in\E_\intr^{(i)}}\abs{\Ds}\abs{(\rho^\veps)^n_\s(\partial_\E^{(i)}(\phi^\veps)^{n+1})_\s-\tilde{\rho}^n_\s(\partial^{(i)}_{\E}\ln\tilde{\rho}^{n+1})_\s}\\
    &\leqslant C\Big(\sum_{i=1}^{d}\sum_{\s\in\E_\intr^{(i)}}\abs{\Ds}\abs{(\rho^\veps)^n_\s-\tilde{\rho}^n_\s}^2\Big)^{\frac{1}{2}}+C(\tilde{\rho}^n)\Big(\sum_{i=1}^{d}\sum_{\s\in\E_\intr^{(i)}}\abs{\Ds}\abs{(\partial_\E^{(i)}(\phi^\veps)^{n+1})_\s-(\partial^{(i)}_{\E}\ln\tilde{\rho}^{n+1})_\s}^2\Big)^{\frac{1}{2}}.
    \end{align*}

    Taking $\veps\to 0$ and using the convergence of the density $\rho^\veps$ from Lemma~\ref{lem:ql_bdd_den} and the potential $\phi^\veps$ from Lemma~\ref{cor:ql_conv_phi}, we obtain
    \begin{equation*}
    \sum_{i=1}^{d}\sum_{\s\in\E_\intr^{(i)}}\abs{\Ds}\abs{\theta^{n, n+1}_\s-\tilde{\rho}^n_\s(\partial^{(i)}_{\E}\ln\tilde{\rho}^{n+1})_\s}=0,
    \end{equation*}
    which yields the desired expression \eqref{eq:theta_expr}.
\end{proof}
Now, we are in a position to state the main theorem of this section, based on the boundedness and convergence results obtained above. 
\begin{theorem}
\label{thm:qn-limit_dis}
Let $\left(\veps^{(m)}\right)_{m \in \mathbb{N}}$ be a sequence of positive real numbers tending to zero, and let $\left(\rho^{(m)}, \uu{u}^{(m)}, \phi^{(m)}\right)$ be a corresponding sequence of solutions of the scheme \eqref{epb_dis}. Let us assume that the initial data $\left(\rho^{\veps^{(m)}}_0, \uu{u}^{\veps^{(m)}}_0, \phi^{\veps^{(m)}}_0\right)_{\veps^{(m)}}$ is well-prepared in the sense of Definition~\ref{def:well_prep_data}. Then the sequence $\left(\rho^{(m)}, \uu{u}^{(m)}, \phi^{(m)}\right)$ tends, in any discrete norm, to the function $(\tilde{\rho}, \uu{U}, \ln\tilde{\rho})\in L^{\infty}(0,T;\Lm\times\uu{H}_{\E,0}(\Omega)\times\Lm)$ when $m\to\infty$, and $(\tilde{\rho}, \uu{U})$ satisfies the following semi-implicit scheme for all $0\leqslant n \leqslant N-1$:
\begin{subequations}
\label{eq:isoth_dis}
\begin{equation}
\label{eq:isoth_mass_dis}
    \frac{1}{\dt}\big(\tilde\rho_K^{n+1}-\tilde{\rho}_K^{n}\big)+\frac{1}{\left|K\right|}\sum_{\s\in\Ek}F^{n,n+1}_{\s,K}=0,\; \forall K \in \mcal{M},
\end{equation}
\begin{align}
\label{eq:isoth_mom_dis}
    \frac{1}{\dt}\big(\tilde{\rho}_{\Ds}^{n+1}U_\s^{n+1}-\tilde{\rho}_{\Ds}^n U_\s^{n}\big)+\frac{1}{\left|\Ds\right|}\sum_{\epsilon\in\tilde{\E}(\Ds)}F^{n,n+1}_{\epsilon,\sigma}U_{\epsilon,\mathrm{up}}^{n}+(\partial^{(i)}_{\E}\tilde{\rho}^{n+1})^{*}_{\sigma} = 0, \ 1\leqslant i\leqslant d, \ \forall \s\in\E_\intr^{(i)},
\end{align}
\end{subequations}
where 
\begin{align}
   F^{n,n+1}_{\s, K}&=\abs{\s}\big(\tilde{\rho}^n_\s U^n_{\s, K}-\eta\dt(\partial_{\E}^{(i)}\ln\tilde{\rho}^{n+1})_{\s,K}\big),\\
(\partial^{(i)}_{\E}\tilde{\rho}^{n+1})^{*}_{\sigma}&=(\partial^{(i)}_{\E}\tilde{\rho}^{n+1})_{\sigma}-(\tilde{\rho}_\s^{n+1}-\tilde{\rho}_\s^n)(\partial^{(i)}_{\E}\ln\tilde{\rho}^{n+1})_{\sigma}-\dt(\partial^{(i)}_{\E}(\divm U^n))_{\s}, \label{eq:rho_grad*}
\end{align}
with $\eta>\frac{5(\tilde{\rho}_\s^n)^2}{4\tilde{\rho}_{\Ds}^n}$, $\tilde{\rho}_\s=\tilde{\rho}_{KL}$, $\s=K|L$, as given by Lemma~\ref{lem:rho_sig}.
\end{theorem}
\begin{proof}
    From Lemmas \ref{lem:ql_bdd_den} and \ref{lem:ql_bdd_grdPhi}, it follows that $\left(\rho^{(m)}, \phi^{(m)}\right)_{m\in\mbb{N}}$ converges to $(\tilde{\rho}, \ln\tilde{\rho})\in L^{\infty}(0,T;\Lm\times\Lm)$ as $m \to \infty$. Also, by Corollary~\ref{cor:bdd_vel}, $(\uu{u}^{(m)})_{m\in\mbb{N}}$ is bounded in any norm, which implies that there exist a subsequence of $(\uu{u}^{(m)})_{m\in\mbb{N}}$ which converges, in any discrete norm, to some $\uu{U}\in L^{\infty}(0,T;\uu{H}_{\E, 0}(\Omega))$ as $m \to \infty$. Passing to the limit cell by cell in the scheme \eqref{epb_dis} easily shows that $(\tilde{\rho}, \uu{U})$ is a solution of the stabilised semi-implicit scheme \eqref{eq:isoth_dis}.
\end{proof}
Next, we establish a Lax-Wendroff-type weak consistency of the scheme \eqref{eq:isoth_dis}. Precisely, we show that the scheme \eqref{eq:isoth_dis} is consistent with the weak formulation of the ICE system \eqref{eq:iso_euler}, when the mesh parameters tend to zero. As in Theorem~\ref{thm:weak_cons}, it is customary to make the requisite boundedness and convergence assumptions on the solutions $(\tilde{\rho},\uu{U})$.

\begin{theorem}
\label{thm:iso_weak_cons}
Assume that $\big(\mcal{T}^{(m)},\delta t^{(m)}\big)_{m\in\mbb{N}}$ is a sequence of discretisations such that both $\lim_{m\rightarrow \infty}\delta t^{(m)}$ and $\lim_{m\rightarrow \infty}h_{\mcal{T}^{(m)}}$ are $0$. Let $\big(\tilde{\rho}^{(m)},\uu{U}^{(m)}\big)_{m\in\mbb{N}}$ be the corresponding sequence of discrete solutions of the scheme \eqref{eq:isoth_dis} with respect to an initial data $(\tilde{\rho}_0,\uu{U}_0)\in L^\infty(\Omega)^{1+d}$. We assume that $(\tilde{\rho}^{(m)},\uu{U}^{(m)})_{m\in\mbb{N}}$ satisfies the following.
\begin{enumerate}[label=(\roman*)]
\item $\big(\tilde{\rho}^{(m)},\uu{U}^{(m)}\big)_{m\in\mbb{N}}$ is uniformly bounded in $L^\infty(Q)^{1+d}$, i.e.\ 
\begin{align}
\ubar{C}<(\tilde{\rho}^{(m)})^n_K &\leqslant \bar{C}, \ \forall K\in\mcal{M}^{(m)}, \ 0\leqslant n\leqslant N^{(m)}, \ \forall m\in\mbb{N}, \label{eq:isoth_dens_abs_bound} \\
|(U^{(m)})^n_\sigma| &\leqslant C, \ \forall \sigma\in\mcal{E}^{(m)}, \ 0\leqslant n\leqslant N^{(m)}, \ \forall m\in\mbb{N}. \label{eq:isoth_u_abs_bound}
\end{align}
\item $(\tilde{\rho}^{(m)})_{m\in\mbb{N}}$ is uniformly bounded in $L^{1}([0,T);BV(\Omega))$, i.e.\
\begin{equation}
\norm{\tilde{\rho}^{(m)}}_{\mcal{T},\mcal{M},x,BV}=\sum_{n=0}^{N^{(m)}} \delta t^{(m)}\sum_{\sigma=K|L\in \E_\intr^{(m)}}\abs{(\tilde{\rho}^{(m)})^n_L-(\tilde{\rho}^{(m)})^n_K} \leq C,\; \forall m\in\mbb{N}, \label{eq:bv_dens_space}
\end{equation}
where $\ubar{C}, \bar{C}, C>0$ are constants independent of the discretisations. 
\item $\big(\tilde{\rho}^{(m)},\uu{U}^{(m)}\big)_{m\in\mbb{N}}$ converges to $(\tilde{\rho},\uu{U})\in L^\infty(0, T; L^\infty(\Omega)\times L^\infty(\Omega)^d)$ in $L^r(Q_T)^{1+d}$ for all $1\leqslant r<\infty$.
\end{enumerate}
Assume furthermore that the sequence of discretisations $\big(\mcal{T}^{(m)},\delta t^{(m)}\big)_{m\in\mbb{N}}$ satisfies the mesh regularity conditions:
\begin{equation}
\label{eq:isoth_CFL_meshpar}
\frac{\delta t^{(m)}}{\min_{K\in\mcal{M}^{(m)}}\abs{K}}\leqslant\theta,\ \max_{K \in \mcal{M}^{(m)}} \frac{\diam(K)^2}{\abs{K}}\leqslant\theta,\ \max_{K \in \M^{(m)}} \max_{\sigma \in \E^{(m)}(K)} \frac{|D_\sigma|}{|K|} \leq \theta
 \;\forall m\in\mbb{N},
\end{equation}
where $\theta>0$ is independent of the discretisations. Then $(\tilde{\rho},\uu{U})$ is a weak solution of \eqref{eq:iso_euler}.
\end{theorem}
\begin{proof}
We follow the same approach as in the proof of Theorem~\ref{thm:weak_cons} and hence omit most of the details. However, the second term in \eqref{eq:rho_grad*} arising from the pressure gradient in the momentum update \eqref{eq:isoth_mom_dis} needs attention. Proceeding as in Theorem~\ref{thm:weak_cons}, the above term can be estimated to get
\begin{equation}
R^{(m)}=C_{\uu{\psi}}\sum_{n=0}^{N^{(m)}-1}\dt^n\sum_{\s\in\E^{(m)}_\intr}\abs{\Ds}\abs{(\tilde{\rho}_\s^{n+1}-\tilde{\rho}_\s^n)(\partial^{(i)}_{\E}\ln\tilde{\rho}^{n+1})_{\sigma}},
\end{equation}
where $C_{\uu{\psi}}$ is a constant depending only on the test function $\uu{\psi} \in C^{\infty}_{c}([0, T ) \times\bar{\Omega})^d$.
Thanks to the uniform boundedness assumptions \eqref{eq:isoth_dens_abs_bound} and \eqref{eq:bv_dens_space}, we can estimate the above term to get
\begin{equation}
R^{(m)}\leqslant C(\uu{\psi}, \bar{C},\ubar{C})\sum_{n=0}^{N^{(m)}-1}\dt^n\sum_{\s\in\E^{(m)}_\intr }\abs{\Ds}\abs{(\tilde{\rho}_\s^{n+1}-\tilde{\rho}_\s^n)}.
\end{equation}
Using analogous techniques as in \cite[Lemma A.6]{HLN+23} regarding the convergence of discrete time translates, we see that $R^{(m)}$ tends to $0$ as $m\to\infty$, using the $L^1$ convergence of $\tilde{\rho}^{(m)}$ and the regularity of the sequence of meshes.
\end{proof}
\begin{remark}  
    Thanks to the choice $\rho_\s=\rho_{KL}$ in Lemma~\ref{lem:rho_sig}, we obtain a conservative approximation of the pressure gradient in the limiting system. Use of any other choice for the interface density, e.g.\ a centred choice, in the source term will not give a conservative approximation and consequently the resulting scheme can lead to wrong shock speeds when there are shock discontinuities.
\end{remark}

\section{Numerical Results}
\label{sec:num_res}
In this section, we report the results of our numerical experiments using the proposed scheme \eqref{eq:dis_cons_mas}-\eqref{eq:dis_poisson} in order to validate the claims made in the preceding sections. The numerical implementation of the scheme \eqref{eq:dis_cons_mas}-\eqref{eq:dis_poisson} is carried out as follows. First, we eliminate the density term in \eqref{eq:dis_cons_mas} using \eqref{eq:dis_poisson} to obtain a non-linear elliptic problem for $\phi^{n+1}$. The presence of non-linear terms in $\phi^{n+1}$ necessitates a Newton iteration in the numerical solution. Once $\phi^{n+1}$ is evaluated, the density $\rho^{n+1}$ and velocity $\uu{u}^{n+1}$ are calculated explicitly using \eqref{eq:dis_cons_mas} and \eqref{eq:dis_cons_mom}, respectively. 

It should be emphasised that the time-step condition \eqref{eq:cfl} required by the stability analysis in Section~\ref{sec:stable_scheme} is implicit, contains flux terms, and hence it is difficult to implement. In order to overcome this challenge, we derive a sufficient condition that is easier to enforce. 

\begin{proposition}
Suppose $\dt>0$ is such that for each $\s\in\E_\intr^{(i)},\; i\in\{1,\dots,d\},\; \s = K|L$, the following holds: 
\begin{equation}
\label{eq:time-step_suff}
    \dt\max\Bigg\{\frac{\abs{\D K}}{\abs{K}},\frac{\abs{\D L}}{\abs{L}}\Bigg\}\Bigg(\abs{u^n_\s} + \sqrt{\tilde{\eta}\left|\phi^{n+1}_L - \phi^{n+1}_K\right|}\Bigg)\leq \frac{1}{5}\mu^{n}_{K,L},
\end{equation}
where $\abs{\D K}=\sum_{\s\in\E(K)}\abs{\s}$, \, $\tilde{\eta}\max\big\{\frac{\abs{\D K}}{\abs{K}},\frac{\abs{\D L}}{\abs{L}}\big\}=\eta\frac{\abs{\s}}{\abs{\Ds}}$ and $\mu^{n}_{K,L} = \displaystyle\frac{\min\{\rho^n_K,\rho^n_L\}}{\rho^{n}_\s}$. Then $\dt$ satisfies the time-step restriction \eqref{eq:suftime}.
\end{proposition}
\begin{proof}
    The proof proceeds along the lines of \cite[Proposition 3.2]{CDV17}, where a similar result derived for an explicit scheme; see also \cite{AGK24,DVB17,DVB20} for analogous treatments.
\end{proof}
\begin{remark}
    Note that the above sufficient condition \eqref{eq:time-step_suff} is again implicit, and in our calculations we implement the same in an explicit fashion. 
\end{remark}

\begin{remark}
    The positivity of the density can be obtained using \eqref{eq:time-step_suff} with straightforward calculations.  
\end{remark}

\begin{remark}
Throughout this subsection we consider the following classical scheme on a collocated MAC grid \cite{HL+20} as a reference scheme to compare the solutions:
\begin{subequations}
\label{eq:cl_dis_epb}
\begin{gather}
\frac{1}{\dt}(\rho_{K}^{n+1}-\rho_{K}^{n})+\frac{1}{\abs{K}}\sum_{\s\in\E(K)}F_{\s,K}^{n}=0, \ \forall K\in \M, \\
\frac{1}{\dt}(\rho_{K}^{n+1}\uu{u}_{K}^{n+1}-\rho_{K}^{n}\uu{u}_{K}^{n})+\frac{1}{\abs{K}}\sum_{\s\in\E(K)}F_{\s, K}^{n}\uu{u}_{\s}^{n}=-\rho_{K}^{n+1}(\grd_{\M}\phi^{n+1})_K, \ \forall K\in\M,  \label{eq:cl_dis_cons_mom}\\
-\veps^2(\Delta_{\M}\phi^{n+1})_{K}=\rho_{K}^{n+1}-e^{\phi_K^{n+1}}, \ \forall K\in \M,  \label{eq:cl_dis_poisson}
\end{gather}
\end{subequations}
where $F_{\s,K}^{n}=\abs{\s}F_{\s}(\rho^n,\uu{u}^n)\cdot\uu{\nu}_{\s,K}$ with $F_\s$ being chosen as the Rusanov flux; see \cite{CDV07} for more details.
\end{remark}

%We test our scheme by several 1D and 2D test cases, and with different values of the Debye length $\veps$ ranging from dispersive to quasineutral regimes.

\subsection{Five Branch Problem}
\label{sec:1d_5-branch}
In this case study we consider the so-called five-branch test problem from \cite{DLS+12}. The aim of experiment is to explore the capabilities of the numerical scheme in both the dispersive and the quasineutral regimes. The initial condition consists of a Gaussian profile for the density and a sinusoidal profile for the velocity, i.e.\   
\begin{equation}
    \rho(0, x) = \frac{1}{\pi} e^{-(x-\pi)^2},\; u(0, x) = \sin^3(x).
\end{equation}
The computational domain $[0,2\pi]$ is divided into $100$ mesh points. The boundary conditions are homogeneous Neumann for the density and the velocity and periodic for the potential. In Figure~\ref{fig:five-branch_eps1}, we depict the density, the velocity and the potential profiles at $T=1$ and the corresponding initial values corresponding to the choice $\veps=1$. Analogously, we present the same profiles corresponding to $\veps=10^{-2}$ in Figure~\ref{fig:five-branch_eps-01}. In both figures, we make a comparison of the computed solutions against a reference solution obtained using the explicit Rusanov scheme on a fine mesh of $500$ points for the same values of $\veps$. We observe that when $\veps=1$, the density profiles exhibit peaks which ultimately leads to singularities. In a semi-classical setting, these singularities in the density profile lead to multi-valued solutions; see \cite{LW07} for more details. On the other hand, the density profile remains smooth and it spreads out in the entire computational domain at $T=1$ when $\veps=10^{-2}$. Our results are well in agreement with those reported in \cite{DLS+12}. 

\begin{figure}[htbp]
  \centering
    \includegraphics[height=0.185\textheight]{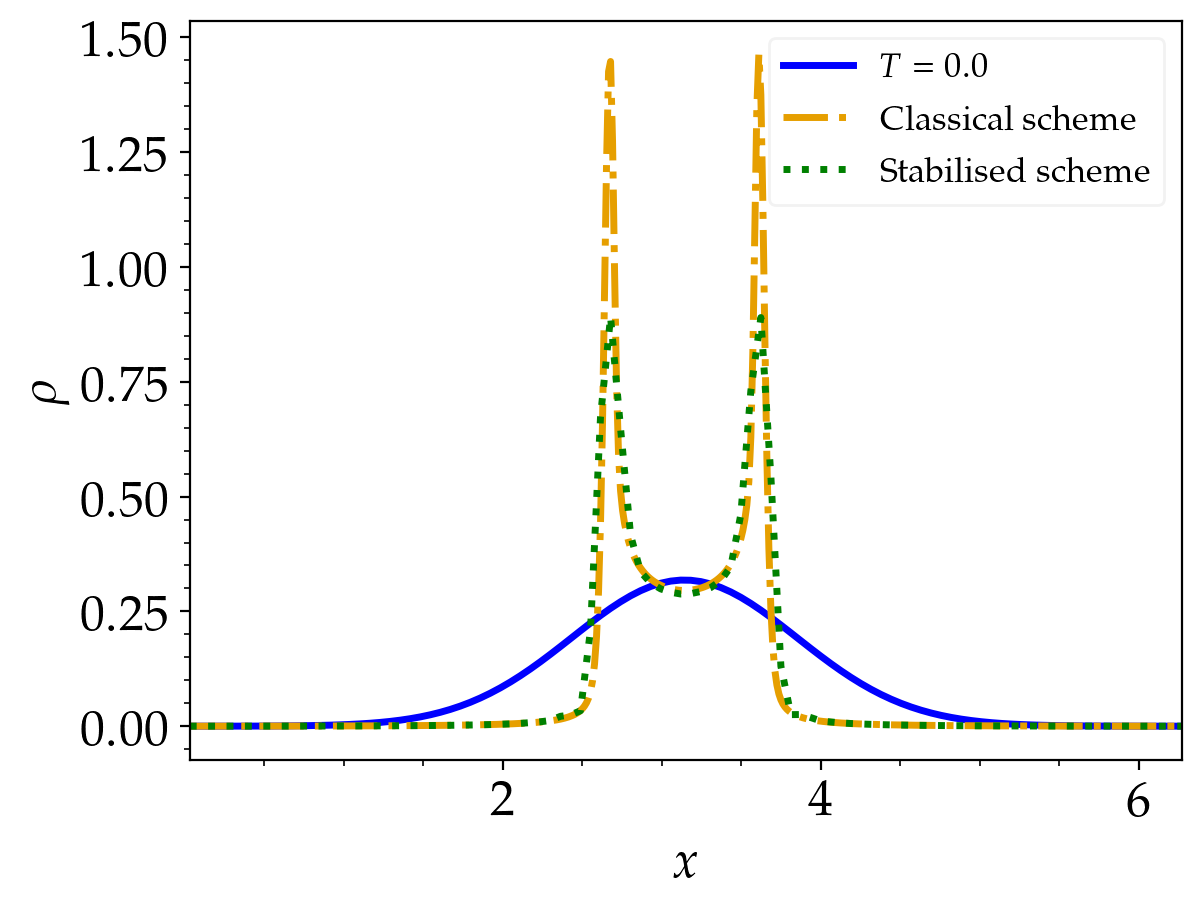}
    \includegraphics[height=0.185\textheight]{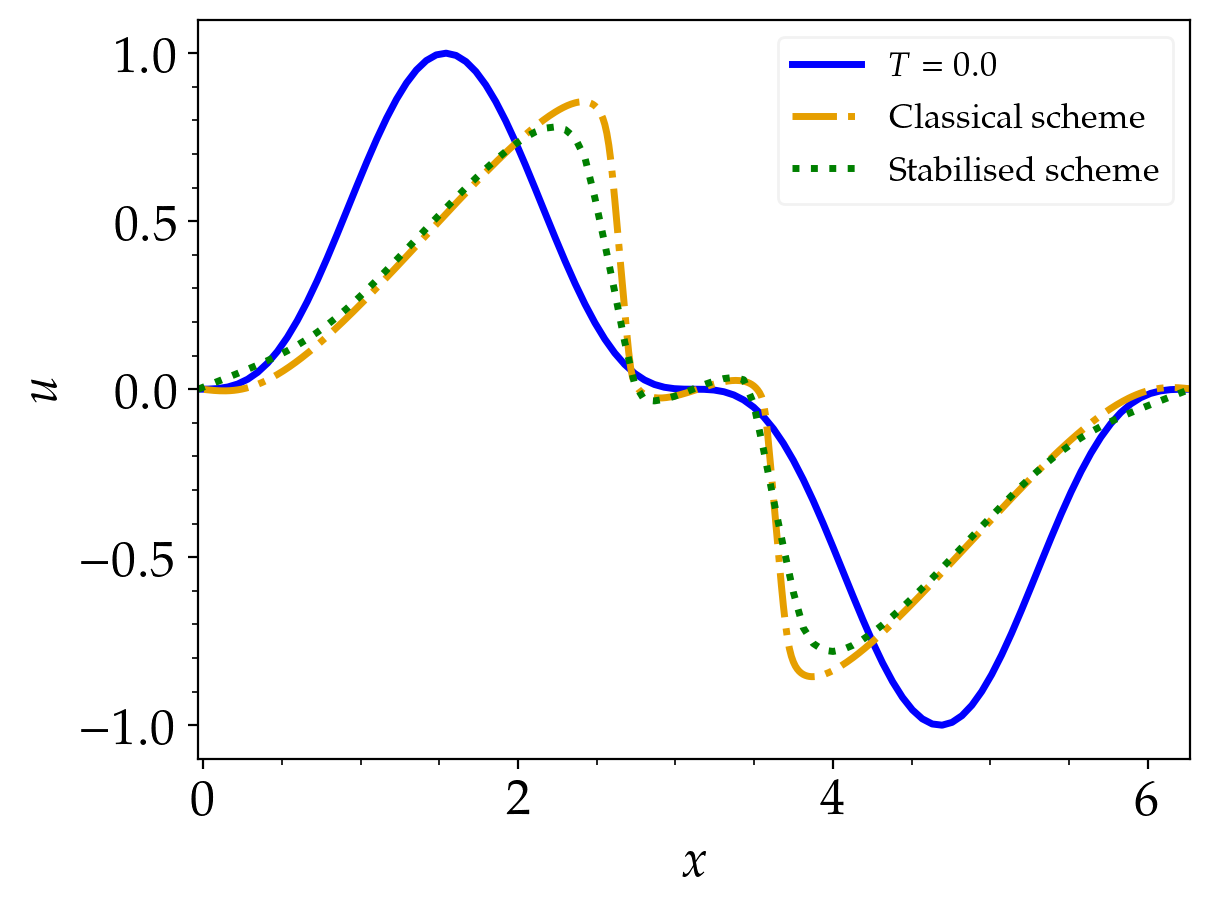}
    \includegraphics[height=0.185\textheight]{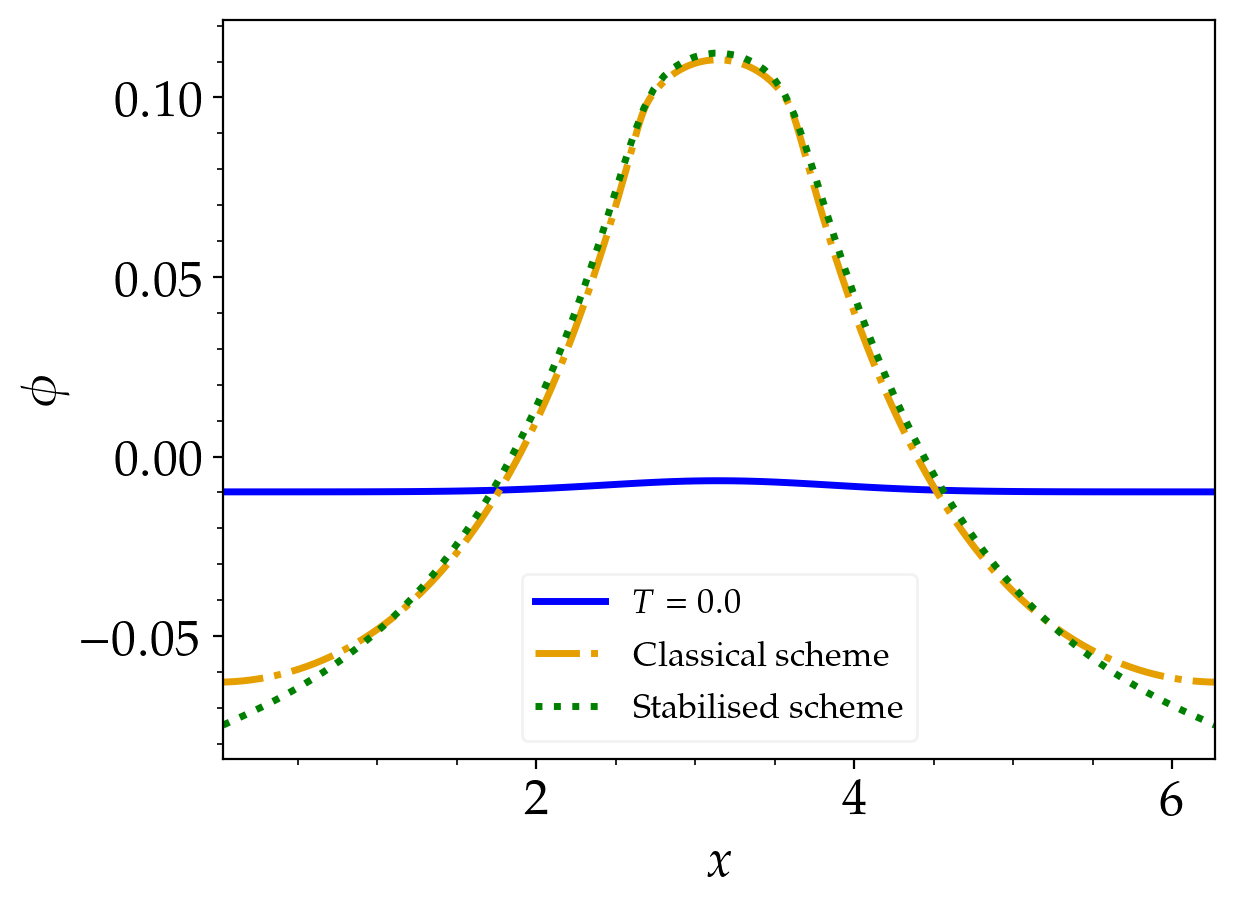}
    \caption{Ion density, velocity, potential plots for the five-branch test case at $T=1$ for $\veps=1$.}  
    \label{fig:five-branch_eps1}
\end{figure}
 
\begin{figure}[htbp]
  \centering
  \includegraphics[height=0.191\textheight]{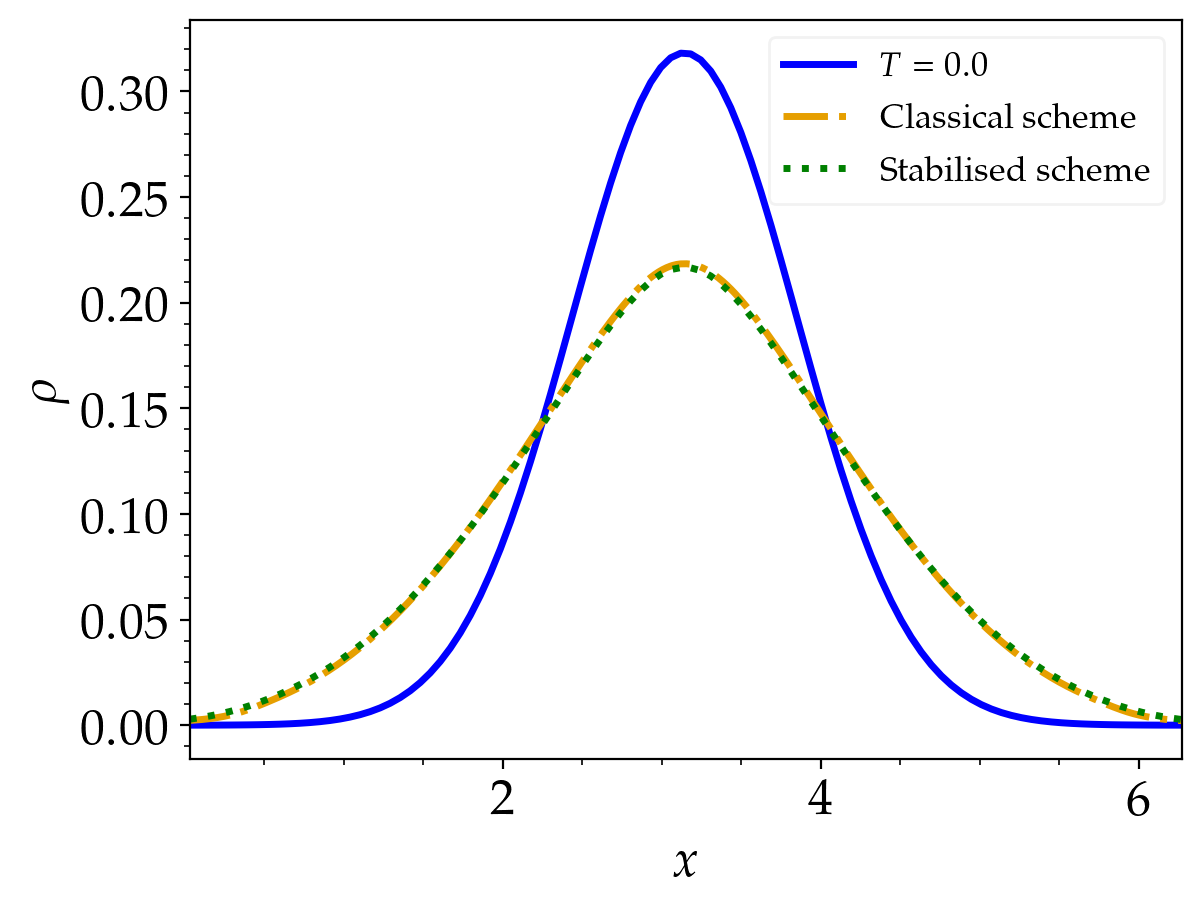}
    \includegraphics[height=0.191\textheight]{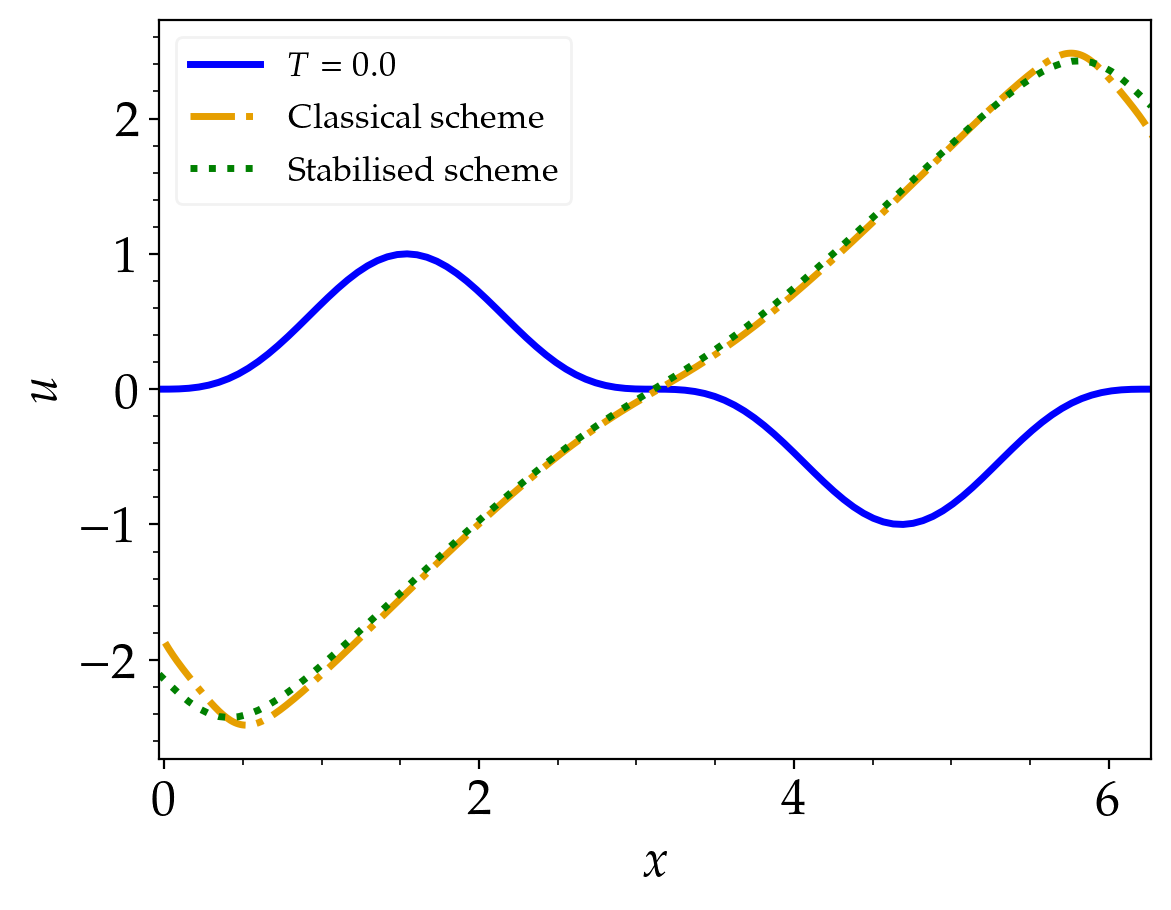}
    \includegraphics[height=0.191\textheight]{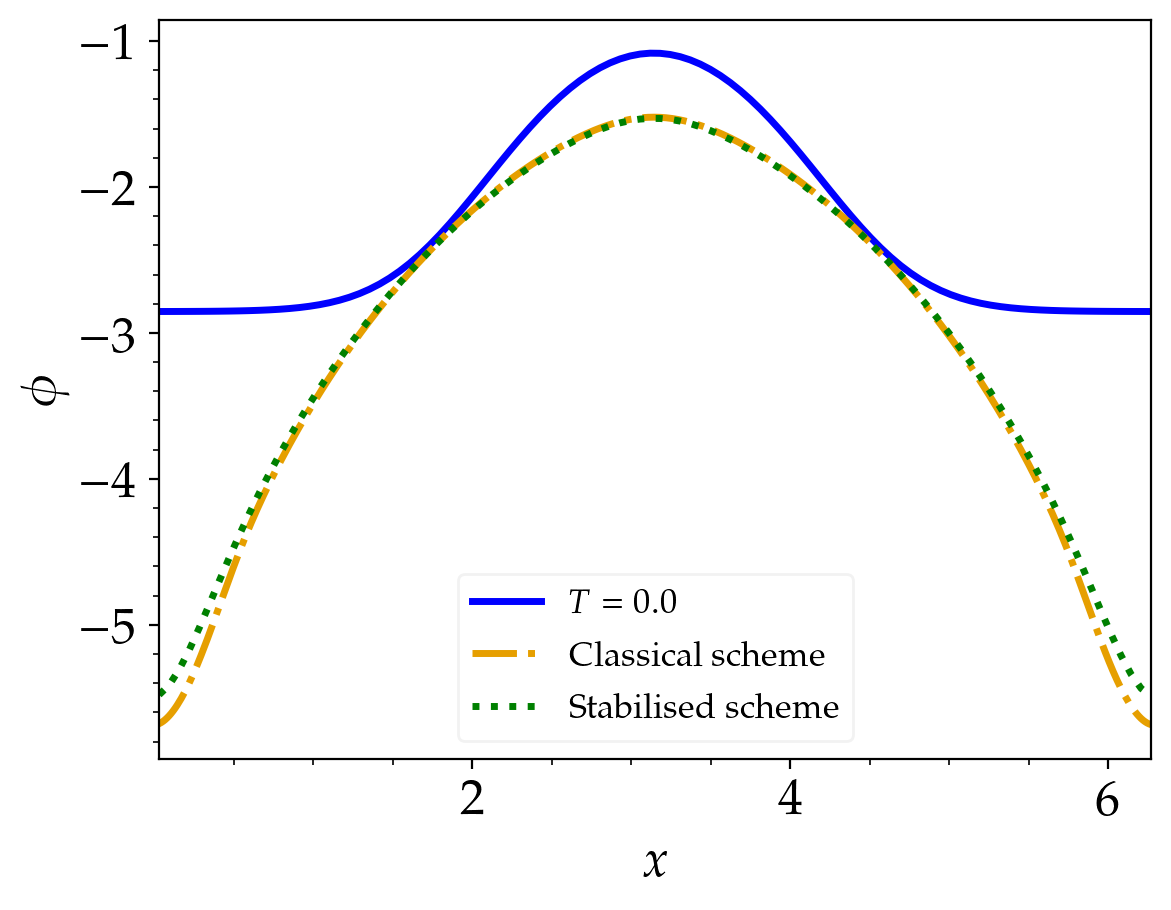}
    \caption{Ion density, velocity, potential plots for the five-branch test case at $T=1$ for $\veps=10^{-2}$.}  
    \label{fig:five-branch_eps-01}
\end{figure}

\subsection{1D Shock Tube Problem}
\label{sec:1d_shock-tube}
We consider the following shock tube problem from \cite{DLS+12} which involves two outgoing shock waves. The initial density is a constant, while the initial velocity has a jump at $x = 0$. The computational domain is $[-0.2, 0.2]$. The Debye length $\veps$ is varied from $10^{-2}$ to $10^{-4}$. The initial conditions read
\begin{equation}
    \rho(0, x) = 1,\; u(0, x) = \begin{cases}
       -1, & \text{if } x<0, \\
        +1, & \text{if } x\geqslant 0.
    \end{cases}
\end{equation}
The problem is supplemented with extrapolation boundary conditions for the density and the velocity and periodic boundary condition for the potential. The domain is discretised using $100$ (resp.\ $1000$) grid points corresponding to $\veps=10^{-2}$ (resp.\ $\veps=10^{-4}$) so that the grid remains unresolved with respect to $\veps$.

\begin{figure}[htbp]
    \centering
    \includegraphics[height=0.191\textheight]{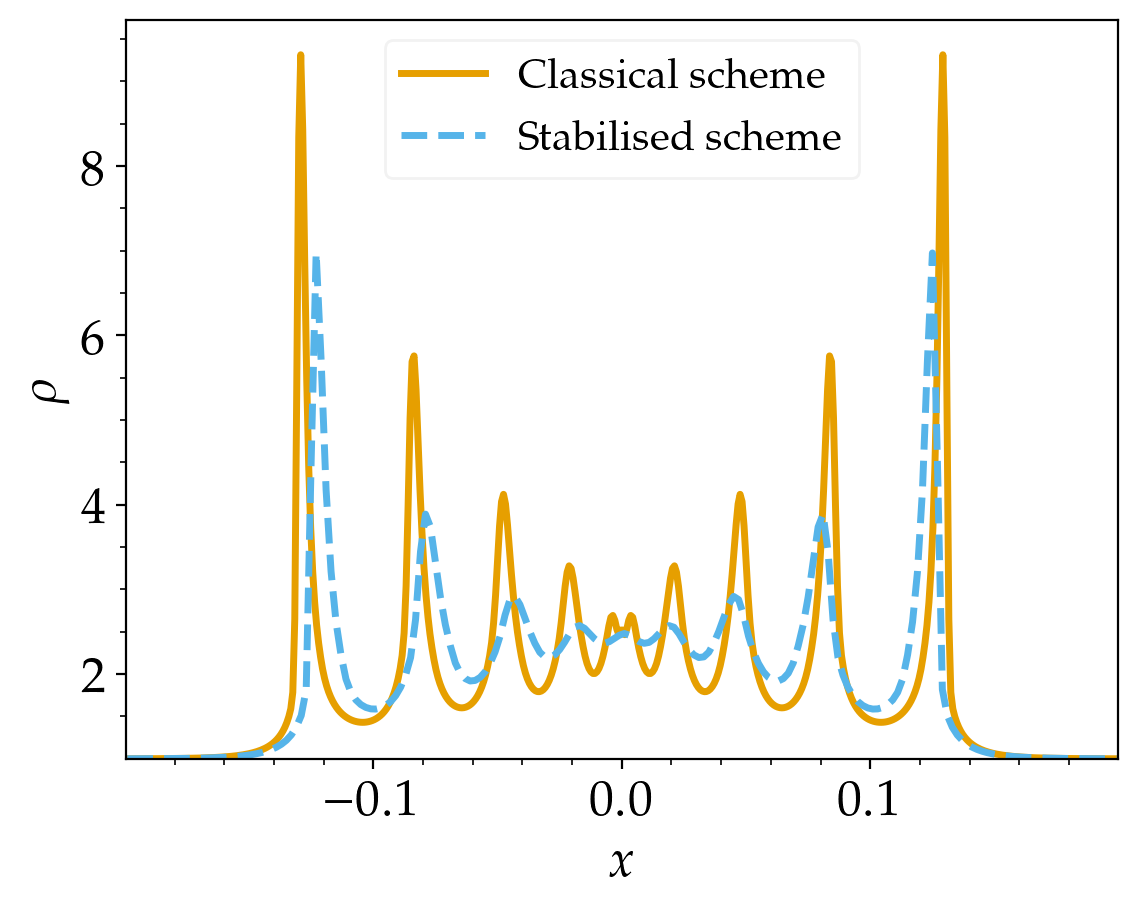}
    \includegraphics[height=0.191\textheight]{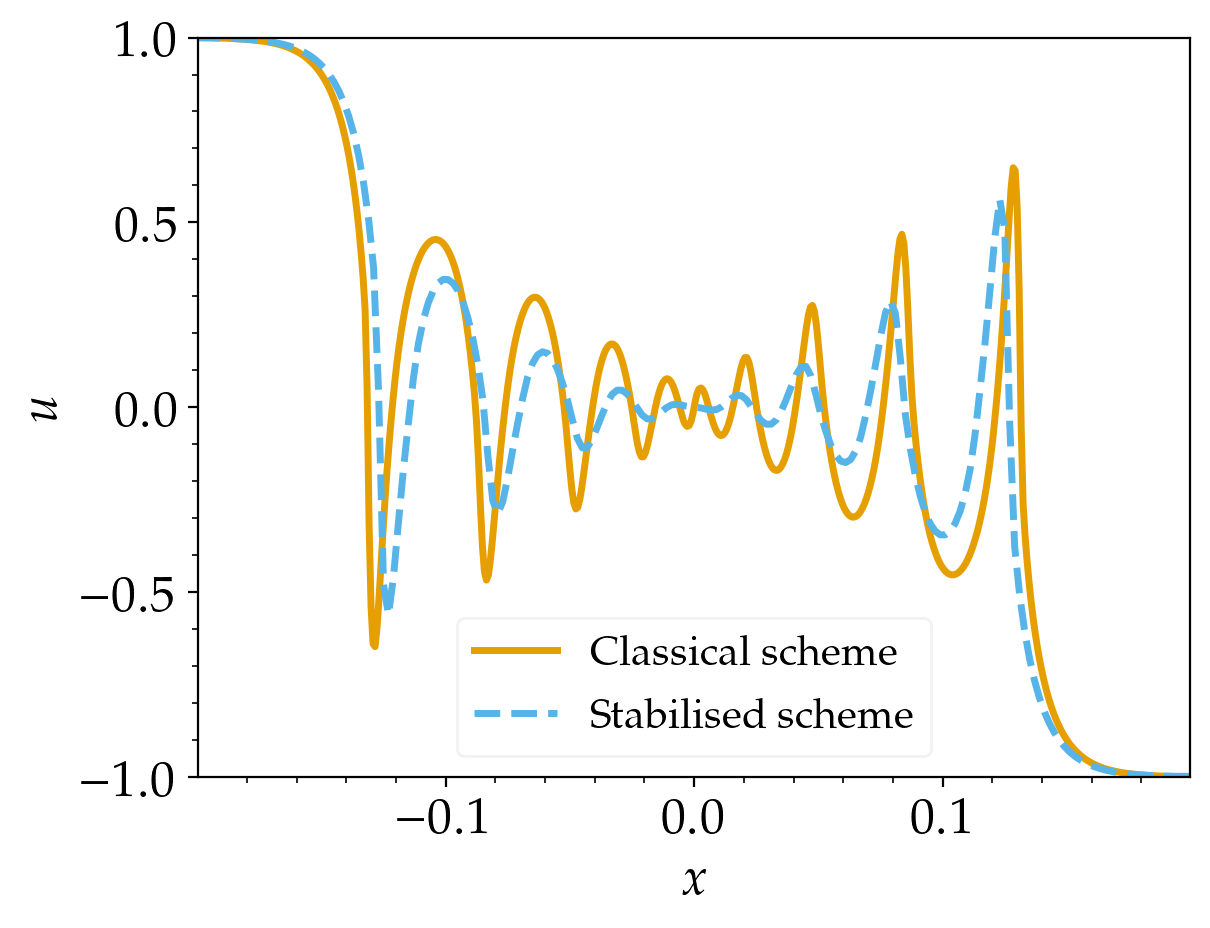}
    \includegraphics[height=0.191\textheight]{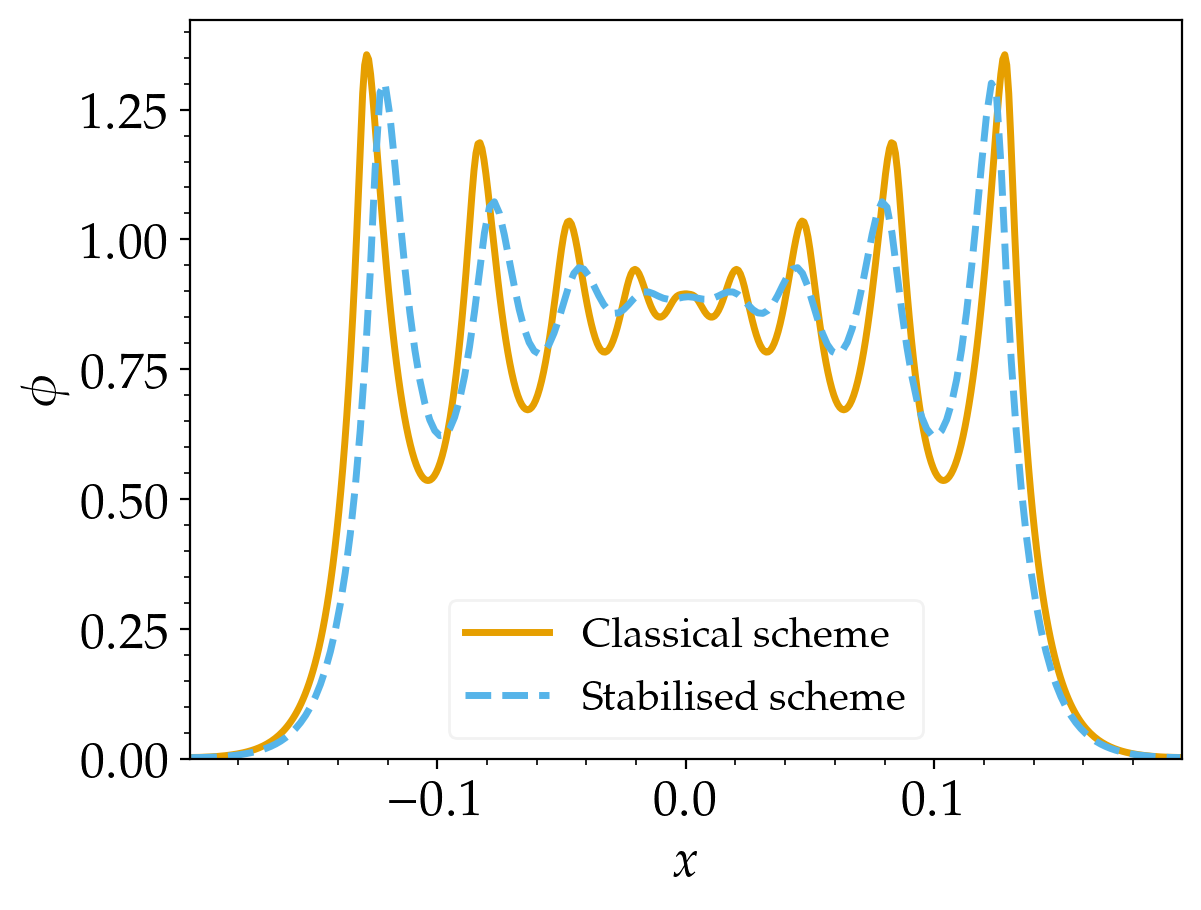}
    \caption{ Density, velocity and potential profiles for the shock tube test case at $T=0.2$ for $\veps=10^{-2}$.} 
    \label{fig:1d_shock-tube_eps-02}
\end{figure}

In Figure~\ref{fig:1d_shock-tube_eps-02} we present the density, the velocity and the potential profiles corresponding to $\veps=10^{-2}$ against a reference solution obtained using the explicit Rusanov scheme with a resolved mesh containing $500$ points. The chosen value $\veps=10^{-2}$ is not small enough for the quasineutrality assumption to be valid. The numerical results contain oscillations and the consequent singularities which are nothing but signatures of the dispersive EPB system with moderately small $\veps$. The proposed scheme effectively captures these dispersive oscillations, despite being on a coarse mesh. Furthermore, both the classical explicit scheme and the present scheme show roughly same amount of numerical dissipation and the results are comparable. We plot the density, the velocity and the potential profiles corresponding to $\veps=10^{-4}$ against a reference solution obtained using the explicit Rusanov scheme with a mesh resolution of $2000$ points in Figure~\ref{fig:1d_shock-tube_eps-04}. The value $\veps=10^{-4}$ is quite small to observe the hydrodynamic quasineutral regime. We can clearly notice the formation and propagation of two shocks in the solutions from Figure~\ref{fig:1d_shock-tube_eps-04}. As in \cite{DLS+12}, we notice oscillations near the discontinuities. Nevertheless, the energy stable scheme's performance is far superior to that of the classical scheme. In order to corroborate further, we present the zoomed portions of the solutions near the discontinuities in Figure~\ref{fig:zoom_1d_shock-tube_eps-04}. We note that there are almost no oscillations in the density and the potential profiles for the present scheme, whereas the oscillations present in the velocity are far less compared to the classical scheme. 

\begin{figure}[htbp]
    \centering
    \includegraphics[height=0.19\textheight]{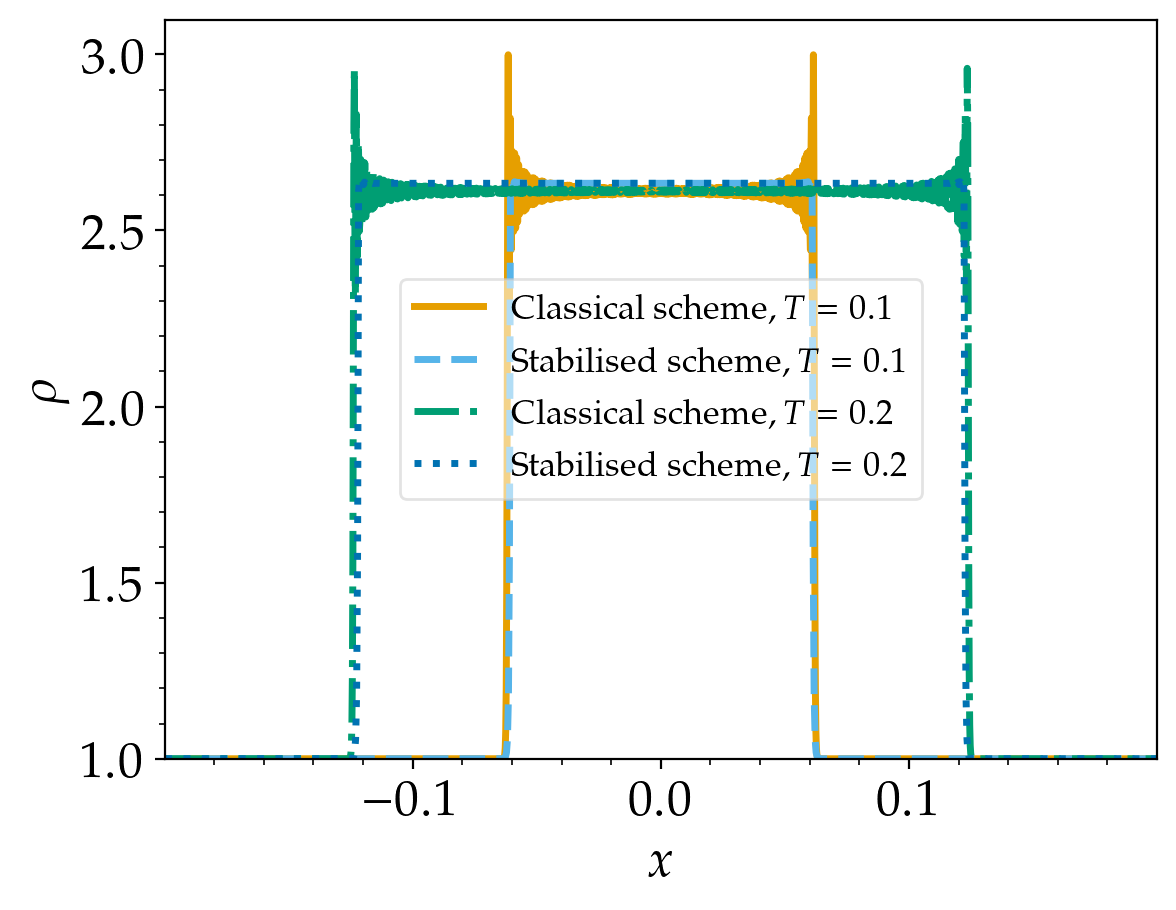}
    \includegraphics[height=0.19\textheight]{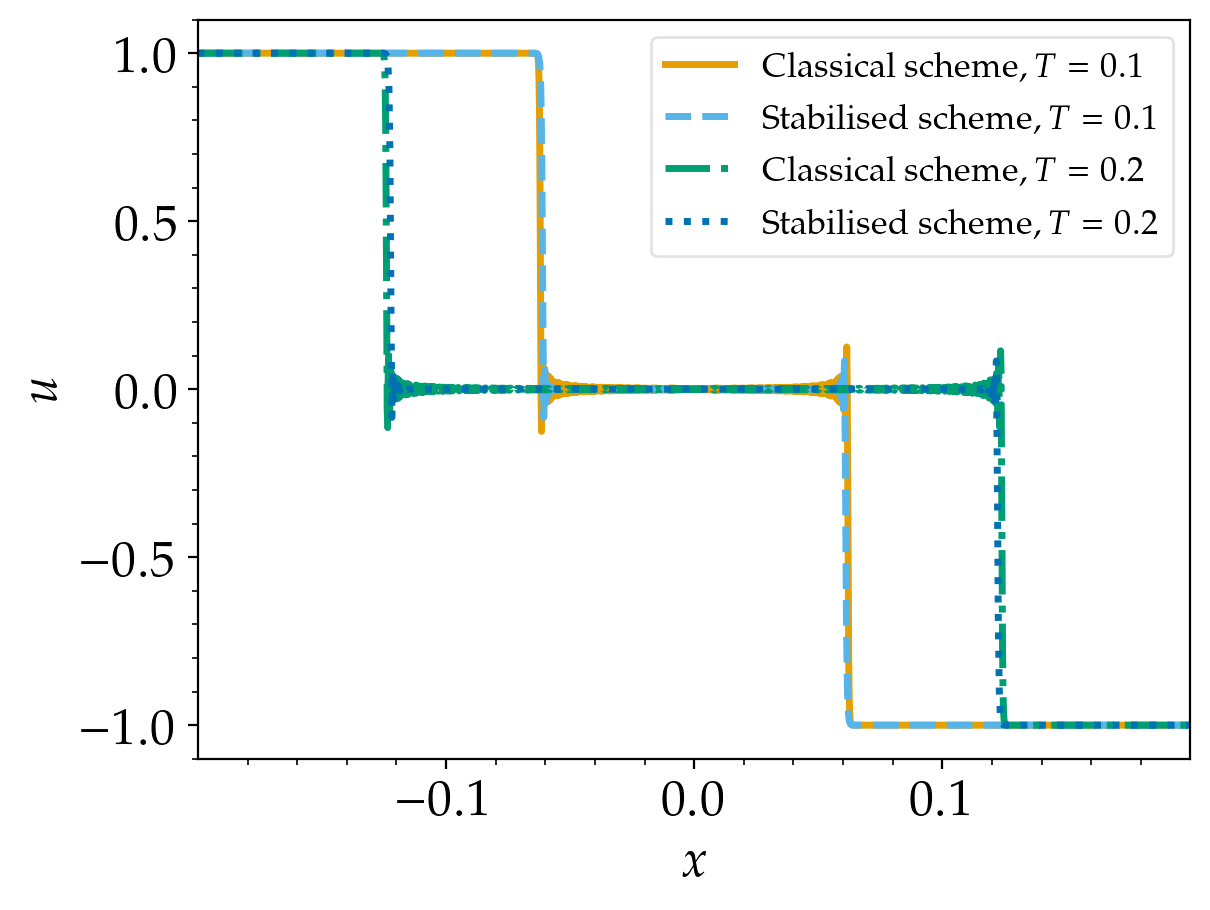}
    \includegraphics[height=0.19\textheight]{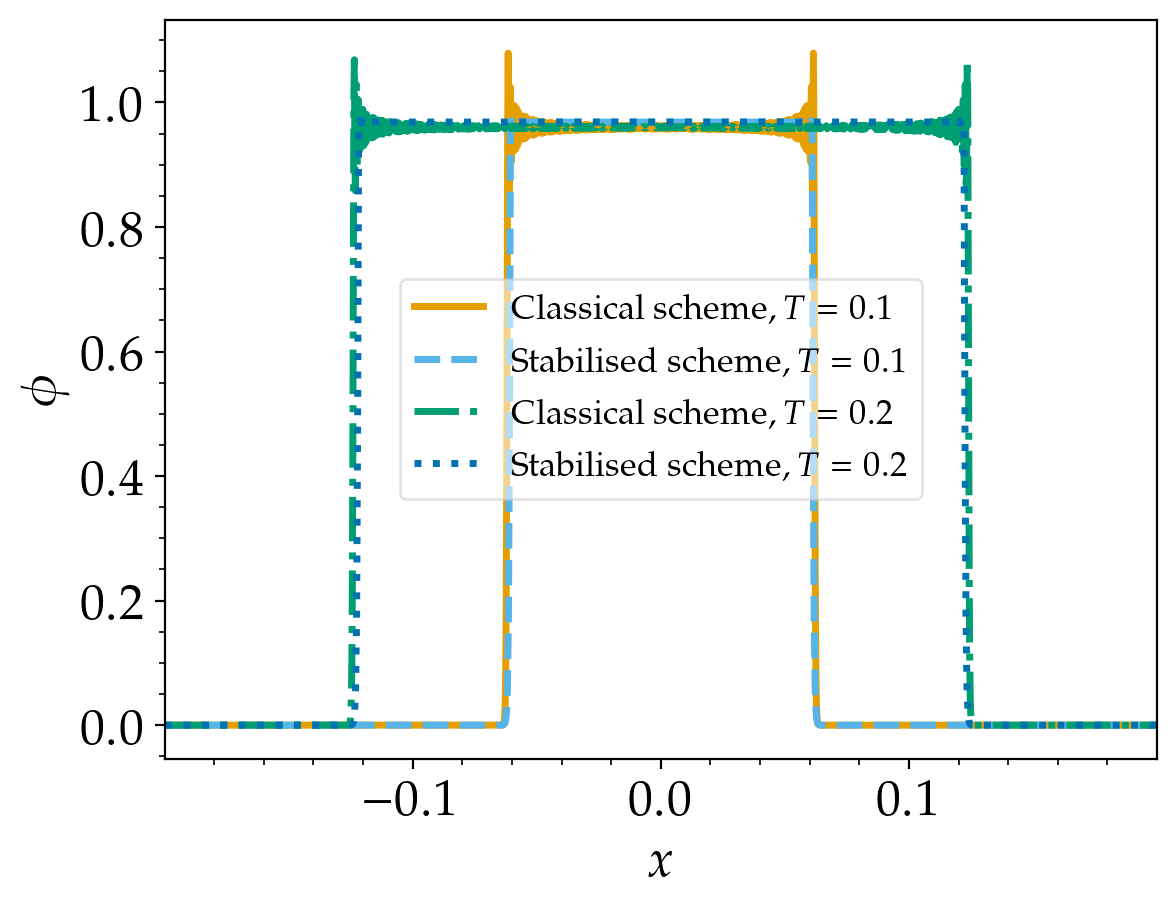}
    \caption{The density, velocity and potential profiles for the shock tube test case for different times corresponding to $\veps=10^{-4}$.} 
    \label{fig:1d_shock-tube_eps-04}
\end{figure}

\begin{figure}[htbp]
    \centering
    \includegraphics[height=0.19\textheight]{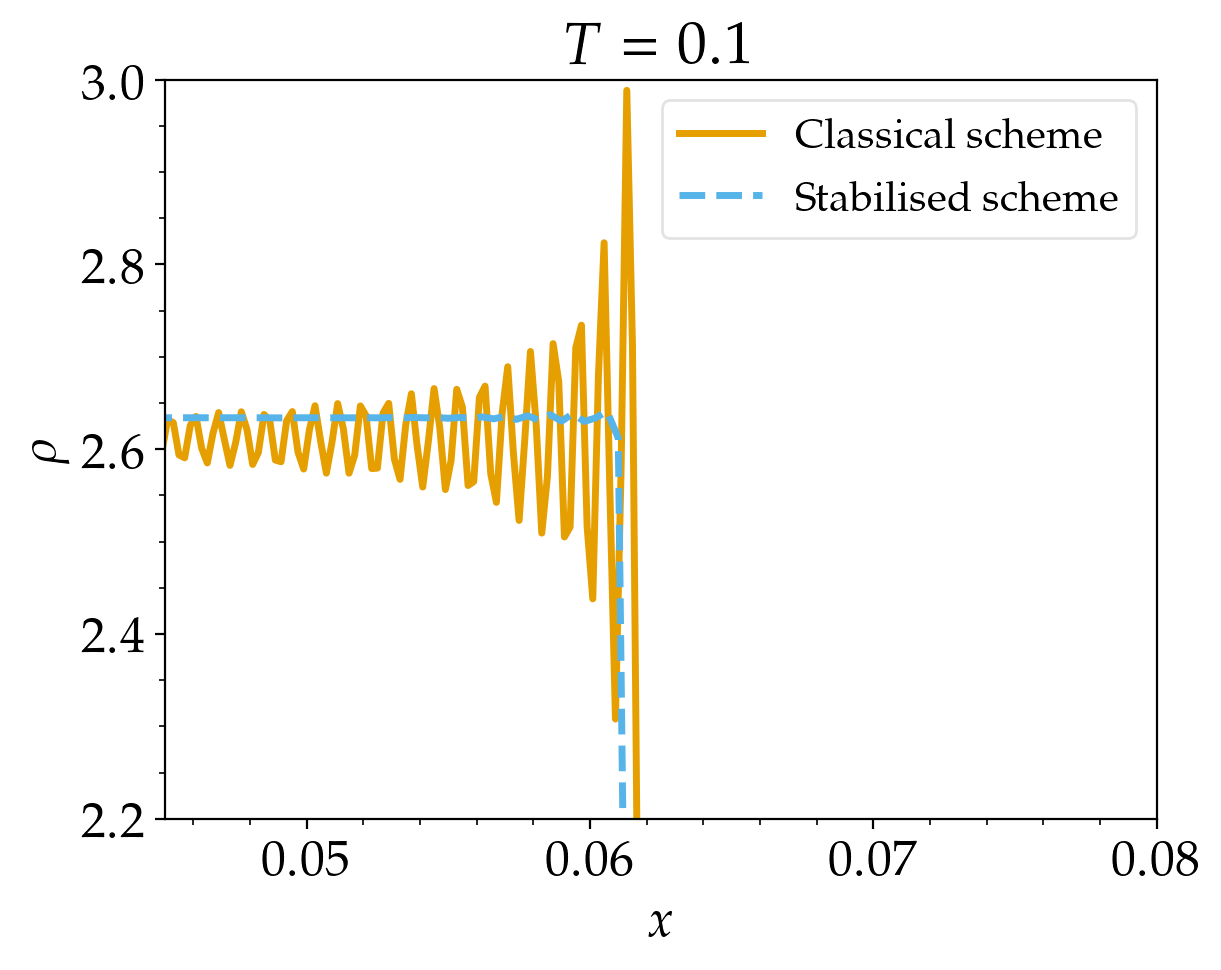}
    \includegraphics[height=0.19\textheight]{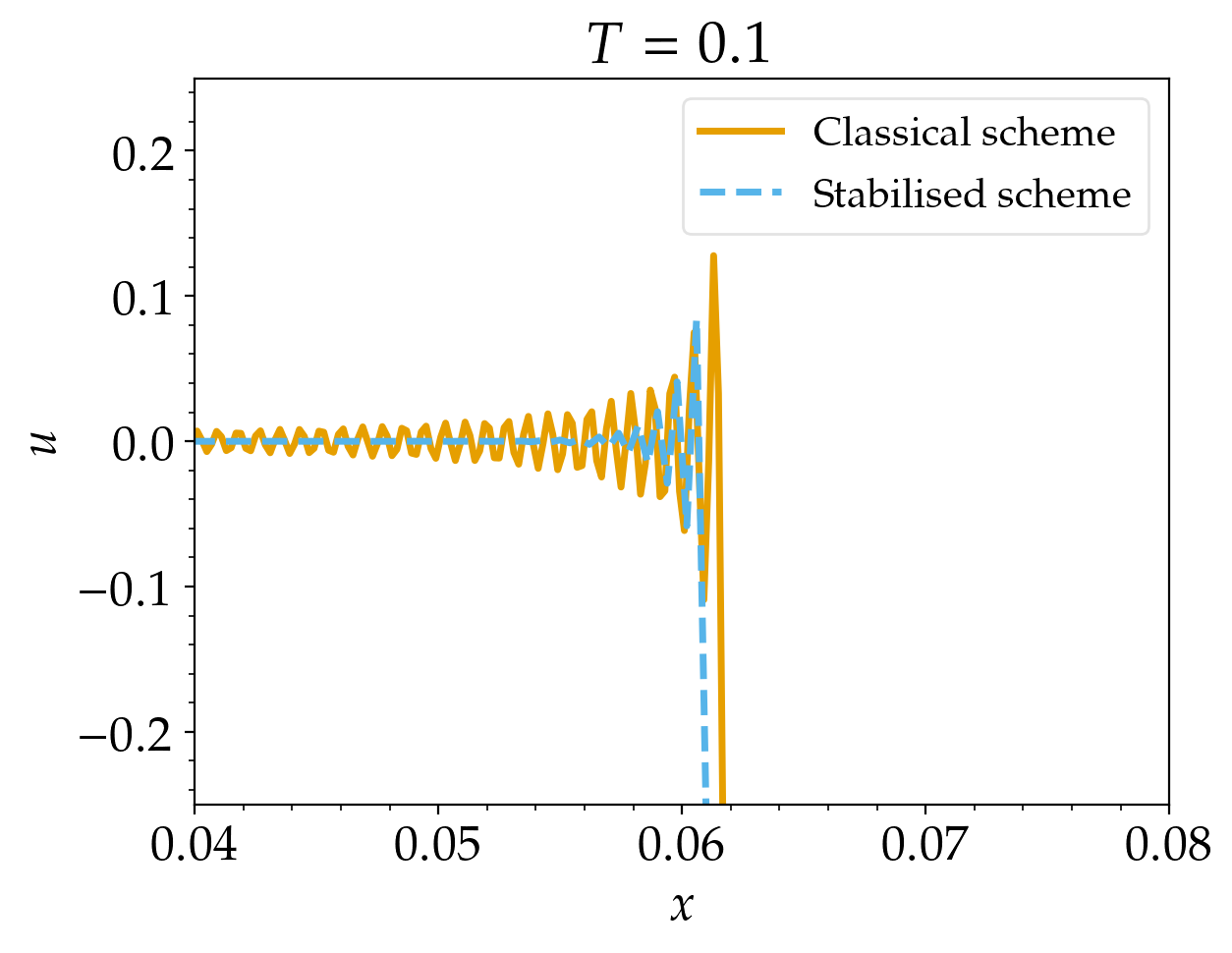}
    \includegraphics[height=0.19\textheight]{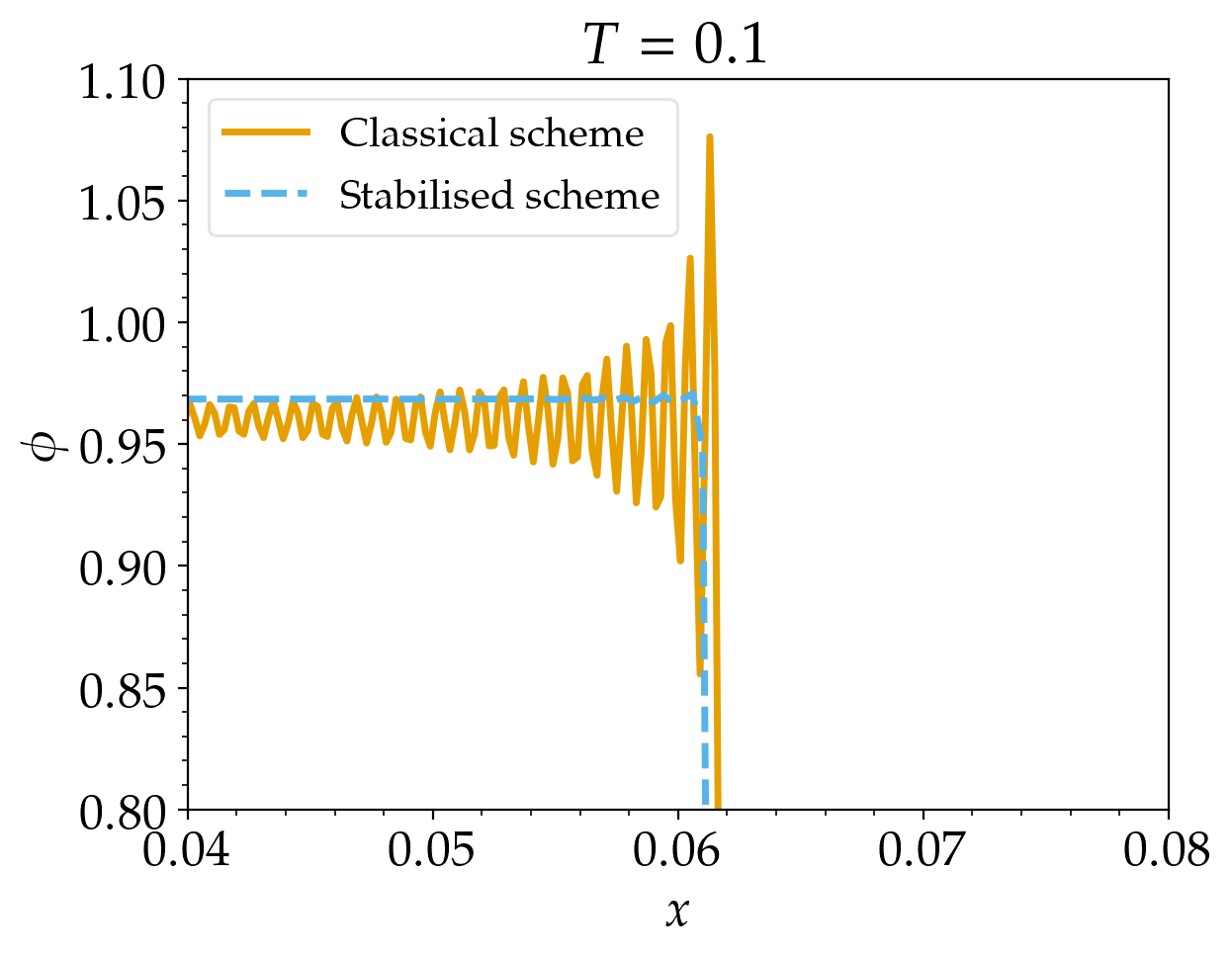}
    \caption{Zoom of the density, velocity and potential profiles for the shock tube test case at $T=0.1$ corresponding to $\veps=10^{-4}$.} 
    \label{fig:zoom_1d_shock-tube_eps-04}
\end{figure}

\subsection{1D Plasma Expansion}
\label{sec:1D-plasma_expansion}
The expansion of a plasma into vacuum is widely used in plasma dynamics, particularly in fields such as space science, laser ablation and electric propulsion among others. We consider the following classical 1D plasma expansion test problem from  \cite{AP14}. The initial setup consists of a smooth ion front at rest which then starts moving to the right under the influence of an electric field. The initial data are given by
\begin{equation}
    \rho(0, x) = 0.5 -\frac{1}{\pi}\arctan(\pi x),\; u(0, x) = 0.
\end{equation}
The computational domain $[-10, 40]$ is divided using $10000$ mesh points. We assign a small value $\veps = 10^{-2}$ for the Debye length. We set homogeneous Neumann boundary conditions for the density and the potential on both sides and a no-slip boundary condition on the left and a Neumann condition on the right for the velocity. In Figure \ref{fig:1d_plasma-expn}, we plot the ion density, the electron density given by the Boltzmann relation \eqref{eq:boltz_relation} after scaling, the velocity and the potential profiles. It is evident from the plots that the positive ions generate a peak in their density, at different times, at the ion front before the density drops to zero. The peaks are formed at those points where the ion and the electron densities are equal, and consequently the electric field is a maximum. The increased electric field results in the ion velocity exceeding the electron thermal velocity. The same trend is visible in the velocity plots where abrupt jumps can be seen at the points where the velocity attains maximum. The points of jump also correspond to a drastic fall in the potential, which creates a sharp increase in the electric field. This behaviour could result in wave breaking and the formation of multi-valued ion velocity distributions. Our results are highly consistent with the theoretical and the numerical studies reported in the literature; see, e.g.\ \cite{Mor03,SS85}. As reported in these references, we also observe that the peak intensity increases with mesh refinement though we have not displayed those results because of space constraints. 

\begin{figure}[htbp]
    \centering
    \includegraphics[height=0.25\textheight]{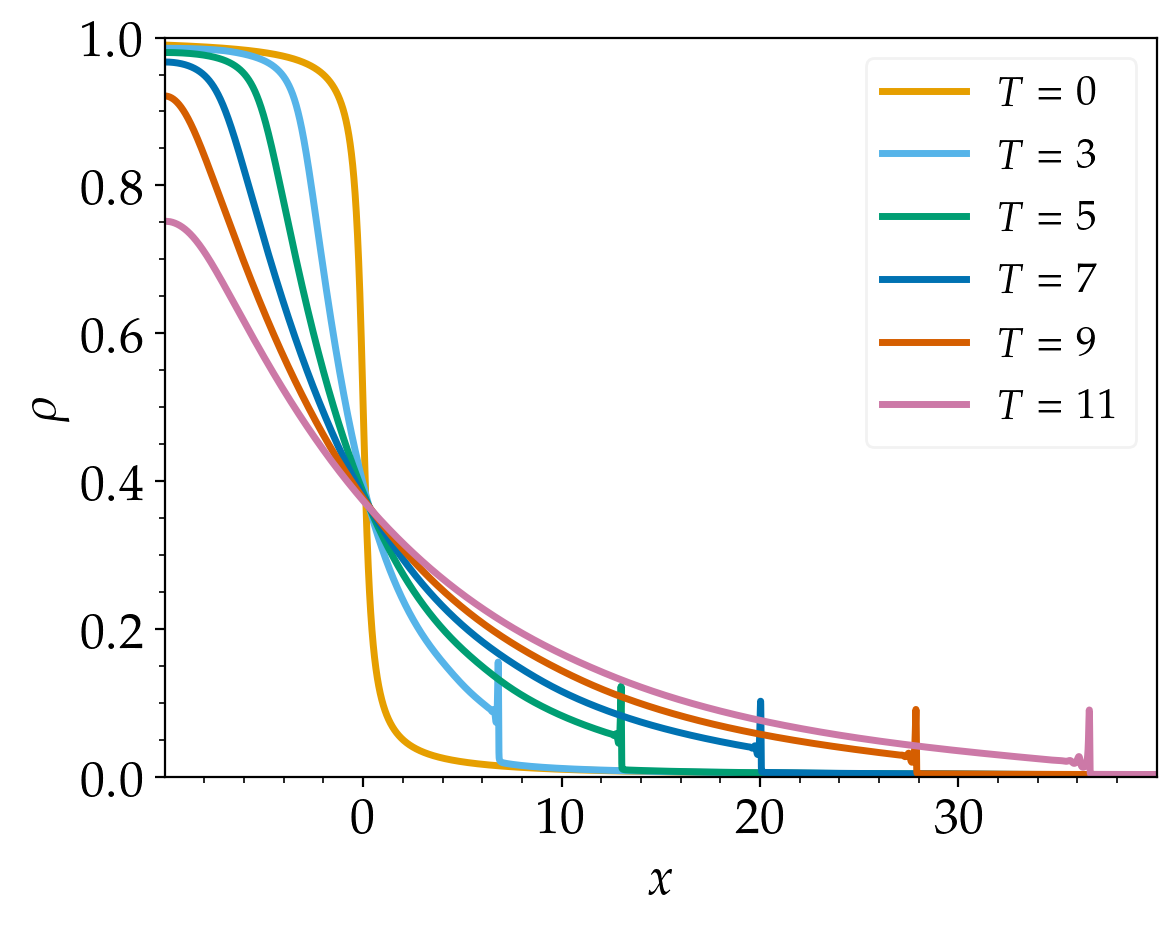}
    \includegraphics[height=0.25\textheight]{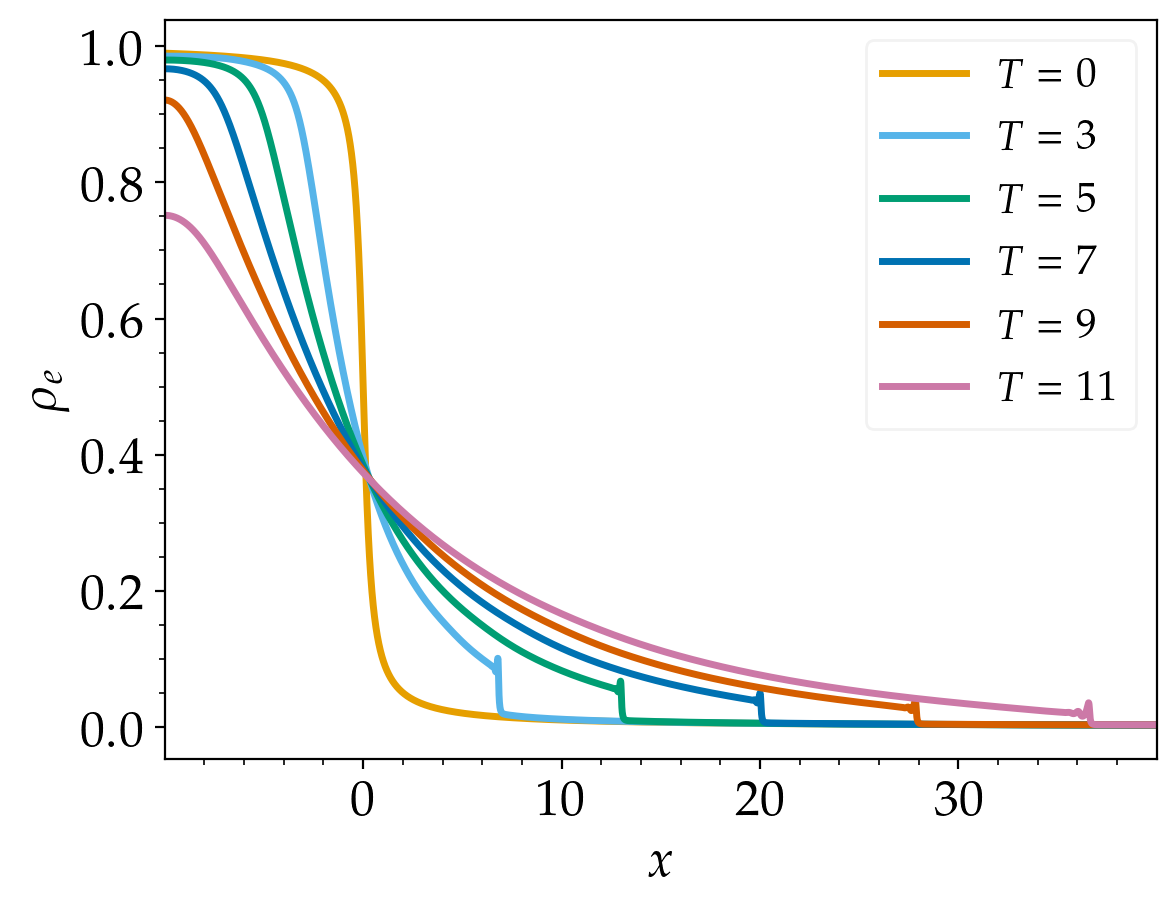}
    \includegraphics[height=0.25\textheight]{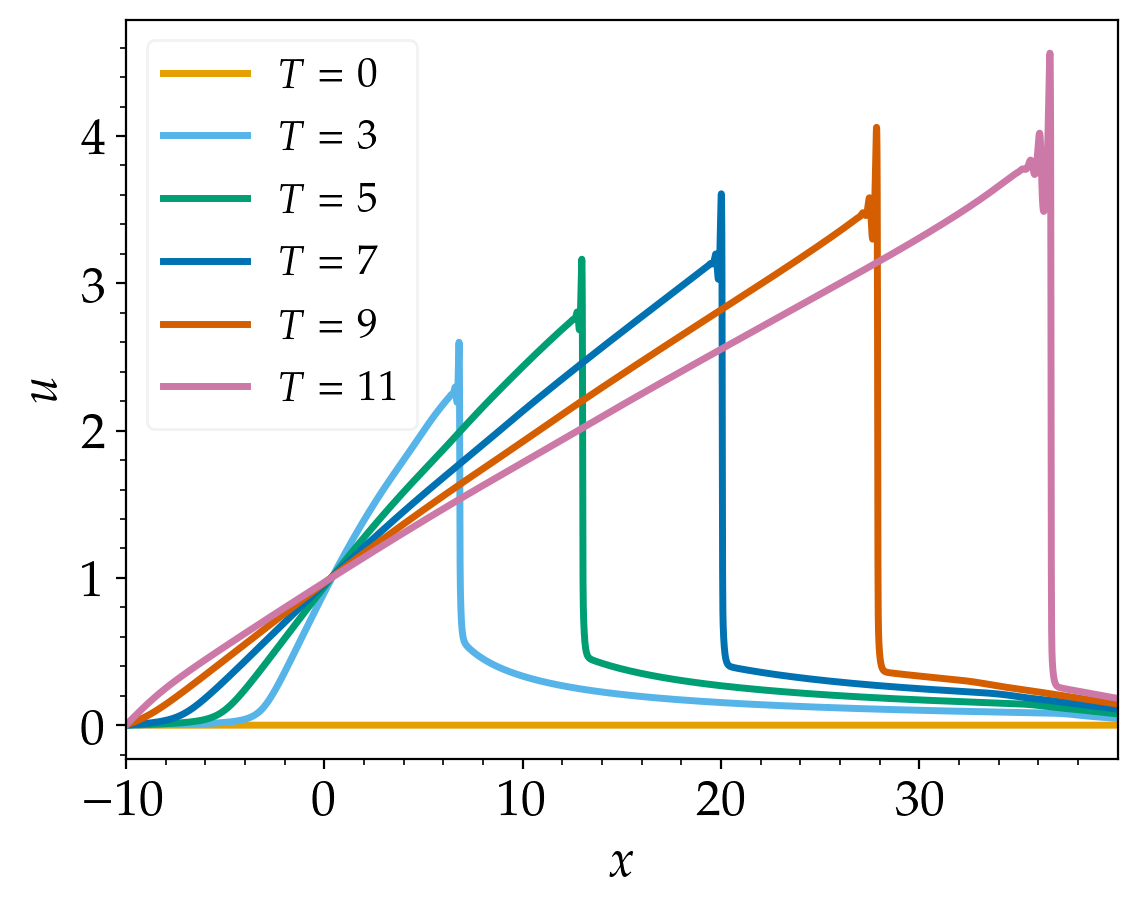}
    \includegraphics[height=0.25\textheight]{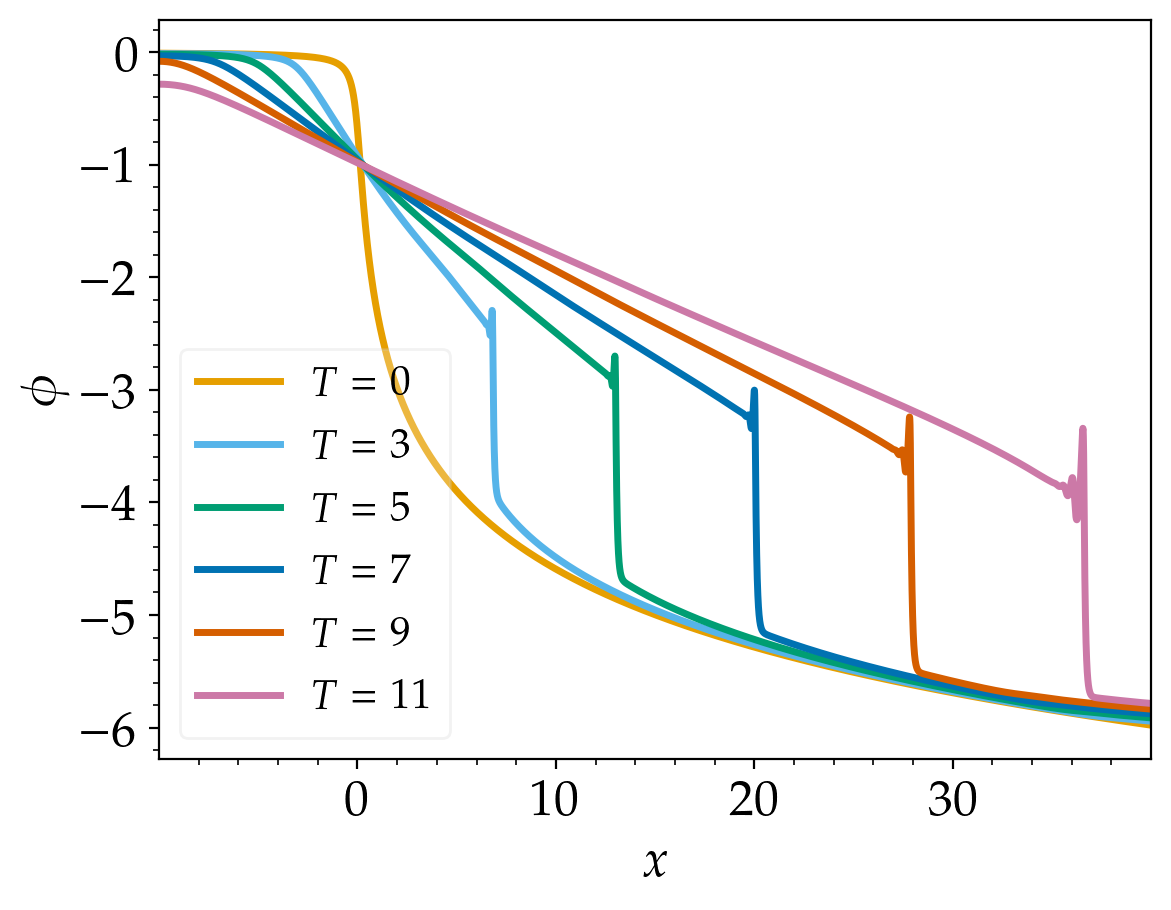}
    \caption{The ion and the electron densities, the velocity and the potential profiles for the plasma expansion test case.} 
    \label{fig:1d_plasma-expn}
\end{figure}

\subsection{Riemann Problem}
\label{sec:riemann_test}

In this test problem we intent to test the scheme's capability to capture the quasineutral limit regime for a 1D Riemann problem considered in \cite{PHG+13}. The initial density and velocity read
\begin{equation}
    \rho (0, x) = \begin{cases}
        1, & \text{if } x < 0,\\
        n_r, & \text{if } x \geqslant 0,
    \end{cases} 
    \quad u(0, x) = 0.
\end{equation}

Corresponding to this initial data, a self-similar solution of the ICE system can be computed as 
\begin{equation}
\label{eq:iso_exact_soln}
    [\rho(t, x), u(t, x)]^{T} = 
    \begin{cases}
        [1,0]^T, & \text{if } x \leqslant -t, \\
        [\exp(-\frac{x}{t}-1),\frac{x}{t}+1]^T, & \text{if }  -t < x \leqslant (u_m - 1)t, \\
       [\exp(-u_m),u_m]^T, & \text{if }  (u_m - 1)t < x \leqslant u_s t, \\
        [n_r,0]^T, & \text{if } x > u_st,
    \end{cases}
\end{equation}
where the maximum velocity $u_m$ is obtained by solving the equation
\begin{equation}
    \left(1 - n_r \exp(u_m)\right)\left(u^2_m - 2u_m - 2 \log(n_r)\right) - 2u^2_m = 0
\end{equation}
and the shock velocity $u_s$ is given by
\begin{equation}
    u_s = \frac{u_m}{1 - n_r \exp(u_m)}.
\end{equation}
We expect the numerical solution obtained using the AP scheme to match this self-similar solution when $\veps\to0$. To this end, we set $\veps=10^{-4}$ in order to emulate a quasineutral regime. The parameter $n_r$ in the initial density is given the values $0.5,0.75$ and $0.95$ to vary the strength of the discontinuity. We take a computational domain $[-80, 100]$ which is discretised using $9000$ mesh points. Extrapolation, no-slip, and homogeneous Neumann boundary conditions have been applied at both ends for the density, velocity and potential respectively. The simulation runs up to a final time $T=50$.

\begin{figure}[htbp]
    \centering
    \includegraphics[height=0.24\textheight]{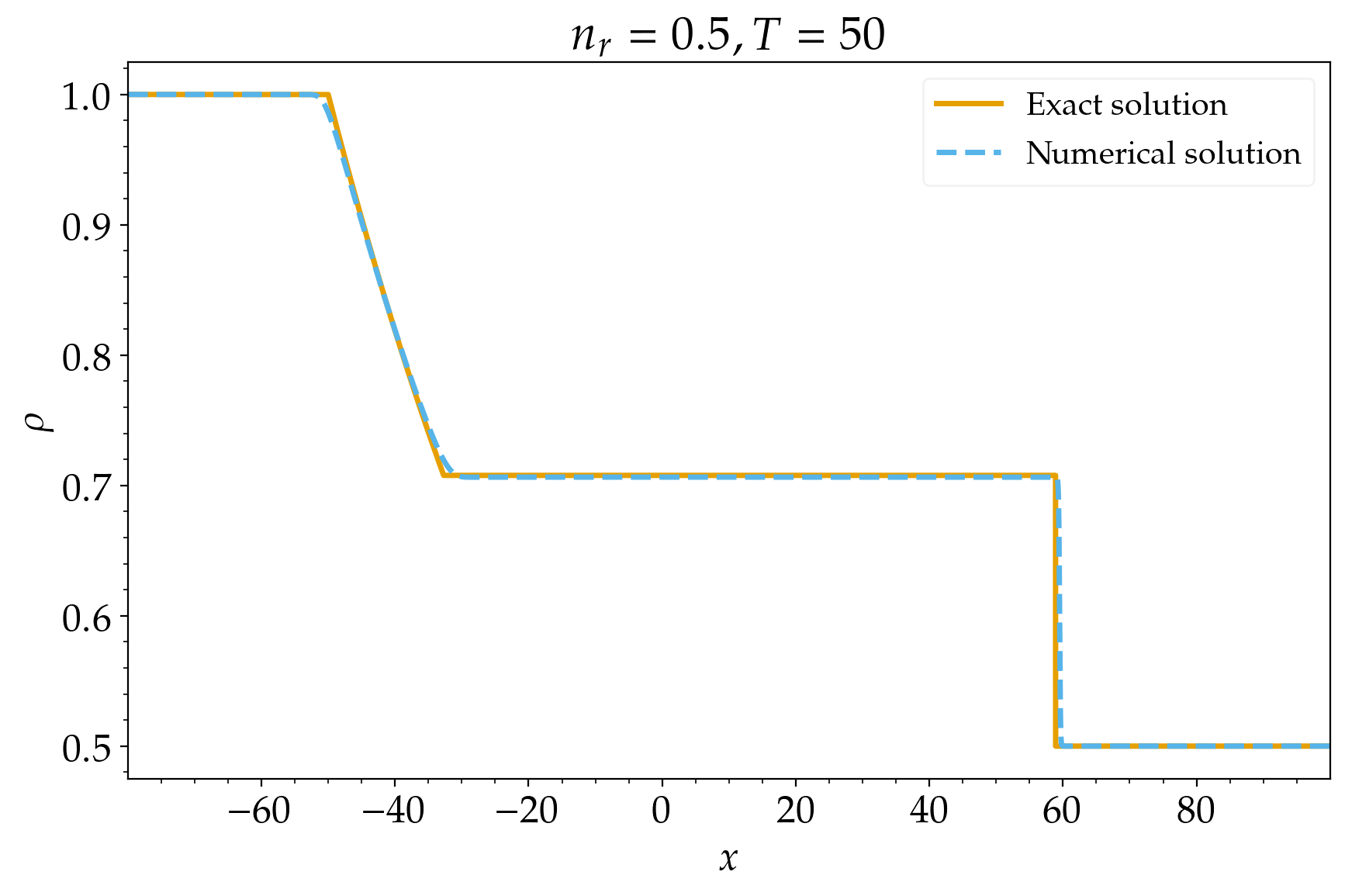}
    \includegraphics[height=0.24\textheight]{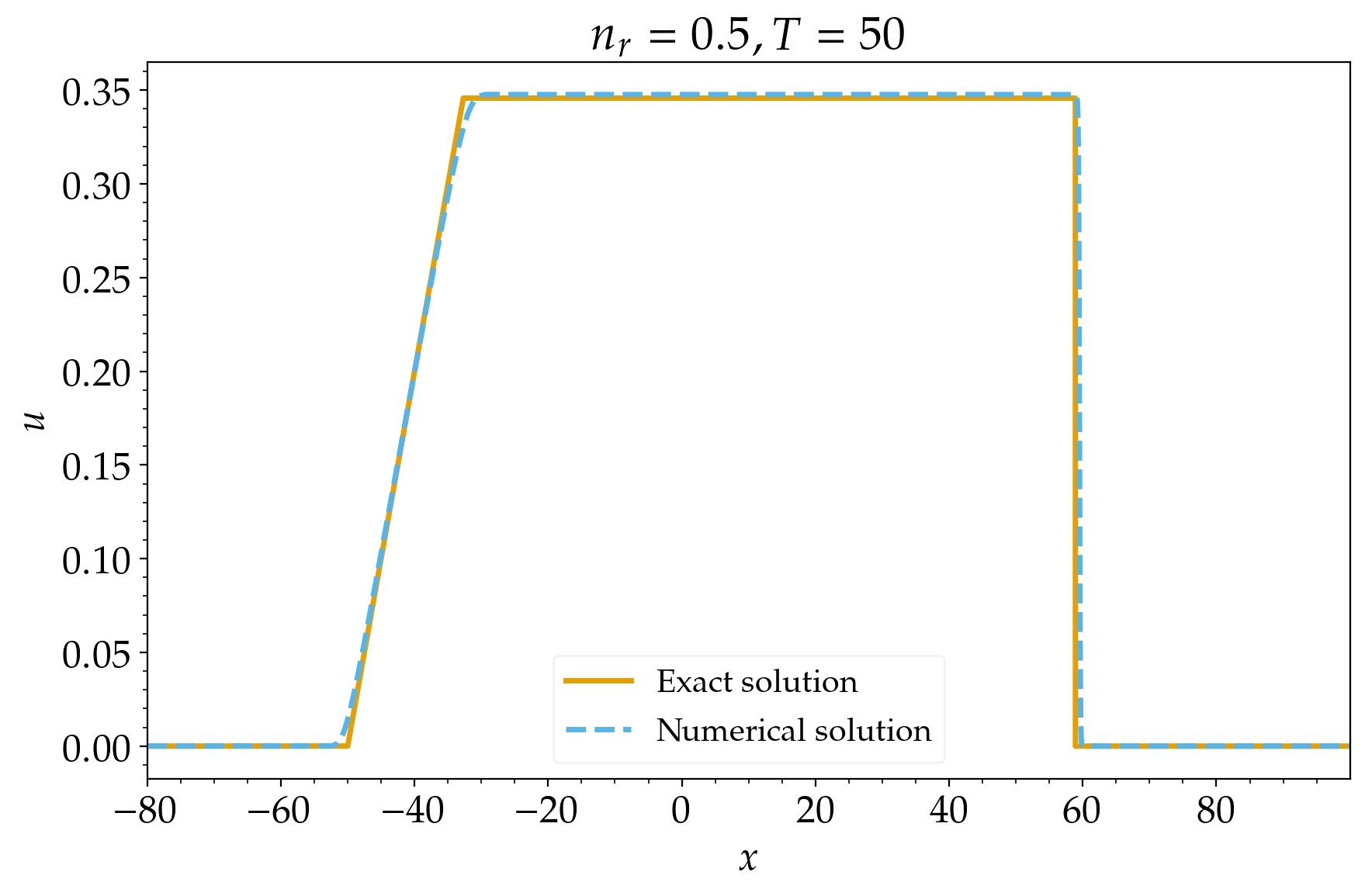}
    \includegraphics[height=0.24\textheight]{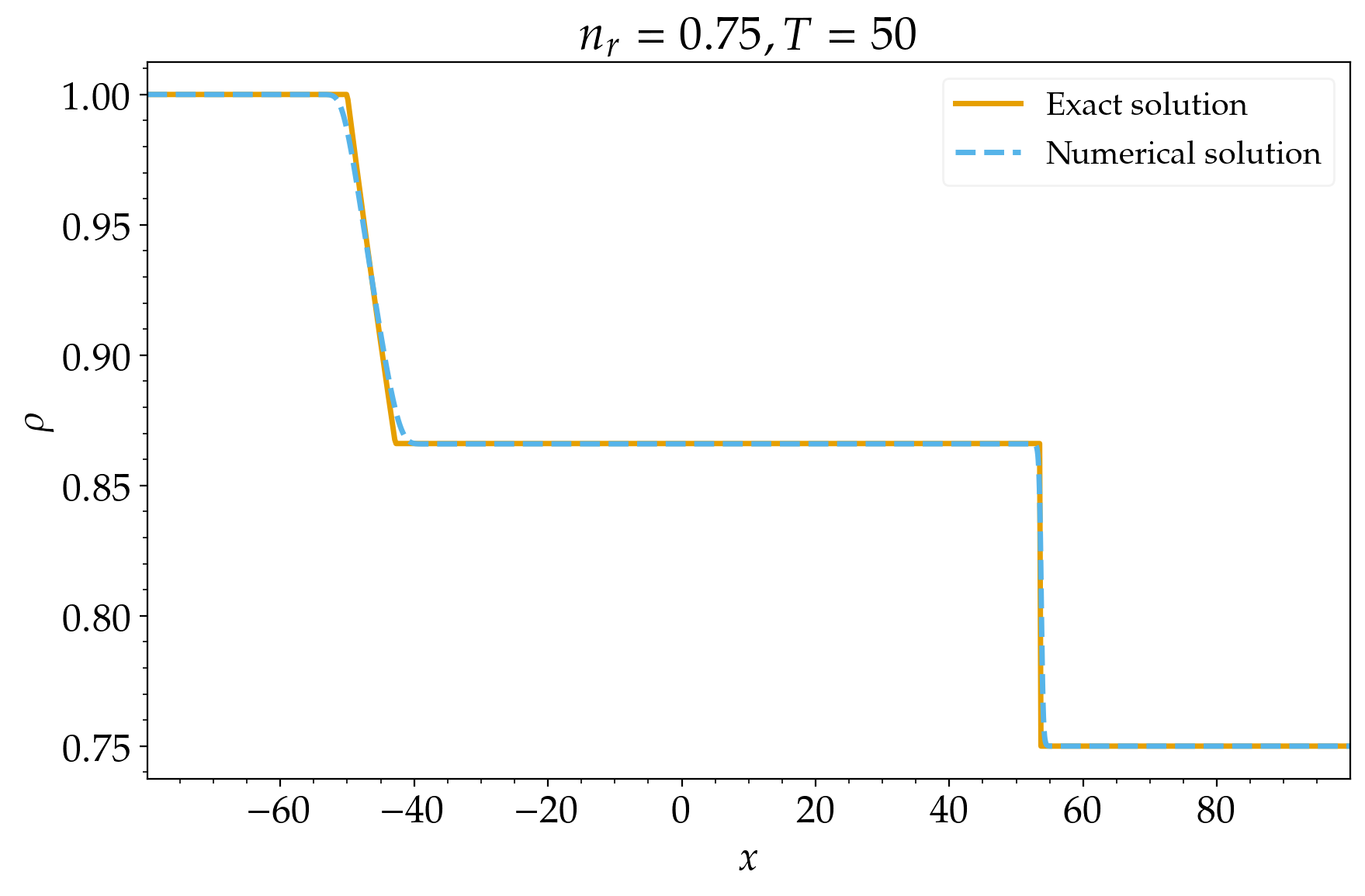}
    \includegraphics[height=0.24\textheight]{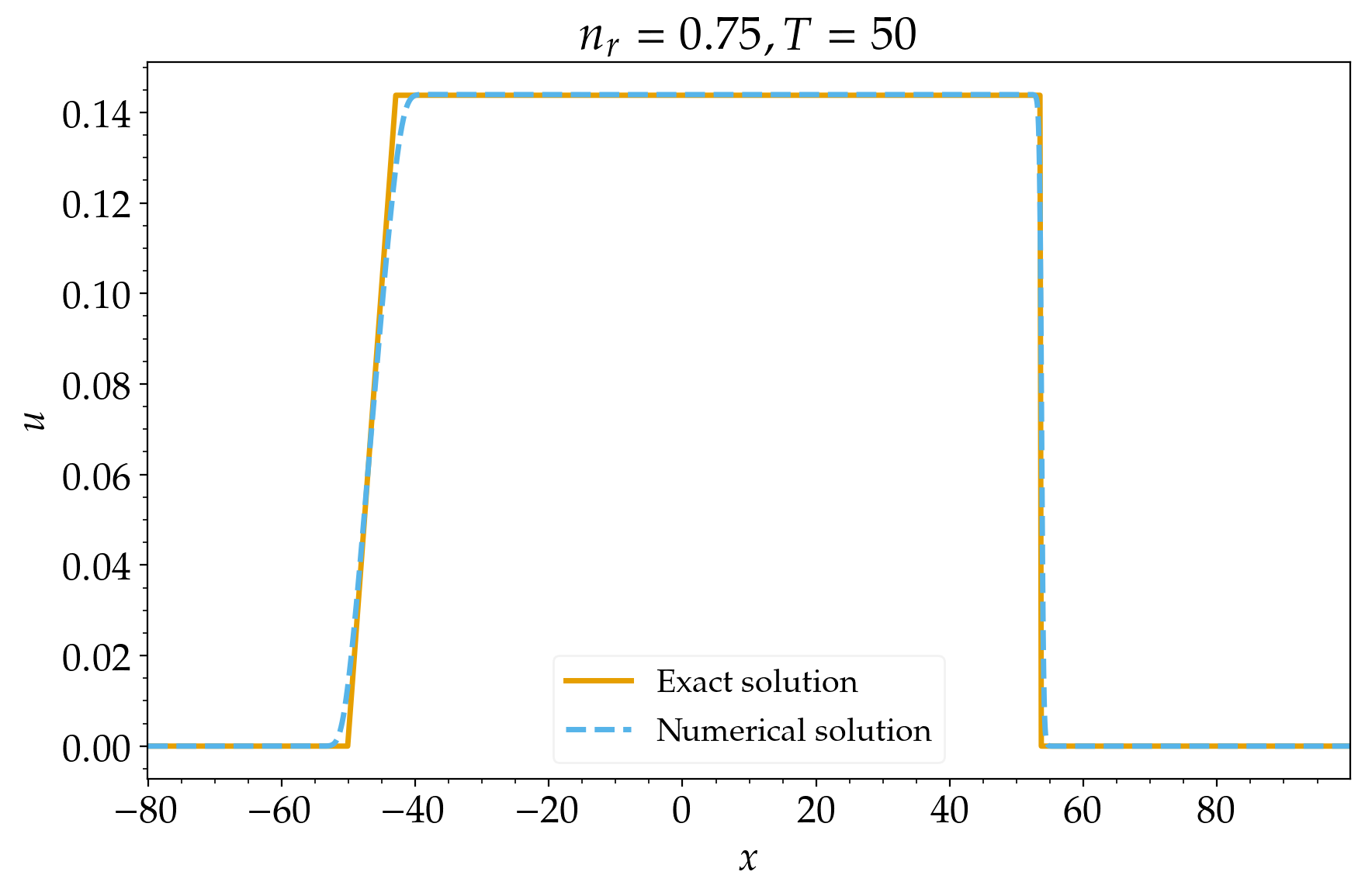}
    \includegraphics[height=0.24\textheight]{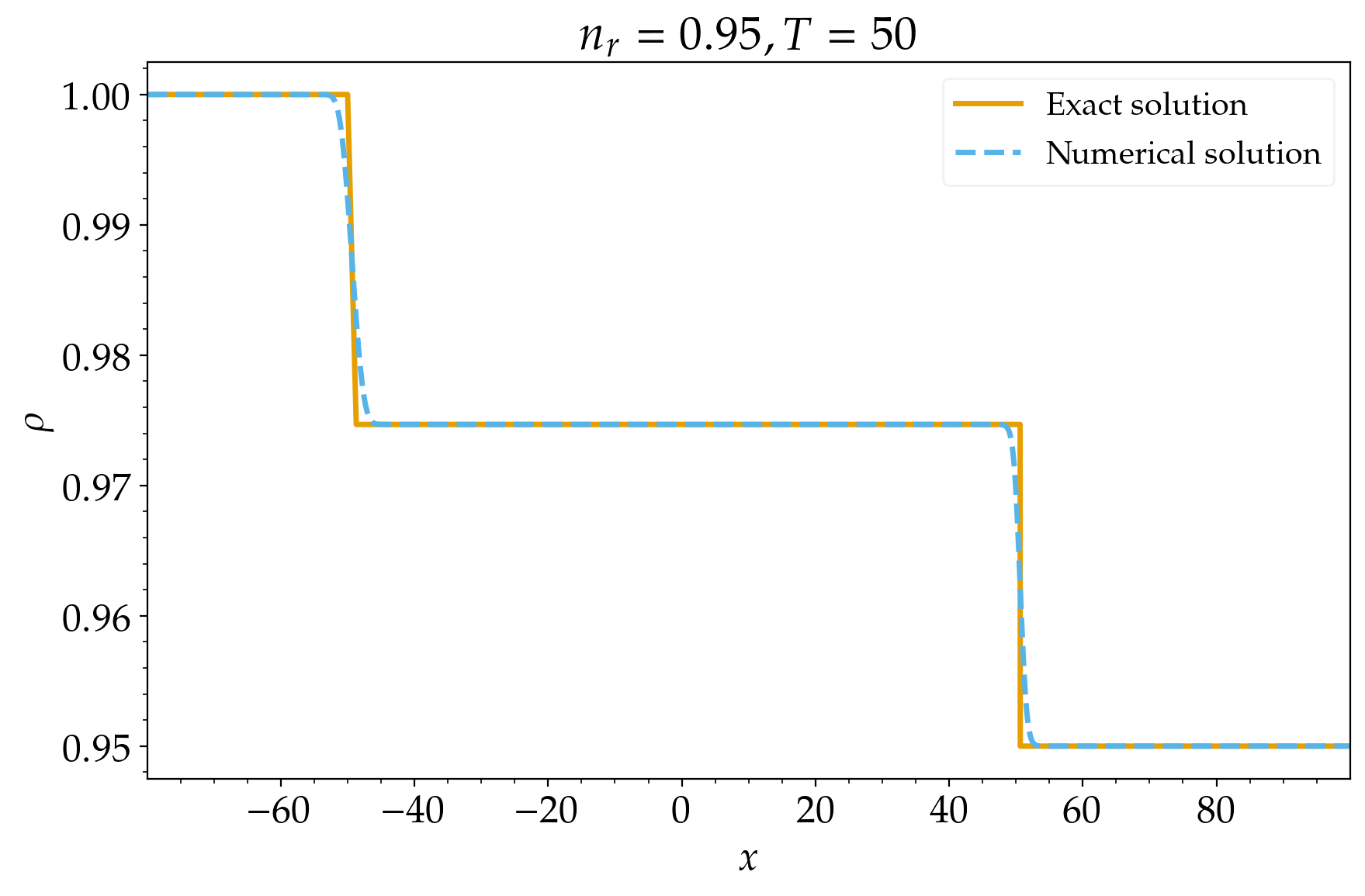}
    \includegraphics[height=0.24\textheight]{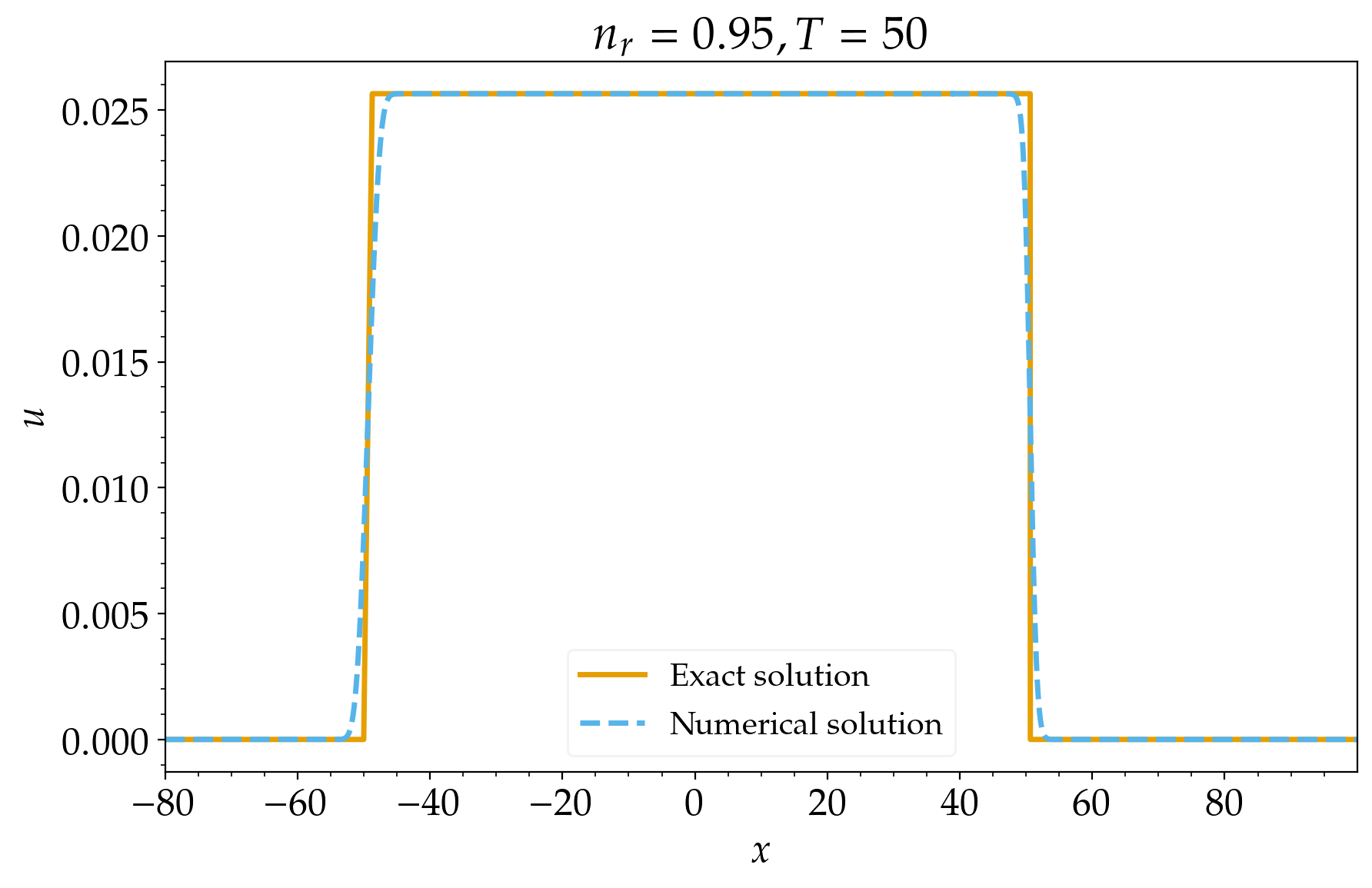}
    \caption{ The density, velocity profiles for the one-dimensional Riemann problem at time $T=50$ for $\veps=10^{-4}$.} 
    \label{fig:1d_riemann}
\end{figure}

In Figure~\ref{fig:1d_riemann}, we present the density and velocity plots for different values of $n_r$ against the exact solution \eqref{eq:iso_exact_soln}. We observe that the numerical solution obtained using the present semi-implicit scheme and the exact solution of ICE system are indistinguishable irrespective of the values of $n_r$, i.e.\ we obtain excellent results for different shock strengths. The shock position and shock speed are also in exact agreement with no oscillations in the neighbourhood of shocks. It is clear therefore that the scheme is indeed able to capture the quasineutral ICE solution as $\veps\rightarrow 0$, which corroborates the AP property. 

\subsection{Cylindrical Explosion}
\label{sec:cyl_expl}
Here we consider an axi-symmetric 2D plasma explosion problem. This case study showcases the shock capturing capability of the scheme in the quasineutral regime. The computational domain is $\Omega=[-1,1]\times[-1,1]$ and the initial condition is given by 
\begin{equation}
    \rho(0, x, y) = 
    \begin{cases}
        1, & \text{if } r \leqslant \frac{1}{2}, \\
        0.1, & \text{otherwise, }
    \end{cases}
    \, u(0, x, y) = v(0, x, y)=0.
\end{equation}
where $r=\sqrt{x^2+y^2}$. The problem is supplemented with no-flux boundary conditions for all the sides for the density and the potential, no-slip boundary for the velocities. The domain is divided uniformly by a $200\times 200$ mesh. We simulate this problem in a quasineutral regime by setting a very small value $\veps=10^{-4}$.

\begin{figure}[htbp]
    \centering
    \includegraphics[height=0.3\textheight]{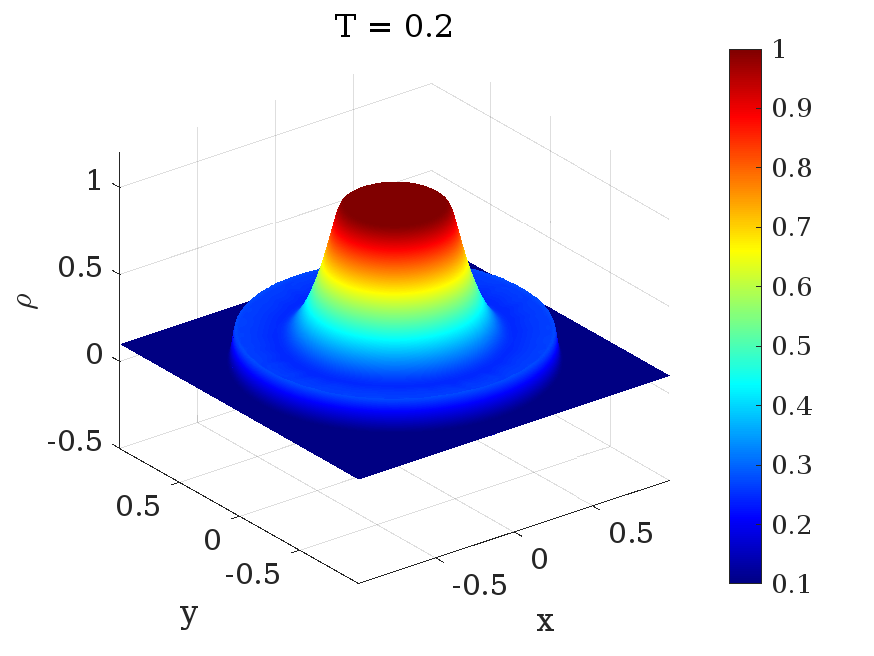}
    \caption{The density plot for cylindrical explosion problem at time $T=0.2$ for $\veps=10^{-4}$.} 
     \label{fig:den_cylExp}
\end{figure}

In Figure~\ref{fig:den_cylExp}, we present the surface plot of the density at $T=0.2$. Clearly, we observe that the plasma cloud with high density inside the circle of radius $0.5$ spreads out and creates a circular shock wave. In order to further elucidate the results, in Figure~\ref{fig:1d_cut_cylExp}, we compare the cross-sections of the density and the velocity along the radial direction with a reference solution obtained from the ICE system at $T=0.2$. The plots clearly show that the shock and expansion and well captured and that the numerical solution is in excellent agreement with the limiting solution.

\begin{figure}[htbp]
    \centering
    \includegraphics[height=0.3\textheight]{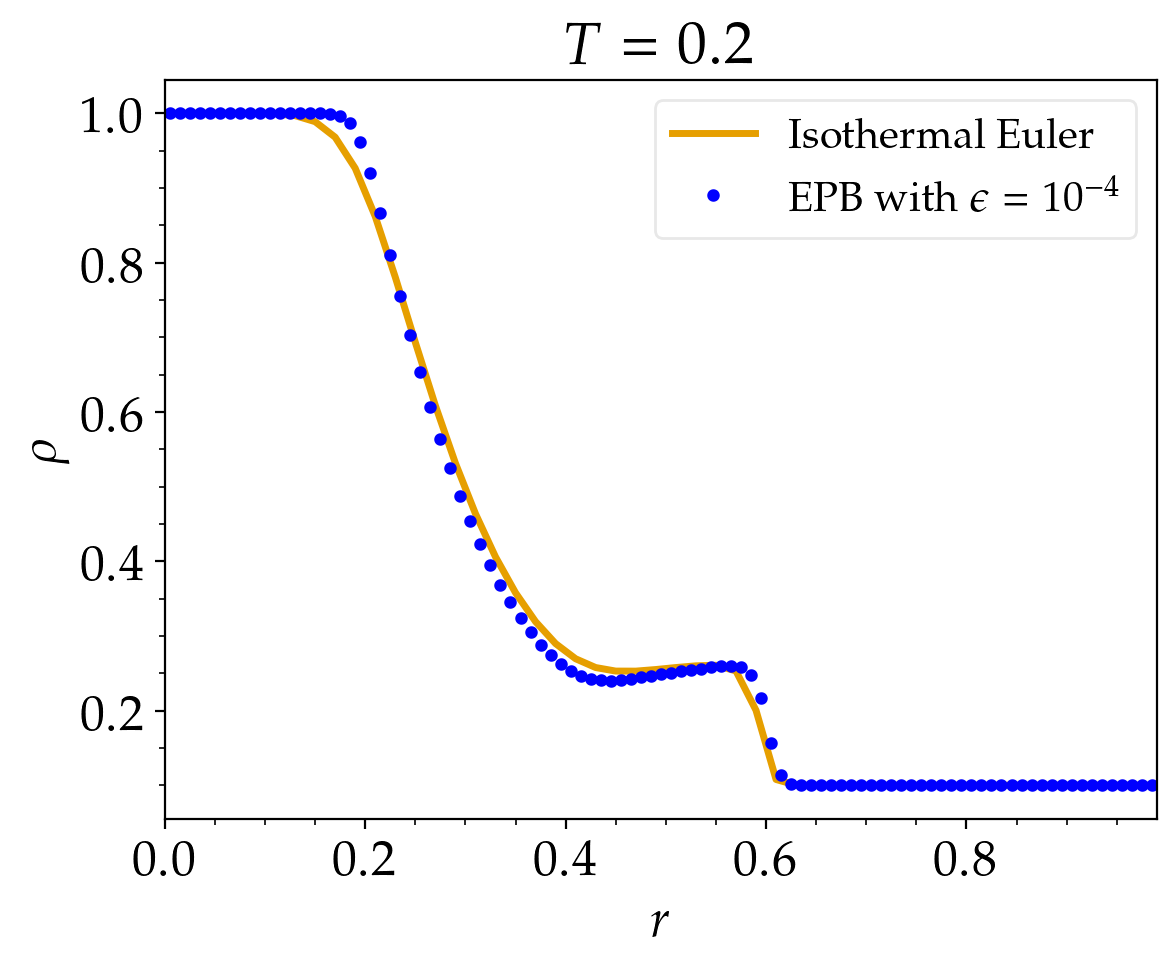}
    \includegraphics[height=0.3\textheight]{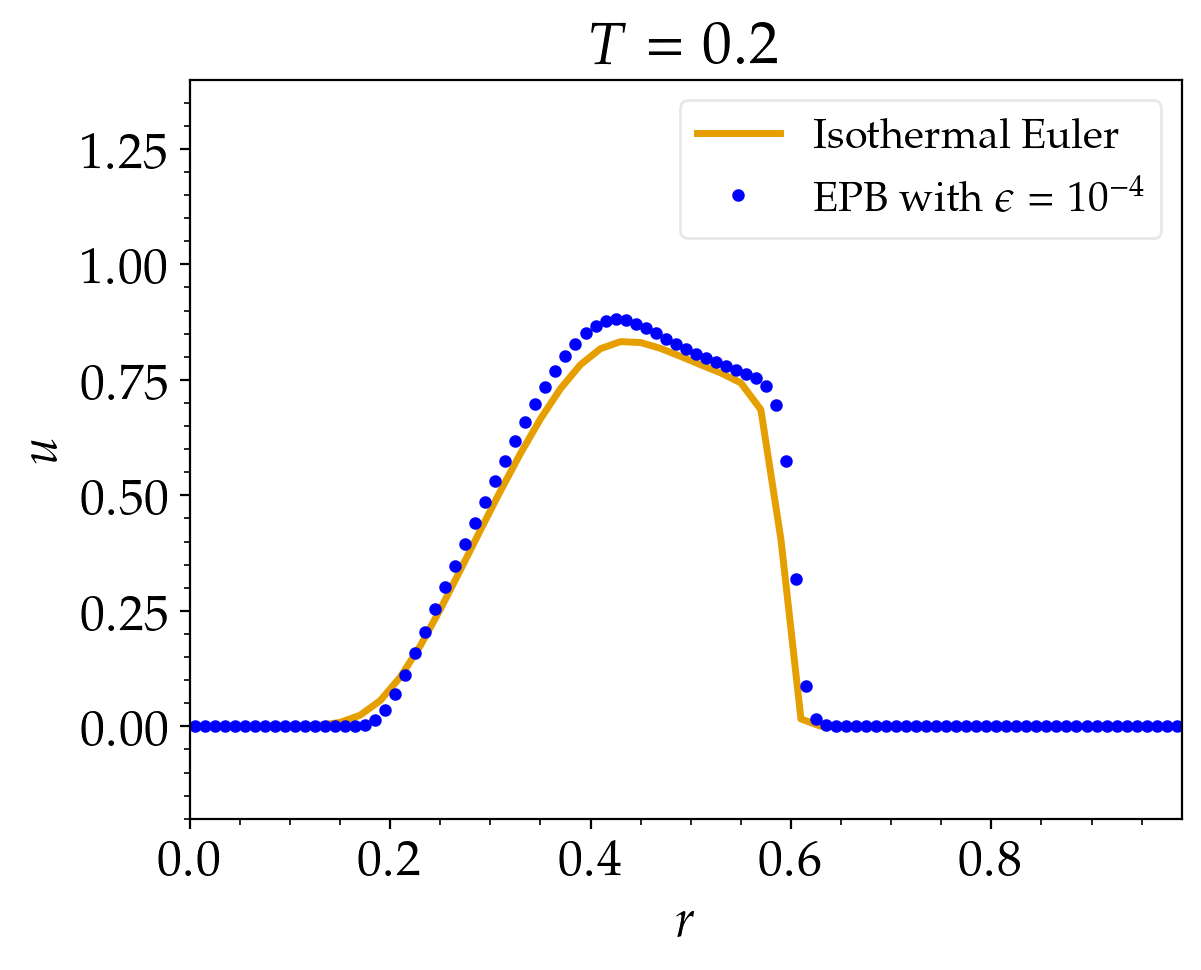}
    \caption{ 1D radial cut of Density and velocity at time $T=0.2$.} 
    \label{fig:1d_cut_cylExp}
\end{figure}

\subsection{Ion Extraction Problems}
The next case study involves the ion extraction from plasmas which has earned significant attention over the years due to its relevance in various applications including particle implantation processes and the generation of intense particle beams. We consider a benchmark test problem from \cite{MG05,VCB92}, which involves 1D and 2D hydrodynamics of ion extraction from a uniform quasineutral plasma in an electrostatic field using two parallel plates. The ions are extracted from a drifting, high density, grounded plasma onto a negatively charged electrode. The extracted ions form a low density sheath which shields the quasineutral plasma from the electrode voltage. The thickness of the sheath is of the order of a few Debye lengths and accordingly, the ion extraction problems demand a spatial grid resolution smaller than the Debye length $\veps$ to resolve the sheath. In what follows, we consider both 1D and 2D configurations of ion extraction motivated by the studies from \cite{MG05,VCB92}.    

\subsubsection{1D Problem}
\label{sec:1D-sheath}
In the 1D setting, we show that the present scheme is able to capture the physics of classical plasma sheath by computing the potential drop as predicted by the theory. To this end, we consider a 1D steady-state ion sheath problem studied in \cite{VCB92}. In order to replicate the initial vacuum state, we prescribe a very small initial density inside the computational domain $\Omega=[0,1]$ by setting
\begin{equation}
    \rho (0, x) = 10^{-5}, \, \text{ if } 0<x<1.
\end{equation}
An ion flux is injected at the left boundary $x = 0$ which corresponds to specifying $\rho(t,0)=5$ and $u(t,0)=1.25$ for $t\geqslant 0$. We take extrapolation boundary conditions at the right end for the density and velocity. Further, we choose a homogeneous Neumann boundary condition for the potential at the left end of the boundary and place an extractor at the right end by imposing a boundary condition of the form
\begin{equation}
    \phi (t, 1) = -500, \quad t\geqslant 0.
\end{equation}
The domain is discretised using $100$, $500$ and $1000$ grid points and the simulations are run up to a final time $T=2.5$ corresponding to $\veps=10^{-2}$.

\begin{figure}[htbp]
    \centering
    \includegraphics[height=0.191\textheight]{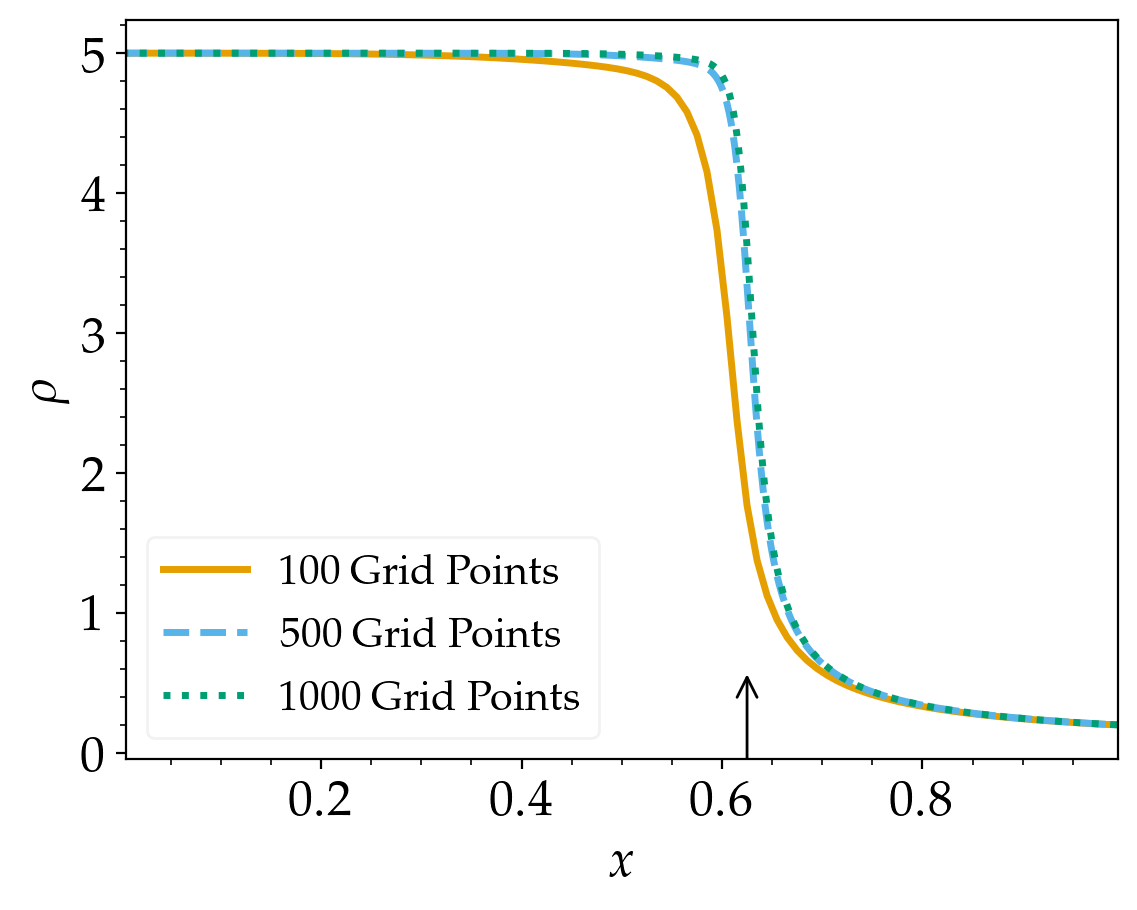}
    \includegraphics[height=0.191\textheight]{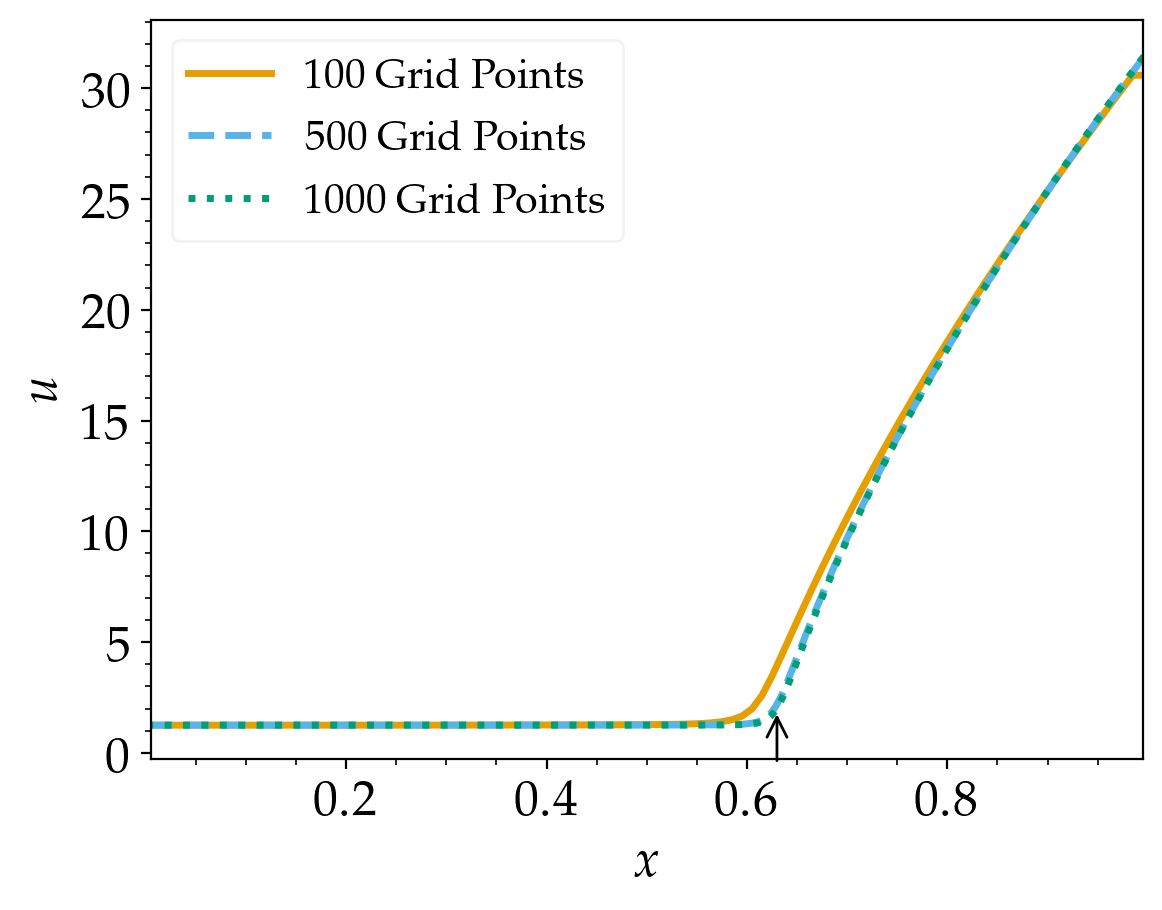}
    \includegraphics[height=0.191\textheight]{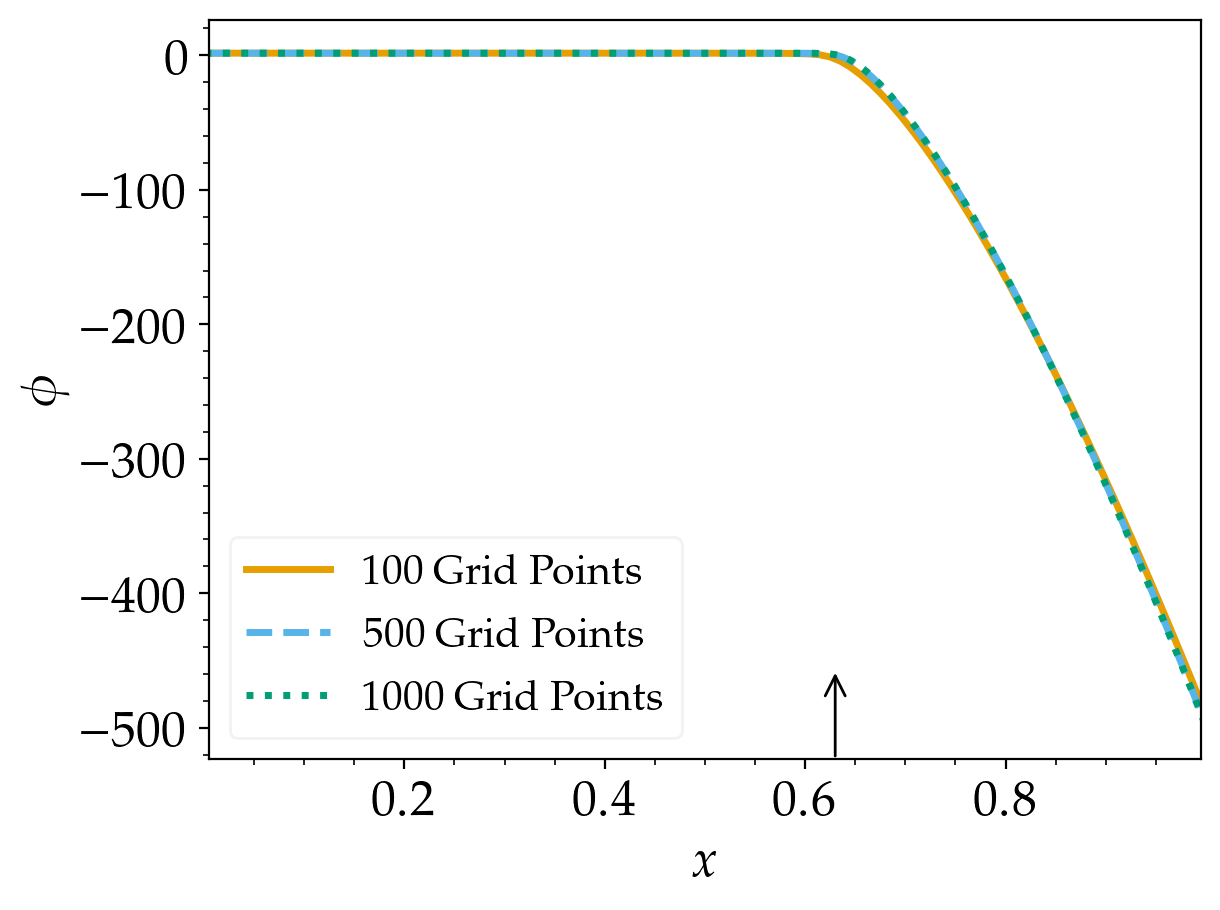}
    \caption{The density, velocity and potential profiles for the 1D sheath test case at time $T=2.5$ for $\veps=10^{-2}$.} 
    \label{fig:1d_sheath}
\end{figure}
We plot the density, velocity and potential profiles in Figure~\ref{fig:1d_sheath} for the different grid sizes. We clearly see the formation of plasma and the sheath region. The arrow in the figures represents the sheath transition region which forms at a distance $s$ from the extractor, where $s$ is given by the Child-Langmuir formula
\begin{equation}
    \label{eq:Child_Lang}
    s = \half\frac{2^{5/4}}{3}\frac{V^{3/4} \veps}{\sqrt{u}},
\end{equation}
with $V=500$ being the voltage difference and $u$, the initial velocity. Plugging in the values of the parameters, we obtain $s\approx 0.375$. We note that the location of the sheath transition region obtained using the present scheme and the location of the arrow determined by the equation \eqref{eq:Child_Lang} agree very well. Furthermore, we observe that there are no significant changes in the resulting sheath transition region when changing the grid, which shows the ability of the scheme to capture the sheath regions even on a coarse mesh. 

\subsubsection{2D Problem}
\label{sec:sheath}
We consider another benchmark test problem from \cite{MG05,VCB92}, which involves the 2D hydrodynamics of ion extraction from a uniform quasineutral plasma in an electrostatic field using two parallel plates. Figure~\ref{fig:schematic_extractor} portrays the computational domain, the initial density profile of the uniform plasma, the cathode plate and the boundaries of the plasma. It is assumed that the cathode is on the right side of the plasma, connected to a high voltage $\phi_c = -1000$ and the other side of the plasma has a free boundary. The top and bottom side electrodes are connected to a typical small voltage, $\phi_t = \phi_b = -10$. As shown in the schematic diagram, we have assumed that ion plasma with initial density $\rho(0, x, y) = 5$ is created near the free boundary at the beginning of the simulation. The initial velocity is taken as $u(0, x, y) = 0$ and $v(0, x, y) = 0.25$. We set the scaled Debye length $\veps= 0.18\times 10^{-2}$. The computational domain is divided uniformly into $80\times 200$ mesh cells. The boundary conditions on all four boundaries are extrapolations for the density. For both velocities, no-slip conditions are set on the left end and extrapolation on the remaining sides.

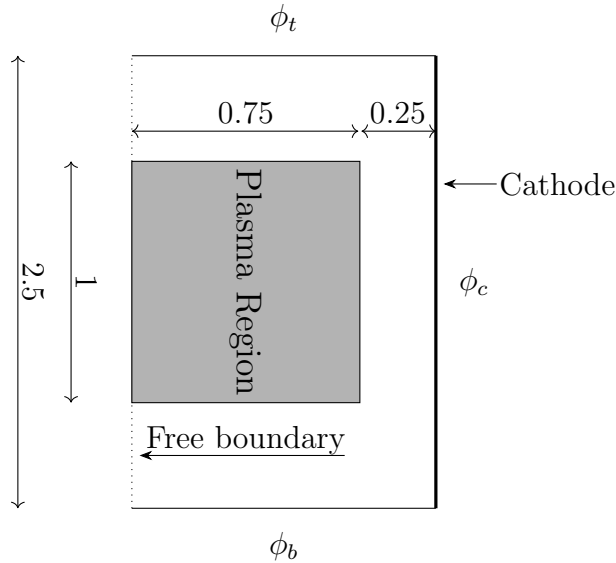
\begin{figure}
\centering
\begin{tikzpicture}
\draw[dotted] (0, 0) -- (0, 2.4);
\draw (0, 2.4) -- (0, 4.6);
\draw[dotted] (0, 4.6) -- (0,6);
\draw[very thick] (4, 0) -- (4,6);
\draw (0, 6) -- (4,6);
\draw (0, 0) -- (4,0);
\fill[black!30!white] (0,1.4) rectangle (3.0, 4.6);
\draw[black] (0, 1.4) rectangle (3,4.6);
\filldraw[black] (1.5, 3.0) node[align=center, rotate=-90]{\scalebox{1.12}{Plasma Region}};

\draw[<->] (-1.5, 0) -- (-1.5,6);
\filldraw[black] (-1.3, 3.0) node[align=center, rotate=-90]{$2.5$};
\draw[<->] (-0.8, 1.4) -- (-0.8,4.6);
\filldraw[black] (-0.6, 3.0) node[align=center, rotate=-90]{$1$};

\draw[<->] (0.0, 5) -- (2.98, 5);
\filldraw[black] (1.5, 5.25) node[align=center]{$0.75$};
\draw[<->] (3.02, 5) -- (4,5);
\filldraw[black] (3.5, 5.25) node[align=center]{$0.25$};

\draw[Stealth-] (0.1, 0.7) -- (2.8, 0.7);
\filldraw[black] (1.5, 0.9) node[align=center]{Free boundary};
%---------phi-values------------
\filldraw[black] (2, -0.5) node[align=center]{$\phi_b$};
\filldraw[black] (2, 6.5) node[align=center]{$\phi_t$};
\filldraw[black] (4.5, 3.0) node[align=center]{$\phi_c$};

%---------cathode----------
\draw[Stealth-] (4.1, 4.3) -- (4.8,4.3);
\filldraw[black] (5.6, 4.3) node[align=center]{Cathode};

\end{tikzpicture}
\caption{Schematic geometry for Ion extraction test problem}
\label{fig:schematic_extractor}
\end{figure}

\begin{figure}[htbp]
  \centering
    \includegraphics[height=0.21\textheight]{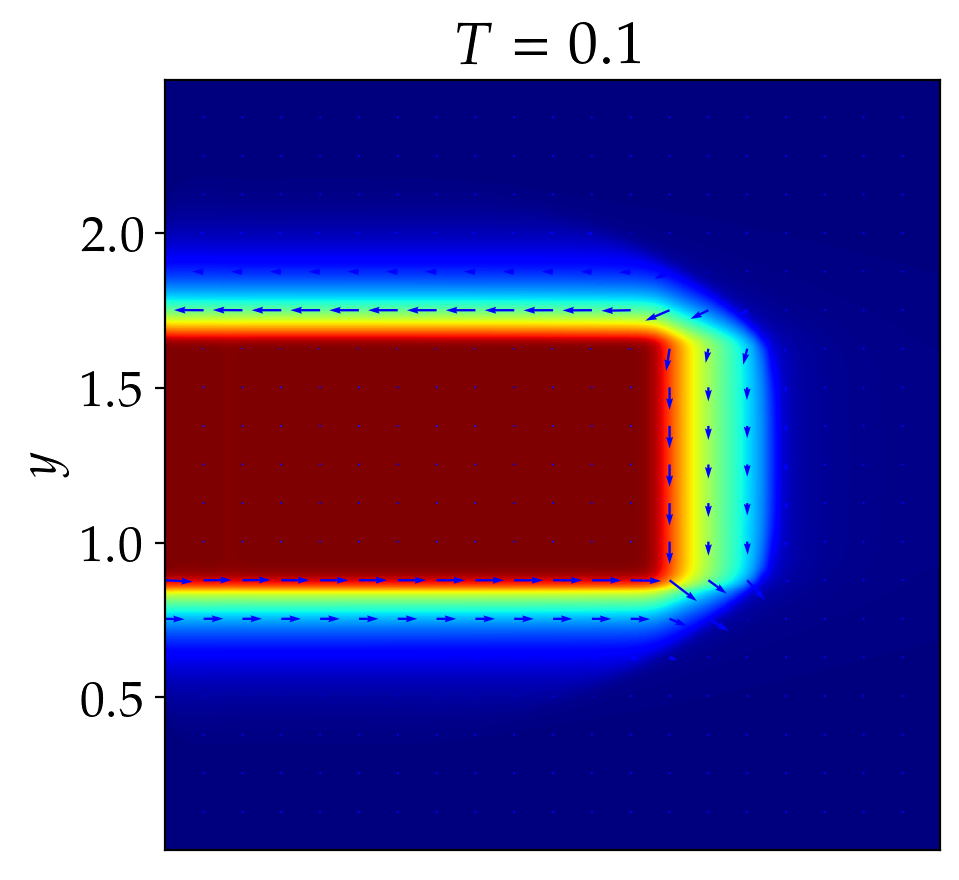}
    \includegraphics[height=0.21\textheight]{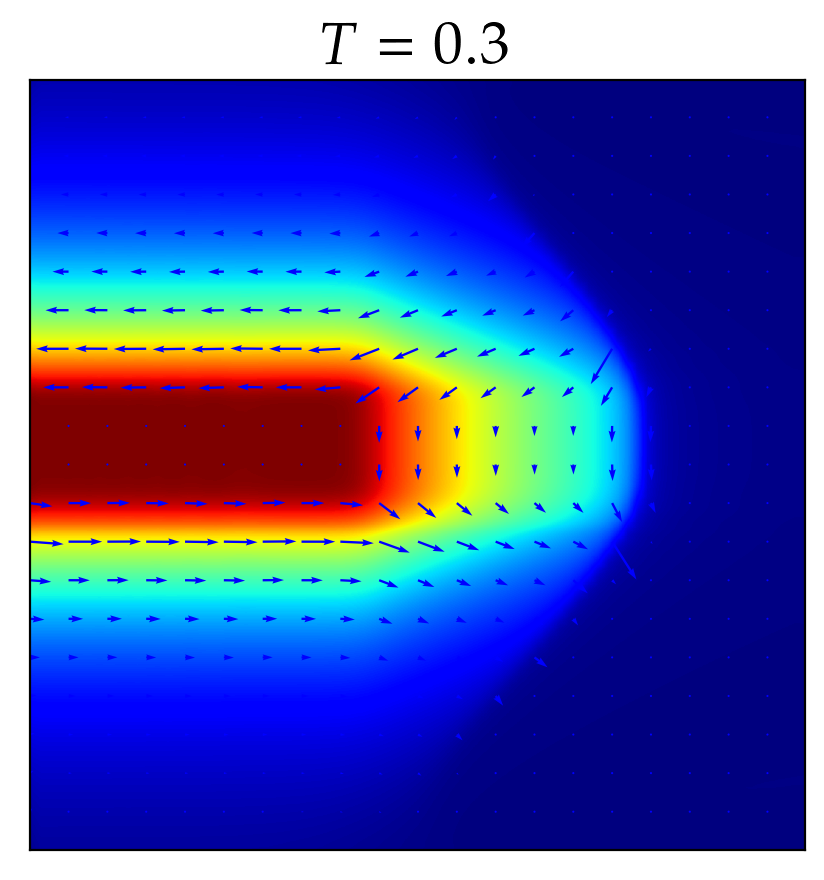}
    \includegraphics[height=0.21\textheight]{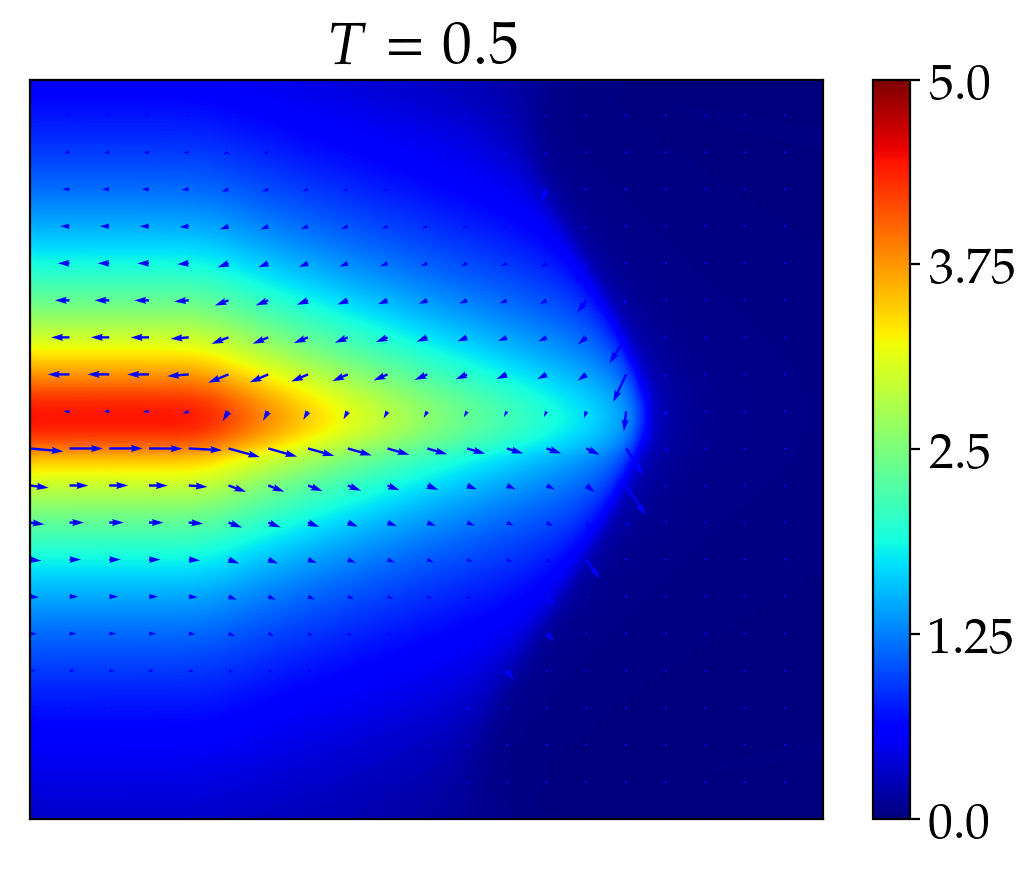}
    \includegraphics[height=0.19\textheight]{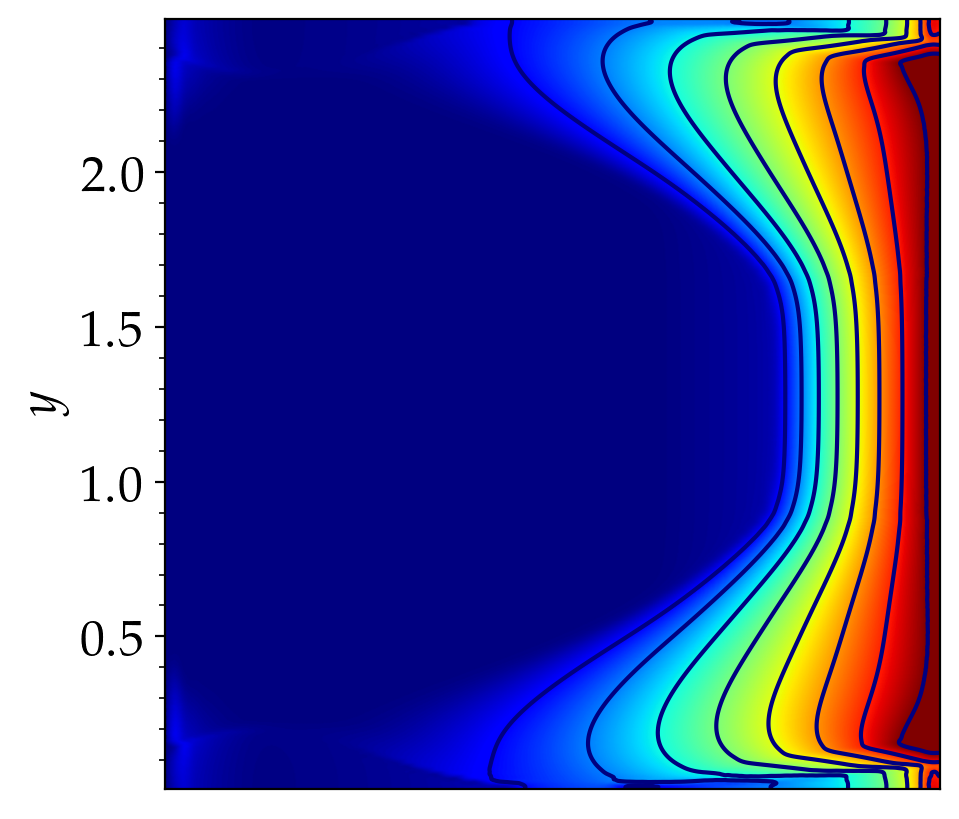}
    \includegraphics[height=0.19\textheight]{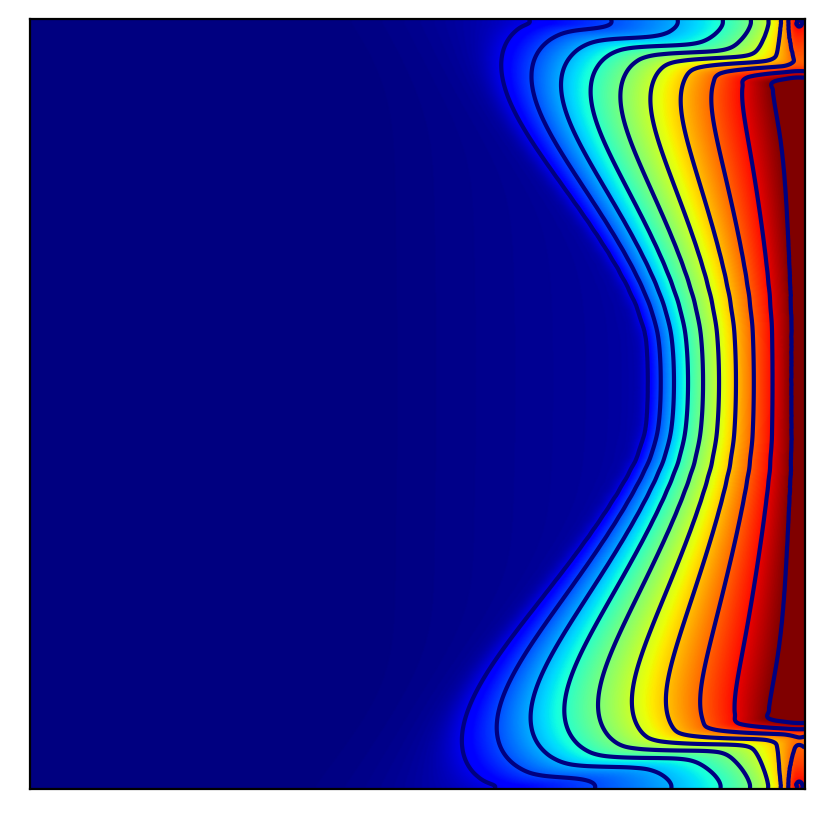}
    \includegraphics[height=0.2\textheight, width=0.335\textwidth]{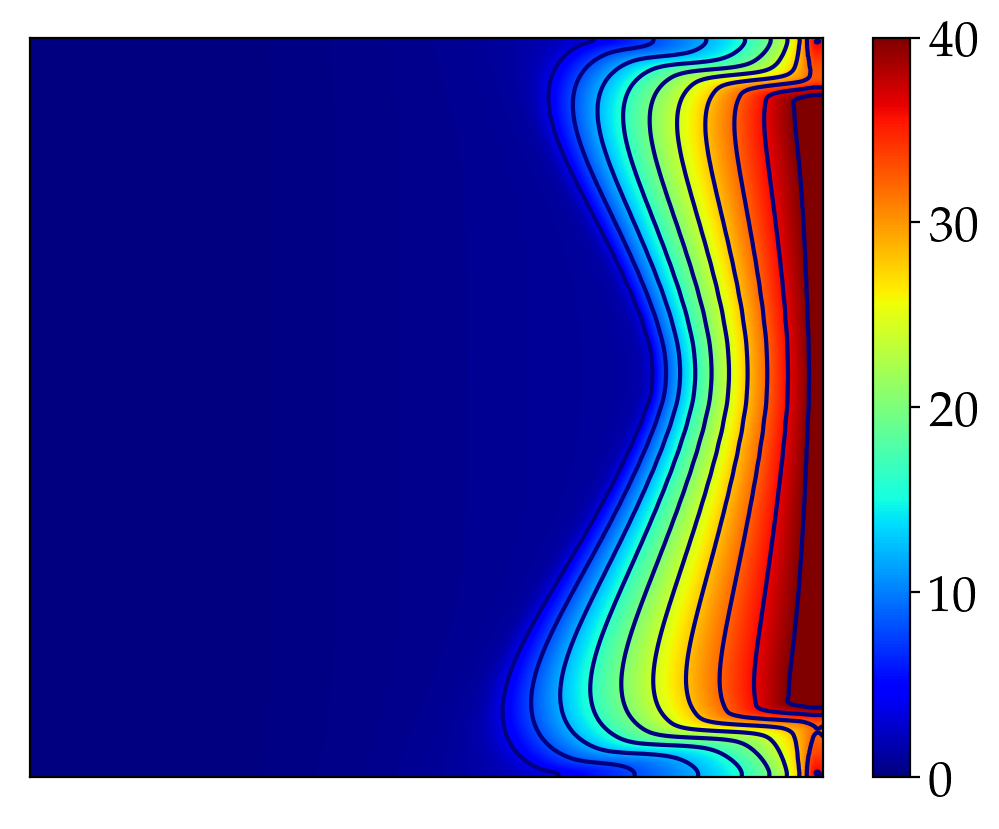}
    \includegraphics[height=0.22\textheight]{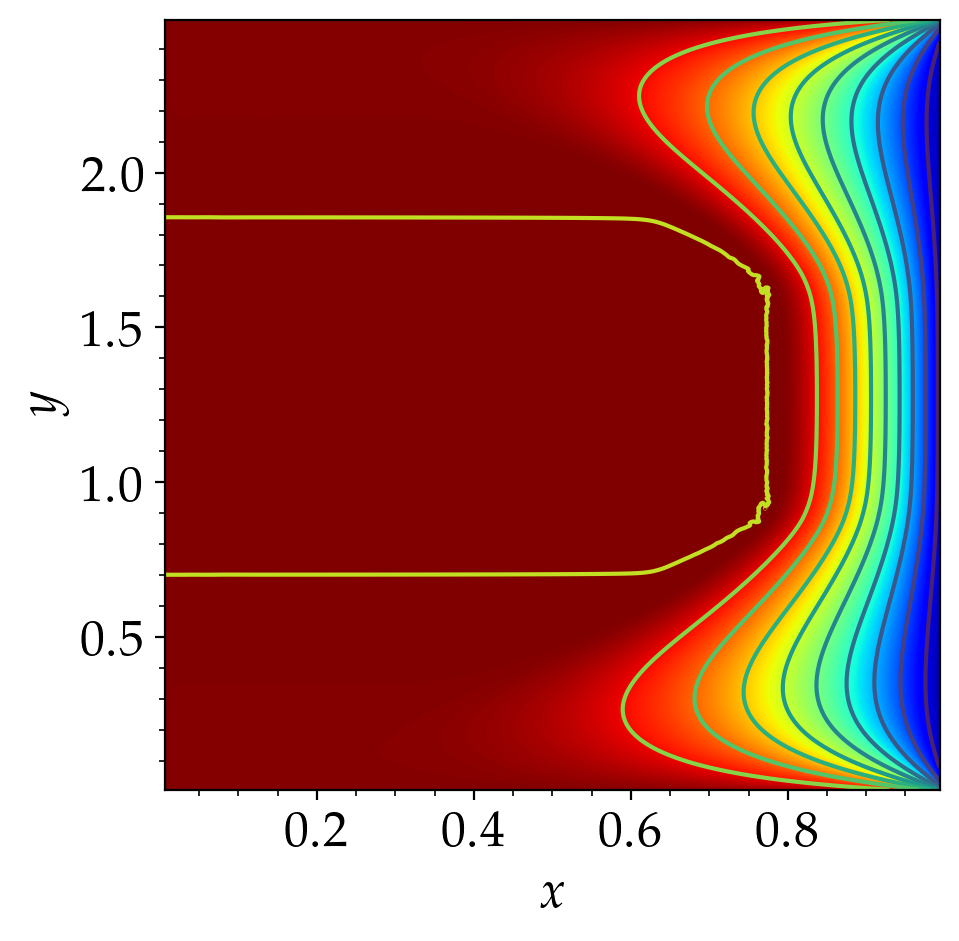}
    \includegraphics[height=0.22\textheight]{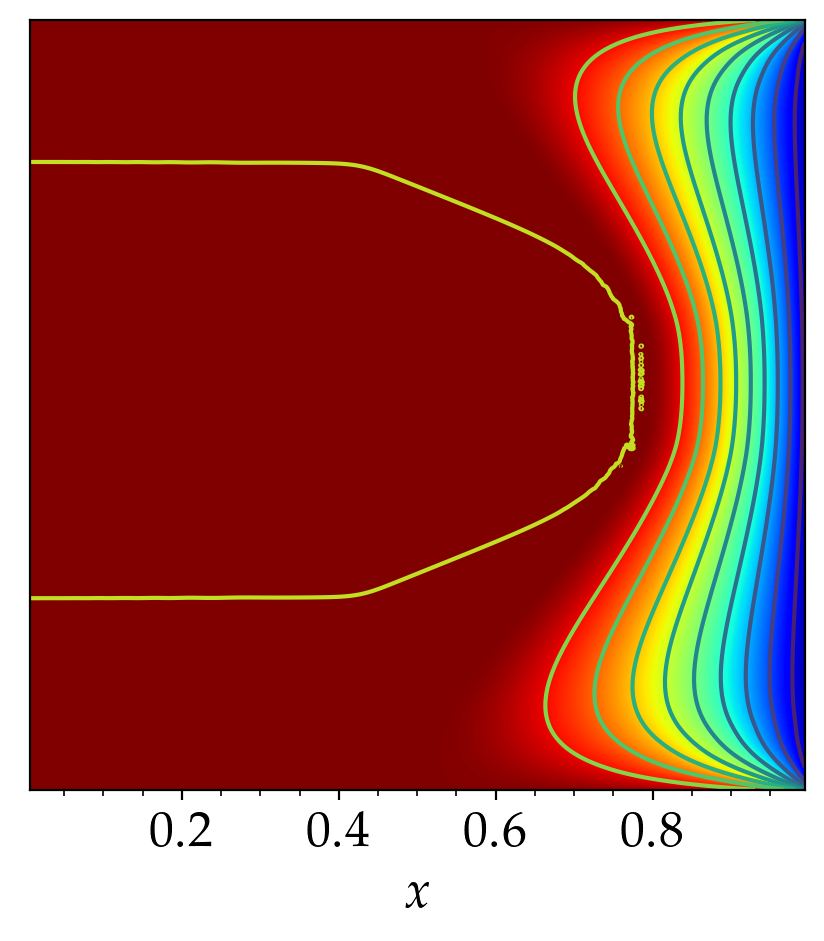}
    \includegraphics[height=0.23\textheight, width=0.345\textwidth]{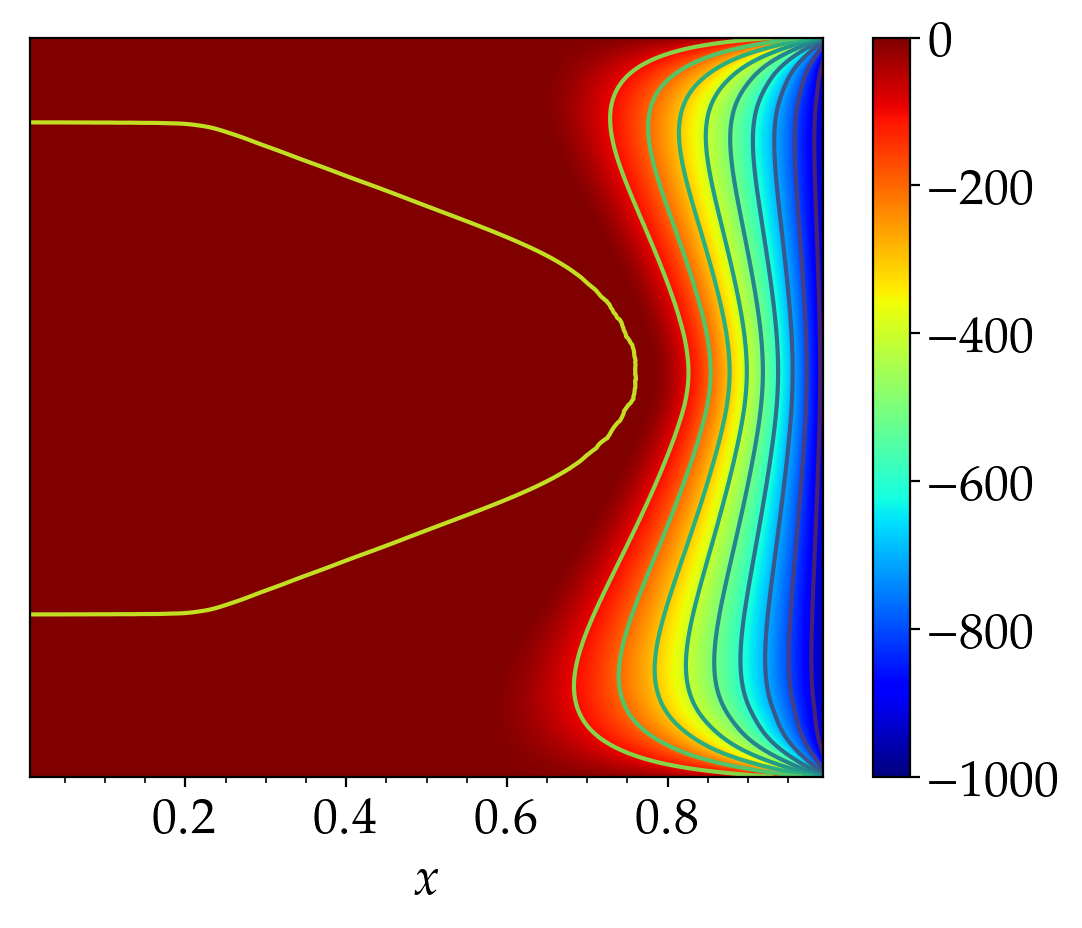}
    \caption{The ion density, velocity(x-component), potential plots at different times.}  
    \label{fig:DN_i-extraction}
  \end{figure}

The extraction process is simulated for three different times $T=0.1, 0.3$ and $0.5$. In Figure~\ref{fig:DN_i-extraction}, we plot the density, the $x$-component of the velocity and the potential profiles at the chosen times. The figures clearly show the immediate formation of an ion sheath. During the initial times, the curvilinear potential contours develop in the peripheral regions. At later times we see the symmetric axis of the plasma shifting towards the positive $y$-axis due to the initial $y$-velocity. This potential accelerates the ions and gives rise to the formation of ion acoustic waves. In Figure~\ref{fig:DN_i-extraction}, a gradual density decline can be clearly seen and the propagating ion acoustic wave is visible from the elliptical shape, particularly noticeable during the early stages around the peak density. As the pre-sheath field pushes the ions to the sheath boundary (near the high voltage cathode area), the peak density is decreased. The ion acoustic wave fundamentally creates an internal electric field that reduces significantly the plasma extraction duration. This electric field accelerates the plasma outward in all directions from the peak density region. We can also see from the colour-bar of the velocity plots that the velocity becomes supersonic after the sheath transition region as the scaled velocity of sound is equal to $11/8$ in a plasma. As a result, a substantial ion current is extracted from the top and bottom of the plasma at a supersonic velocity.      

\section{Concluding Remarks}
\label{sec:conclusion}
We have designed, analysed and implemented an AP and energy stable scheme for the EPB system in the quasineutral limit. The energy stability is achieved by introducing suitable stabilisation terms in the convective fluxes and the source term. The scheme is shown to be energy stable under a CFL-type condition which also yields the positivity of the density. The existence of the numerical solution is established using some topological degree theory results and a comparison principle for the discrete nonlinear Poisson equation. Using apriori energy bounds and some reasonable boundedness assumptions, the Lax-Wendroff-type consistency of the numerical scheme with the weak solutions of the continuous model as well as its consistency with the quasineutral ICE limit system are shown rigorously. The results of several numerical case studies showcase the robustness of the present scheme in both the dispersive and the quasineutral regimes. The numerical results also reveal that the present scheme can very well resolve plasma sheaths and the related dynamics which indicates its potential to applications involving low-temperature plasma problems. 

\appendix

\section{Existence and Uniqueness for the Discrete Poisson-Boltzmann Equation}
\label{sec:Comp_PB}

The goal of this Appendix is to establish the existence and uniqueness of the discrete Neumann problem \eqref{eq:dis_neumann}. We also prove a comparison principle for the solution which is vital in proving Theorem~\ref{thm:dis_existence} on the existence of the numerical solution. 

We start with the discrete problem
\begin{subequations}
\label{eq:dis_poisson_comp}
   \begin{align}
    -\varepsilon^2(\Delta_\M \phi)_{K}+e^{\phi_K}&=\rho_K, \quad \forall K\in\M, \\
    (\partial_{\E}^{(i)}\phi)_\s &= 0, \quad \forall \s\in\E_{\extr}^{(i)}.
\end{align} 
\end{subequations}

\begin{theorem}
\label{thm:dis_existence}
    Given a $\rho\in L_\M(\Omega)$ such that $\rho_K>0$ for all $K\in\M$, there exist a solution $\phi\in L_\M(\Omega)$ to \eqref{eq:dis_poisson_comp}. 
\end{theorem}

\begin{proof}
Consider
\begin{equation*}
-\veps^2\left(\Delta_\M \phi\right)_K+e^{\phi_K}=\rho_K, \quad \forall K \in \M.    
\end{equation*}
Making the change of dependent variable $e^{\phi_K}=z_K \quad$ we obtain
\begin{equation*}
-\veps^2\left(\Delta_\M(\ln z)\right)_K+z_K=\rho_K.    
\end{equation*}
Using the definition of the discrete Laplacian and Lemma~\ref{lem:rho_sig} we write the above equation as
\begin{equation}
\label{eq:dis_poisson_z}
    -\veps^2\Big(\dive_\M\Big(\frac{1}{z}\nabla_\E z\Big)\Big)_K+z_K=\rho_K.
\end{equation}
Since $\rho_K>0$ for all $K$ on a fixed grid, we can assume that $0 < \alpha \leqslant \rho_K \leqslant \beta, \; \forall K$.
We will now prove existence of a solution of \eqref{eq:dis_poisson_z} in the compact subset $\mcal{K}=\left\{\left(v_K\right)_{K \in \M}\colon \alpha \leqslant v_K \leqslant \beta, \; \forall K\right\}$. To this end, we define the map $S\colon \mcal{K}\rightarrow L_\M(\Omega)$ as $S(v)=u$, where $u\in L_\M(\Omega)$ satisfies the linear system:
\begin{equation} 
\label{eq:cov_poisson}
-\varepsilon^2\left(\operatorname{div}_\M\left(\frac{1}{v} \nabla_\E u\right)\right)_K+u_K=\rho_K.
\end{equation}
Clearly, the linear system \eqref{eq:cov_poisson} admits a unique solution as its matrix is an M-matrix \cite{Hac92}. Hence, $S$ is well-defined and is obviously continuous. Using the maximum principle, cf.\ \cite[Theorem 3.24]{Win89}, we get the bound $\alpha\leqslant u_K\leqslant \beta$ for all $K\in\M$. Hence, $S$ becomes a continuous self-map on a compact set of $\mathbb{R}^{\#\M}$ and thus it has a fixed point. In other words, there exist $u\in L_\M(\Omega)$ which satisfy \eqref{eq:dis_poisson_z} and we are done with the proof.
\end{proof}

Having proved the existence of solutions to \eqref{eq:dis_poisson_comp}, we now establish a comparison principle for the solutions. The comparison principle is key to get bounds on the potential $\phi$. 

\begin{theorem}[Comparison principle]
\label{thm:dis_comp}
    Suppose $\phi$ and $\phi^\prime\in L_\M(\Omega)$ solve \eqref{eq:dis_poisson_comp} with the right hand sides $\rho$ and $\rho^\prime\in L_\M(\Omega)$ respectively. Then $\phi_K\leqslant \phi^\prime_K, \; \forall K\in\M$ if $\rho_K\leqslant \rho^\prime_K, \; \forall K\in\M$.
\end{theorem}
\begin{proof}
   We proceed by contradiction. Suppose there exists a $K\in\M$ such that $\phi_K>\phi_K^{\prime}$. Consider the following subset of $\M$:
   \begin{equation*}
       \Gamma:=\left\{K \in \M: \phi_K>\phi_K^{\prime}\right\}.
   \end{equation*}
Since $\phi$ and $\phi^\prime$ are solutions, we have
\begin{align*}
    -\veps^2(\Delta_\M \phi)_{K}+e^{\phi_K}&=\rho_K,\\
    -\veps^2(\Delta_\M \phi^\prime)_{K}+e^{\phi^\prime_K}&=\rho^\prime_K.
\end{align*}
Subtracting one from the other, we get 
\begin{equation}
    \label{eq:subs_poisson}
    -\varepsilon^2(\Delta_\M (\phi-\phi^\prime))_{K}+e^{\phi_K}-e^{\phi^\prime_K}=\rho_K-\rho^\prime_K
\end{equation}
Letting $\psi_K=\phi_K-\phi_K^{\prime}$ and using mean value theorem on the second term, we simplify \eqref{eq:subs_poisson} to
\begin{equation}
    \label{eq:subs_poisson_psi}
    -\veps^2\sum_{\substack{\s\in\Ek \\ \s=K|L}}\frac{\abs{\s}^2}{\abs{\Ds}}(\psi_L-\psi_K)+\abs{K}e^{\phi_K^{\star}}\psi_K=\abs{K}(\rho_K-\rho^\prime_K),
\end{equation}
where $\phi_K^*\in\llbracket \phi_K,\phi^\prime_K\rrbracket$. We multiply \eqref{eq:subs_poisson_psi} by $\psi_K^{+}$ and then take sum over $K\in\M$ to get
\begin{equation*}
    -\veps^2\sum_{K\in\M}\sum_{\substack{\s\in\Ek \\ \s=K|L}}\frac{\abs{\s}^2}{\abs{\Ds}}(\psi_L-\psi_K)\psi_K^{+}+\sum_{K\in\M}\abs{K}e^{\phi_K^{\star}}\psi_K\psi_K^{+}=\sum_{K\in\M}\abs{K}(\rho_K-\rho^\prime_K)\psi_K^{+}.
\end{equation*}
As $\psi_K^{+}$ vanishes on $\M \smallsetminus \Gamma$, the above equation yields
\begin{equation*}
    \veps^2\sum_{\substack{\s(=K|L)\in\E \\ K,L\in\Gamma}}\frac{\abs{\s}^2}{\abs{\Ds}}(\psi_L-\psi_K)^2-\veps^2\sum_{\substack{\s(=K|L)\in\E \\ K\in\Gamma,L\notin\Gamma}}\frac{\abs{\s}^2}{\abs{\Ds}}\psi_L\psi_K+\sum_{K\in\Gamma}\abs{K}e^{\phi_K^{\star}}\psi_K^2=\sum_{K\in\Gamma}\abs{K}(\rho_K-\rho^\prime_K)\psi_K
\end{equation*}
It is easy to see that the first and third terms on the left hand side are positive and as $\psi_L\leqslant 0$ when $L\in\Gamma$, the second term is non-negative. But on the other hand, right hand side is negative since $\rho_K\leqslant \rho_K^{\prime}$. Hence, we arrive at a contradiction to the fact that there exist a $K\in\M$ such that $\phi_K>\phi_K^{\prime}$. Therefore, we conclude that $\phi_K\leqslant \phi^\prime_K, \; \forall K\in\M$.
\end{proof}

\begin{theorem}
\label{thm:phi_bdd}
  Let $\phi\in L_\M(\Omega)$ solve \eqref{eq:dis_poisson_comp} for $\rho\in L_\M(\Omega)$ satisfying $0 < \underline{c} \leqslant \rho_K \leqslant \bar{C}$ for all $K\in\M$. Then, $\phi$ is also bounded and satisfies $\ln(\underline{c}) \leqslant \phi_K \leqslant \ln(\bar{C}), \; \forall K\in\M$.
\end{theorem}
\begin{proof}
    The proof is straightforward and is easily obtained from the Theorem~\ref{thm:dis_existence} and \ref{thm:dis_comp}.
\end{proof}
\begin{corollary}
    Combining Theorems \ref{thm:dis_existence} and \ref{thm:dis_comp}, we deduce that for a given $\rho\in L_{\M}(\Omega)$ there exist unique $\phi\in L_\M(\Omega)$ which solves \eqref{eq:dis_poisson_comp}.
\end{corollary}

\bibliographystyle{abbrv}
\bibliography{references}
\end{document}